\documentclass[12pt,a4paper]{amsart}
 \usepackage{amsfonts,amssymb,enumerate,color,bm,hhline,makecell,array}
\usepackage{array,mathrsfs}
\usepackage[all,ps,cmtip]{xy} 
\usepackage[normalem]{ulem}
\usepackage{bm,mathtools}
\usepackage{hyperref}

\usepackage{enumitem}
\setlist[itemize]{noitemsep,nolistsep}
\setlist[enumerate]{noitemsep,nolistsep}
\usepackage{colortbl} 

\allowdisplaybreaks
\let\mathcal\mathscr

\def\Z{{\bf Z}}
\def\N{{\bf N}}

\def\C{{\bf C}}

\def\R{{\bf R}}
\def\Q{{\bf Q}}
\def\P{{\bf P}}

\def\hK{hyperk\"ahler}
\def\hKm{hyperk\"ahler manifold}
\def\ggs{generated by global sections}
\def\SS{S^{[2]}}

\def\phi{\varphi}
\def\kp{\kappa_{\rm prim}}

\def\bs{{\mathbf s}}
\def\bv{{\mathbf v}}

\def\cI{\mathcal{I}}
\def\cD{\mathcal{D}}
\def\cA{\mathcal{A}}

\def\cF{\mathcal{F}}

\def\cO{\mathcal{O}}

\def\cP{\mathcal{P}}
\def\cH{\mathcal{H}}
\def\cE{\mathcal{E}}
\def\cS{\mathcal{S}}
\def\cC{\mathcal{C}}
\def\cM{\mathcal{M}}

\def\cK{\mathcal{K}}

\def\cQ{\mathcal{Q}}
\def\cU{\mathcal{U}}

\def\lra{\longrightarrow}
\def\llra{\hbox to 10mm{\rightarrowfill}}
\def\lllra{\hbox to 15mm{\rightarrowfill}}

\def\llla{\hbox to 10mm{\leftarrowfill}}
\def\lllla{\hbox to 15mm{\leftarrowfill}}

\def\dra{\dashrightarrow}

\def\hra{\hookrightarrow}
\def\lhra{\ensuremath{\lhook\joinrel\relbar\joinrel\rightarrow}}

\def\isom{\simeq}
\def\eps{\varepsilon}
\def\ie{\hbox{i.e.}}
\def\eg{\hbox{e.g.}}

\def\tO{\widetilde{O}}

 \def\vide{\varnothing}

\DeclareMathOperator{\isomto}{\stackrel{{}_{\scriptstyle\sim}}{\to}}
\DeclareMathOperator{\isomtol}{\stackrel{{}_{\scriptstyle\sim}}{\leftarrow}}
\DeclareMathOperator{\isomdra}{\stackrel{{}_{\scriptstyle\sim}}{\dra}}
\DeclareMathOperator{\isomlra}{\stackrel{{}_{\scriptstyle\sim}}{\lra}}

\DeclareMathOperator{\Amp}{Amp}
\DeclareMathOperator{\Aut}{Aut}
\DeclareMathOperator{\Bir}{Bir}

\DeclareMathOperator{\ch}{ch}
\DeclareMathOperator{\codim}{codim}

\def\div{\mathop{\rm div}\nolimits}

\DeclareMathOperator{\disc}{disc}
\DeclareMathOperator{\Exc}{Exc}

\DeclareMathOperator{\GL}{GL}
\DeclareMathOperator{\Gr}{\mathsf{Gr}}
\DeclareMathOperator{\OGr}{\mathsf{OGr}}

\DeclareMathOperator{\Hilb}{Hilb}
\DeclareMathOperator{\Hom}{Hom}
\DeclareMathOperator{\Id}{Id}

\DeclareMathOperator{\Int}{Int}

\DeclareMathOperator{\Ker}{Ker}

\DeclareMathOperator{\Mon}{Mon}
\DeclareMathOperator{\Mov}{Mov}

\DeclareMathOperator{\NS}{NS}

\DeclareMathOperator{\PGL}{PGL}
\DeclareMathOperator{\Pic}{Pic}
\DeclareMathOperator{\Pos}{Pos}

\DeclareMathOperator{\Nef}{Nef}
\DeclareMathOperator{\Proj}{Proj}

\DeclareMathOperator{\rank}{rank}

\DeclareMathOperator{\Sym}{Sym}

\DeclareMathOperator{\td}{td}
\DeclareMathOperator{\Tr}{Tr}
\DeclareMathOperator{\VSP}{VSP}
 
\def\llra{\hbox to 10mm{\rightarrowfill}}
\def\lllra{\hbox to 15mm{\rightarrowfill}}

\def\bw#1#2{\textstyle{\bigwedge\hskip-0.9mm^{#1}}\hskip0.2mm{#2}}
\def\sbw#1#2{\small{\bigwedge\hskip-0.9mm^{#1}}\hskip0.2mm{#2}}

\newtheorem{lemm}{Lemma}[section]
\newtheorem{theo}[lemm]{Theorem}
\newtheorem{coro}[lemm]{Corollary}
\newtheorem{prop}[lemm]{Proposition}

\theoremstyle{definition}
\newtheorem{defi}[lemm]{Definition}
\newtheorem{rema}[lemm]{Remark}

\newtheorem{exam}[lemm]{Example}

\theoremstyle{remark}
\newtheorem*{remark*}{Remark}
\newtheorem*{note*}{Note}

\newtheorem{exer}[lemm]{\footnotesize{Exercise}}

\def\moins{\smallsetminus}

\def\L{{\Lambda}}
\def\LKKK{{\Lambda_{\KKK}}}
\def\Lkkk[#1]{{\Lambda_{\KKK^{[#1]}}}}
\def\kkk[#1]{{\KKK^{[#1]}}}
\DeclareMathOperator{\KKK}{{K3}}

\def\sss[#1]{{S^{[#1]}}}

\def\setminus{\smallsetminus}

\newcommand{\bijar}[1][]{%
 \ar[#1]
 \ar@<0.7ex>@{}[#1]|-*=0[@]{\sim}}

\begin{document}
\title{
 Hyperk\"ahler manifolds}

\author[O.\ Debarre]{Olivier Debarre}
\address{Universit\'e Paris Diderot-Sorbonne Universit\'e, CNRS, 
Institut de Math\'ematiques de Jussieu-Paris Rive Gauche, IMJ-PRG, 75013 Paris, France}
 \email{{\tt olivier.debarre@imj-prg.fr}}

 \date{\today}


\begin{abstract}
The aim of these notes is to acquaint the reader with important objects in complex algebraic geometry: K3 surfaces and their higher-dimensional analogs, hyperk\"ahler manifolds.\ These manifolds are interesting from several points of view: dynamical (some have interesting automorphism groups), arithmetical (although we will not say anything on this aspect of the theory), and geometric.\ It is also one of those rare cases where the Torelli theorem allows for  a powerful link between  the  geometry of these manifolds   and lattice theory.\ 

We do not prove all the results that we state.\ Our aim is more to provide, for specific families of  hyperk\"ahler manifolds (which are projective deformations of punctual Hilbert schemes of K3 surfaces),   a panorama of    results about projective embeddings, automorphisms, moduli spaces,  period maps and domains, rather than a complete reference guide.\ These results are mostly not new, except perhaps those of Appendix~\ref{imagep} (written with E.~Macr\`i), where we give in Theorem~\ref{thm:ImagePeriodMap} an explicit description of the image of the period map for these polarized manifolds. 
\end{abstract}

  \thanks{These  notes were originally written (and later expanded) for  two different mini-courses, one given for a summer school   on Texel Island, The Netherlands,  August 28--September 1, 2017, and the other 
  for the CIMPA Research School organized at the Pontificia Universidad Cat\'olica del Per\'u in Lima, Peru, September 4--15, 2017.\ The author would like to thank both sets of organizers for their support: Bas Edixhoven, Gavril Farkas, Gerard van der Geer, J\"urg Kramer, and 
Lenny Taelman
  for the first school, Richard Gonz\'ales and Clementa Alonso for the CIMPA school.}

\maketitle

\tableofcontents

\section{Introduction}
As explained in the abstract, the aim of these notes is to gather some results on  projective K3 surfaces and \hKm s, in particular their projective embeddings, their biregular and birational automorphism groups, and their moduli spaces.\ For K3 surfaces, these results have been known for more than twenty years, whereas they are much more recent for \hKm s; for example, the Torelli theorem for K3 surfaces was proved more than 35 years ago, whereas its version for all \hKm s (Theorem~\ref{torthhk}) was published by M.~Verbitsky in 2013.\ The results on the image of the period map (Theorems~\ref{imper} and \ref{thm:ImagePeriodMap}) are new and were obtained on collaboration with E.~Macr\`i.

Section~\ref{sect2} is devoted to K3 surfaces: complex compact surfaces with vanishing irregularity and whose space of holomorphic 2-forms is generated by a  nowhere vanishing  form.\ After describing their topological invariants, we state some characterizations of ample and very ample line bundles on projective K3 surfaces.\ We then describe general K3 surfaces with a polarization of low degree (most are complete intersections in a homogeneous space).\ The Torelli theorem says that K3 surfaces are characterized by their Hodge structure; more precisely, any automorphism between the Hodge structures on their second cohomology groups is induced by an isomorphism.\ This means that their  period map, a regular map between their (quasiprojective) moduli space (whose construction we explain in Section \ref{sec14} for polarized K3 surfaces) and their (quasiprojective) period domain (the quotient of a Hermitian symmetric domain by an arithmetic group of automorphisms) is an open embedding (Section~\ref{sec17}).\ Its image   is also described in that Section, using the description of the ample cone given in Section~\ref{newsecc}: it is the complement of the union of one or two Heegner divisors.\ 
The Torelli theorem is also extremely useful to study automorphisms groups of K3 surfaces.\ We   give   examples in Section~\ref{sec210}.

The rest of the notes deals with \hKm s, which are generalizations of K3 surfaces in higher (even) dimensions and for which many results are still unknown.\ They are defined in Section~\ref{sect3} as simply connected compact K\"ahler manifolds whose space of holomorphic 2-forms is generated by an everywhere nondegenerate    form.\ Their second integral cohomology group carries a canonical integral valued nondegenerate quadratic form defined by Fujiki and Beauville.\ Examples are provided by punctual Hilbert schemes of K3 surfaces and generalized Kummer varieties.\ Most of the results that we state will concern only deformations of the former type (called {\em \hKm s of $\KKK^{[m]}$-type}).\ 

Polarized \hKm s of that type admit quasiprojective moduli spaces whose irreducibility we discuss in Section~\ref{sec25}.\ Even 
in low degrees and dimension 4, their projective embeddings are only known in a few cases, through beautiful but quite involved geometric constructions (Section~\ref{sec36}).\ In the next (long) Section~\ref{sec37}, we define and study two important cones attached to a \hKm: the nef and the movable cones.\ These cones are closed convex cones in a real vector space of dimension the rank of the Picard group of the manifold.\ Their determination  is a very difficult question, only recently settled by works of Bayer, Macr\`i, Hassett, and Tschinkel.\ Many examples are given in Section~\ref{sec37}, essentially when the Picard rank is 2, where their description involves playing with some equations of Pell-type.

The next section (Section~\ref{sec27}) contains two versions of the Torelli theorem.\ We miss all the subtleties (and difficulties) of this general result by restricting ourselves to polarized \hKm s.\ Even with the ``classical'' version, one has to be careful (see Theorem~\ref{torthhk}).\ When one states the Torelli theorem in terms of the injectivity of a period map (Theorem~\ref{torthhk2}), the situation is tricky and more complicated than for K3 surfaces: the moduli spaces, although still quasiprojective, may be reducible, and the period domain is obtained by quotienting (still a Hermitian symmetric domain) by a restricted automorphism group.\ At the same time, this makes the situation richer: the period domains may have nontrivial involutions and may be isomorphic to each other.\ This implies that the some moduli spaces of polarized \hKm s are birationally isomorphic, a phenomenon which we call {\em strange duality}   (Remarks~\ref{stdu} and ~\ref{rem220}).\ 

In the rest of   Section~\ref{sect3}, we determine explicitly the image of the period map for polarized \hK\ fourfolds.\ This involves going through a rather lengthy description of the Heegner divisors in the period domain (Section~\ref{sec29}).\ As for K3 surfaces, the image is the complement of a finite number of Heegner divisors which we describe precisely.\ This result is a simple consequence of the description of the ample cone (which is the interior of the nef cone) given earlier.

In the final Section~\ref{sect4}, we use our knowledge of the movable and nef cones of \hKm s given in Section~\ref{sec37} to describe explicitely the birational and biregular automorphism groups of \hKm s of Picard number 1 or 2 in some cases.\ For Picard number 1 (Section~\ref{sec42}), the result is a rather simple consequence of the Torelli theorem and some easy lattice-theoretic considerations.\ For Picard number 2, a general result was proved by Oguiso (Theorem~\ref{thogui}).\ We end the section with many explicit calculations.

In the final short Section~\ref{sect5}, we use the Torelli theorem (and a deep result of Clozel and Ullmo on Shimura varieties) to prove that in each moduli space of polarized \hKm s of $\KKK^{[m]}$-type, the points corresponding to Hilbert squares of K3 surfaces form a dense subset (our Proposition~\ref{thh} provides a  more explicit statement).

In Appendix \ref{secpell}, we go through a few elementary facts about Pell-type equations.\ In the more difficult Appendix \ref{imagep}, written with    E.~Macr\`i, we revisit the description of the ample cone of a projective \hKm\ in terms of its {\em Mukai lattice} and use it to describe explicitly the image of the period map for polarized \hKm s of $\KKK^{[m]}$-type in all dimensions.

\section{K3 surfaces}\label{sect2}

  K3 surfaces can also be defined over fields of   positive characteristics    (and have been used to prove results about complex K3 surfaces) but we will restrict ourselves here to complex K3 surfaces.\footnote{This strange name was coined by Weil in 1958:~``ainsi nomm\'ees en l'honneur de Kummer, K\"ahler, Kodaira et de la belle montagne K2 au Cachemire.''}\ The reader will find in \cite{huyk3} a very complete study of these surfaces  in all characteristics.\ The reference \cite{laf} is more elementary and covers more or less the same topics as this section (with more details).

\subsection{Definition and first properties}\label{sec11}

\begin{defi}
A {\em K3 surface} is a compact surface $S$ such that $H^0(S,\Omega^2_S)=\C \omega$, where $\omega$ is a nowhere vanishing holomorphic 2-form on $S$, and $H^1(S,\cO_S)=0$.
\end{defi}

K3 surfaces are interesting for many reasons:
\begin{itemize}
\item they have interesting dynamics: up to bimeromorphism, they are the only  compact complex surfaces which can have an automorphism with positive entropy (see footnote~\ref{entro}) and no fixed points;\footnote{A theorem of Cantat (\cite{can}) says that a  smooth compact K\"ahler surface $S$ with an automorphism of positive entropy is bimeromorphic
to either $\P^2$, or to a 2-dimensional complex torus, or to an Enriques surface, or to a K3
surface.\ If the automorphism has no fixed points, $S$  
is birational to a projective
K3 surface of Picard number greater than
1 and conversely, there is a projective
K3 surface of Picard number 2 with a fixed-point-free automorphism of positive entropy (see \cite{ogu1} or Theorem \ref{autk32}).}
\item they have interesting arithmetic properties;
\item they have interesting geometric properties: it is conjectured (and known except when the Picard number is 2 or 4; see   \cite{lili}) that any K3 surface contains countably many rational curves.
\end{itemize}

Let $S$ be a K3 surface.\ The exponential sequence
$$0\to \Z\xrightarrow{\ \cdot 2i\pi\ } \cO_S\xrightarrow{\ \exp\ } \cO_S^*\to 1$$
implies that $H^1(S,\Z)$ is a subgroup of $H^1(S,\cO_S)$, hence $b_1(S)=0$.\ Moreover, there is an exact sequence
\begin{equation}\label{exps}
0\lra \Pic(S)\lra H^2(S,\Z)\lra H^2(S,\cO_S).
\end{equation}

\begin{lemm}\label{pic}
The Picard group of a K3 surface is torsion-free.
\end{lemm}

\begin{proof}
Let $M$ be a torsion element in $\Pic(S)$.\ The Riemann--Roch theorem and Serre duality give  
$$h^0(S,M)-h^1(S,M)+h^0(S,M^{-1}) = \chi(S,M)=\chi(S,\cO_S)=2.
$$
In particular, either $M$ or $M^{-1}$ has a nonzero section $s$.\ If $m$ is a positive integer such that $M^{\otimes m}$ is trivial, the nonzero section $s^m$ of $M^{\otimes m}$ vanishes nowhere, hence so does $s$.\ The line bundle $M$ is therefore trivial.
\end{proof}

Since $H^2(S,\cO_S)$ is also torsion-free, the lemma implies that $H^2(S,\Z)$ is torsion-free, hence so is $H_2(S,\Z)$ by Poincar\'e duality.\ By the Universal Coefficient Theorem, $H_1(S,\Z)$ is torsion-free hence 0 (because $b_1(S)=0$) , hence so is $H^3(S,\Z)$ by Poincar\'e duality again.\ So the whole (co)homology of $S$ is torsion-free.

We have $c_2(S)=\chi_{\textnormal{top}}(S)=b_2(S)+2$.\ Noether's formula 
$$12\chi(S,\cO_S)=c_1^2(S)+c_2(S)
$$
then implies 
$$c_2(S)=24\quad,\quad b_2(S)=22.$$
The abelian group $H^2(S,\Z)$ is therefore free of rank 22.\ The intersection form is unimodular and {\em even.}\footnote{Let $X$ be a  smooth compact real manifold of (real) dimension $4$.\ Its second
 {\em Wu class}  $v_2(X)\in H^2(X,\Z/2\Z)$ is characterized by the property (Wu's formula)
\begin{equation*}\label{cha}
\forall x\in H^2(X,\Z/2\Z)\qquad v_2(X)\cup x=x\cup x.
\end{equation*}
The class $v_2$ can be expressed as $w_2+w_1^2$, where $w_1$ and $w_2$ are the first two Stiefel--Whitney classes.\ When $X$ is a complex surface, one has $w_1(X)=0$ and $w_2(X)$ is the reduction modulo 2 of $c_1(X)$ (which is 0 for a K3 surface).} Its signature is given by Hirzebruch's formula
$$\tau(S)=\frac13(c_1^2(S)-2c_2(S))=-16.$$

\bigskip

Recall from Lemma \ref{pic} that the Picard group is therefore a free abelian group; we let $\rho(S)$ be its rank (the {\em Picard number} of $S$).

The Lefschetz $(1,1)$-theorem tells us that $\rho(S)\le h^1(S,\Omega_S^1)$.\ If $S$ is K\"ahler, we have
$$h^1(S,\Omega_S^1)=b_2(S)-2h^0(S,\Omega_S^2)=20$$
and this holds even if $S$ is not K\"ahler (which does not happen anyway!).\footnote{As explained in the proof of \cite[Lemma~IV.2.6]{bhpv}, 
 there is an inclusion
  $H^0(S,\Omega^1_S)\subset H^1(S,\C) $   for any compact complex  surface $S$, so that $H^0(S,\Omega^1_S)=0$ for a K3 surface $S$.\ We then have $H^2(S,\Omega^1_S)=0$ by Serre duality and the remaining number $h^1(S,\Omega^1_S)$ can be computed using the Riemann--Roch theorem
$\chi(S,\Omega^1_S)=\int_S\ch(\Omega^1_S)\td(S)=\frac12 c_1^2(\Omega^1_S)-c_2(\Omega^1_S)+2\rank(\Omega^1_S)=-20$ (because $c_1(\Omega^1_S)=0$ and $c_2(\Omega^1_S)=24$).} In particular,
we have 
$$0\le \rho(S)\le 20.$$
 
 Finally, one can show that all K3 surfaces are diffeomorphic and that they are all simply connected.

\begin{exer}
\footnotesize\narrower

Let $S\subset \P^3$ be the Fermat surface  defined by the quartic equation $x_0^4+x_1^4+x_2^4+x_3^4=0$.

\noindent (a) Show that $S$ is smooth and that it is a K3 surface.

\noindent (b) Show that $S$ contains (at least) 48 lines $L_1,\dots,L_{48}$.


\noindent (c) How would you show (using a computer) that the rank of the Picard group $\Pic(S)$ is at least 20?


\noindent (d) Prove that the rank of the Picard group $\Pic(S)$ is exactly 20.\footnote{It was proved by Mizukami in 1975 that the group $\Pic(S)$ is generated by $L_1,\dots,L_{48}$ and that its discriminant is $-64$.
}

 \end{exer}

\subsection{Properties of lines bundles}\label{ssec22}

Let $L$ be an ample line bundle on $S$.\ By Kodaira vanishing, we have $H^1(S,L)=0$ hence the Riemann--Roch theorem reads
$$h^0(S,L)=\chi(S,L)=\frac12 L^2+2.$$
Most of the times, $L$ is \ggs, by the following result that we will not prove. 

\begin{theo}\label{th13}
Let $L$ be an ample line bundle on a K3 surface $S$.\ The line bundle $L$ is \ggs\ if and only if there are no divisors $D$ on $S$ such that $LD=1$ and $D^2=0$.
\end{theo}

The hypothesis of the theorem holds  in particular when $\rho(S)=1$.\ When $L$ is \ggs, it  defines a morphism
$$\phi_L\colon S\lra \P^{e+1},$$
where $e:= \frac12 L^2$.\ General elements   $C\in |L|$ are smooth irreducible  (by Bertini's theorem)  curves of genus $e+1$ and the restriction of $\phi_L$ to $C$ is the canonical map $C\to \P^e$.\ 

One can show the morphism $\phi_L$ is then an embedding (in which case all smooth curves in $|L|$ are nonhyperelliptic), that is, $L$ is very ample, except in the following cases, where $\phi_L$   is a double cover of its image, a smooth surface of minimal degree $e$ in $\P^{e+1}$:
\begin{itemize}
\item $L^2=2$;
\item $L^2=8$ and $L=2D$;
\item there is a divisor $D$ on $S$ such that $LD=2$ and $D^2=0$ (all smooth curves in $|L|$ are then hyperelliptic).
\end{itemize}
The last case does not occur when $\rho(S)=1$.\ This implies the following result.


\begin{theo}\label{vample}
Let $L$ be an ample line bundle on a K3 surface.\ The line bundle $L^{\otimes 2}$ is \ggs\ and the line bundle $L^{\otimes k}$ is very ample for all $k\ge 3$.
\end{theo}

\subsection{Polarized K3 surfaces of low degrees}\label{sec13}

A polarization $L$ on a K3 surface $S$ is an isomorphism class of  ample line bundles on $S$ or equivalently, an ample class in $\Pic(S)$, which is not divisible in $\Pic(S)$.\ Its degree is $2e:=L^2$.\ We give below descriptions (mostly due to Mukai; see \cite{muk1, muk2, muk3, muk4, muk5, muk6}) of the  
 morphism $\phi_L\colon S\to\P^{e+1}$ associated with a general polarized K3 surface $(S,L)$ of degree $2e\le 14$  (it is a morphism by Theorem~\ref{th13}) and of $S$ for $e\in\{8,\dots,12,15, 17, 19\}$.\footnote{It might be a bit premature to talk about ``general polarized K3 surfaces'' here but the statements below hold whenever   $\Pic(S)= \Z L$, a condition that can be achieved by slightly perturbing $(S,L)$; we will come back to that in Section~\ref{nlsec}.}
 
 Several of these descriptions involve the Grassmannian $\Gr(r,n)$, the smooth projective variety of dimension $r(n-r)$ that parametrizes $r$-dimensional subspaces in $\C^n$, and its vector bundles $\cS$, the tautological rank-$r$ subbundle, and   $\cQ$, the   tautological rank-$(n-r)$ quotient bundle.\ It embeds via the Pl\"ucker embedding into $\P(\bw{r}\C^n)$ and the restriction of $\cO_{\P(\sbw{r}\C^n)}(1)$ is a generator of $\Pic(\Gr(r,n))$.

\noindent{$\mathbf{L^2=2}$.} The morphism $\phi_L\colon S\to \P^2$ is a double cover branched over a smooth plane sextic curve.\ Conversely, any such double cover is a polarized K3 surface of degree $2$.

\noindent{$\mathbf{L^2=4}$.} The morphism $\phi_L\colon S\to \P^3$ induces an isomorphism between $S$ and a smooth quartic surface.\ Conversely, any smooth quartic surface in $\P^3$ is a polarized K3 surface of degree $4$.

\noindent{$\mathbf{L^2=6}$.} The morphism $\phi_L\colon S\to \P^4$ induces an isomorphism between $S$ and the   intersection of a quadric and a cubic.\ Conversely, any smooth complete intersection of a quadric and a cubic in $\P^4$ is a polarized K3 surface of degree $6$.

\noindent{$\mathbf{L^2=8}$.} The morphism $\phi_L\colon S\to \P^5$ induces an isomorphism between $S$ and the   intersection of 3 quadrics.\ Conversely, any smooth complete intersection of 3 quadrics in $\P^5$ is a polarized K3 surface of degree $8$.

\noindent{$\mathbf{L^2=10}$.} The morphism $\phi_L\colon S\to \P^6$ is a closed embedding.\ Its image is obtained as the transverse intersection of the Grassmannian $\Gr(2,5)\subset \P^9$, a quadric $Q\subset \P^9$, and a $\P^6\subset \P^9$.\  Conversely, any such smooth complete intersection   is a polarized K3 surface of degree $10$.

\noindent{$\mathbf{L^2=12}$.} The morphism $\phi_L\colon S\to \P^7$ is a closed embedding.\ Its image is obtained as the transverse intersection of the orthogonal Grassmannian\footnote{This   is one of the two components of the family of all 5-dimensional isotropic subspaces for a nondegenerate quadratic form on $\C^{10}$.}    $\OGr(5,10)\subset  \P^{15}$ and a $\P^8\subset \P^{15}$.\  Conversely, any such smooth complete intersection   is a polarized K3 surface of degree $12$.

One can also  describe   $\phi_L(S)$ as the zero-locus of a general section of the rank-4 vector bundle $\cO(1)^{\oplus 2}\oplus \cS(2)  $ on $\Gr(2,5)$ (\cite[Theorem 9]{muk5}).

\noindent{$\mathbf{L^2=14}$.} The morphism $\phi_L\colon S\to \P^8$ is a closed embedding.\ Its image is obtained as the transverse intersection of the Grassmannian $\Gr(2,6)\subset \P^{14}$ and a $\P^8\subset \P^{14}$.\  Conversely, any such smooth complete intersection   is a polarized K3 surface of degree $14$.

\noindent{$\mathbf{L^2=16}$.} General K3 surfaces of degree 16 are exactly the zero loci
 of  general sections   of the  of the rank-6 vector bundle $\cO(1)^{\oplus 4}\oplus   \cS(1) $ on $\Gr(3,6)$ (\cite[Example~1]{muk2}).

\noindent{$\mathbf{L^2=18}$.} General K3 surfaces of degree 18 are exactly the zero loci
 of  general sections   of the  rank-7 vector bundle $\cO(1)^{\oplus 3}\oplus   \cQ^\vee(1) $ on $\Gr(2,7)$ (\cite[Example~1]{muk2}).
 
 \noindent{$\mathbf{L^2=20}$.} General K3 surfaces of degree 20 are described in \cite{muk3} as Brill--Noether loci on curves of genus 11.

\noindent{$\mathbf{L^2=22}$.} General K3 surfaces of degree 22 are exactly the zero loci
 of  general sections   of the 
rank-$10$     vector bundle $ \cO(1)  \oplus   \cS(1)^{\oplus 3}  $ on $\Gr(3,7)$    (\cite[Theorem~10]{muk5}). 

\noindent{$\mathbf{L^2=24}$.} General K3 surfaces of degree 24 are exactly the zero loci
 of  general sections   of the  
rank-$10$     vector bundle $ \cS(1) ^{\oplus 2}\oplus   \cQ^\vee(1)  $ on $\Gr(3,7)$    (\cite[Theorem~1]{muk5}).
 
  
    \noindent{$\mathbf{L^2=30}$.} Let $T$ be the 12-dimensional GIT quotient of $\C^2\otimes \C^3\otimes \C^4$ by the action of $\GL(2)\times GL(3)$ on the first and second factors.\ There are tautological vector bundles $\cE$ and $\cF$ on $T$, of respective ranks 3 and 2.\ General K3 surfaces of degree 30 are exactly the zero loci
 of  general sections  of the rank-10 vector bundle $\cE^{\oplus 2} \oplus \cF^{\oplus 2} $ on $T$ (\cite{muk6}).

  \noindent{$\mathbf{L^2=34}$.} General K3 surfaces of degree 34 are exactly the zero loci
 of  general sections   of the  
 rank-10 vector bundle $ \bw2\cS^\vee \oplus\bw2\cQ\oplus  \bw2\cQ  $ on $\Gr(4,7)$ (\cite[Theorem~1]{muk5}).

  \noindent{$\mathbf{L^2=38}$.} General K3 surfaces of degree 38 are exactly the zero loci
 of  general sections   of the  
rank-18 vector bundle $ (\bw2\cS^\vee)^{\oplus 3}$ on $\Gr(4,9)$.


  \subsection{The ample  cone of a projective K3 surface}\label{newsecc}
 
 Let $X$ be a projective manifold.\ We
define the {\em nef cone}
 $$\Nef(X)\subset \NS(X)\otimes \R$$
 as the closed convex cone generated by classes of nef line bundles.\footnote{The group $\NS(X)$ is the group of line bundles on $X$ modulo numerical equivalence; it is a finitely generated abelian group.\ When $H^1(X,\cO_X)=0$ (\eg, for K3 surfaces), it is the same as $\Pic(X)$.}\ Its interior $\Amp(X)$ is not empty and consists of ample classes.

Let now $S$ be a projective K3 surface and let $M$ be a line bundle on $S$ with $M^2\ge -2$.\ The Riemann--Roch theorem and Serre duality imply
$$h^0(S,M)-h^1(S,M)+h^0(S,M^{-1})=2+\frac12M^2\ge 1,
$$
so that either $M$ or $M^{-1}$ has a nonzero section.\ If we fix an ample class $L$ on $S$, the line bundle $M$  has a nonzero section if and only if $L\cdot M\ge 0$.\ Let
$\Pos(S)$ be that of the two components of the cone $\{x\in \Pic(S)\otimes\R\mid x^2> 0 \}$ that contains an (hence all) ample class (recall that the signature of the lattice $\Pic(S)$ is $(1,\rho(S)-1)$).

 We set
$$\Delta^+:=\{ M\in \Pic(S)\mid M^2=-2,\ H^0(S,M)\ne 0\}.
$$
If $L$ is a ample class, we have $L\cdot M> 0$ for all $M\in \Delta^+$.\ If $M\in \Delta^+$, one shows that the fixed part of the linear system $|M|$ contains a smooth rational curve $C$ (with $C^2=-2$).

\begin{theo}\label{amps}
Let   $S$ be a projective K3 surface.\ We have
 \begin{eqnarray*}
\Amp(S)&=&\{x\in \Pos(S)\mid \forall M\in \Delta^+\quad x\cdot M> 0\}\\
&=&\{x\in \Pos(S)\mid \textnormal{ for all smooth rational curves }C\subset S, \  x\cdot C> 0\}.
\end{eqnarray*}
\end{theo}

For the proof, which is elementary, we refer to \cite[Corollary~8.1.6]{huyk3}.

\subsection{Moduli spaces for polarized K3 surfaces}\label{sec14}

We will indicate here the main steps for the construction of quasiprojective coarse moduli space for   polarized complex K3  surfaces of fixed degree.\ Historically, the first existence proof was given in 1971 by Pjatecki-Shapiro--Shafarevich and relied on the Torelli theorem (see next section), but it makes more sense to reverse the course of history and explain a construction based on later work of Viehweg. 

Let $(S,L)$ be a polarized K3 surface of degree $2e$.\ The fact that a fixed multiple (here $L^{\otimes 3}$) is very ample implies that $S$ embeds into some projective space of fixed dimension (here $\P^{9e+1}$), with fixed Hilbert polynomial (here $9eT^2+2 $).\ The Hilbert scheme that parametrizes closed subschemes of $\P^{9e+1}$ with that Hilbert polynomial is projective (Grothendieck) and its subscheme $\cH$ parametrizing K3 surfaces is open and smooth.\ The question is now to take the quotient of $\cH$ by the canonical action of $\PGL(9e+2)$.\ The usual technique for taking this quotient, Geometric Invariant Theory, is difficult to apply directly in that case but Viehweg managed to go around this difficulty to avoid a direct check of GIT stability and still construct a quasiprojective coarse moduli space (over $\C$ only! See \cite{vie}).

\begin{theo}\label{thm16}
Let $e$ be a positive integer.\ There exists an irreducible $19$-dimensional  quasiprojective coarse moduli space $\cK_{2e}$ for 
polarized complex K3 surfaces of degree $2e$.
\end{theo}

The constructions explained in Section \ref{sec13} imply that $\cK_{2e}$ is unirational for  $e\le 7$.\ It is in fact known to be unirational for $ e\le   19$ and $e\notin\{ 14,18\}$ by work of Mukai (and more recently Nuer).\ At the opposite end, $ \cK_{2e}$ is of general type for $e\ge 31$ (and for a few lower values of $e$; see \cite{ghs0,ghssur}).

\subsection{The Torelli theorem}

The Torelli theorem (originally stated and proved for curves) answers the question as to whether a smooth K\"ahler complex manifold is determined (up to isomorphism) by (part of) its Hodge structure.\ In the case of polarized K3 surfaces, this property holds.

\begin{theo}[Torelli theorem, first version]\label{torth}
Let $(S,L)$ and $(S',L')$ be polarized complex K3 surfaces.\ If there exists an isometry of lattices 
$$\phi\colon H^2(S',\Z)\isomto H^2(S,\Z)$$ such that $\phi(L')=L$ and $\phi_\C(H^{2,0}(S'))=H^{2,0}(S)$, there exists an isomorphism $\sigma\colon S\isomto S'$ such that $\phi=\sigma^*$.
\end{theo}

We will see later that the isomorphism $u$  is uniquely determined by $\phi$ (Proposition~\ref{prop110}).

\subsection{A bit of lattice theory}\label{seclat}

A very good introduction to this theory can be found in the superb book \cite{ser}.\ More advanced results are proved in the difficult article \cite{nik}.

A lattice is a free abelian group $\L$ of finite rank endowed with an integral valued quadratic form $q$.\ Its discriminant group is the finite abelian group 
$$D(\L):=\L^\vee/\L,$$
where 
$$\L\subset \L^\vee:=\Hom_\Z(\L,\Z)=\{x\in \L\otimes\Q\mid \forall y\in\L\quad x\cdot y\in\Z\}\subset \L\otimes\Q
.$$
 The lattice $\L$ is {\em unimodular} if the group $D(\L)$ is trivial; it is {\em even} if the quadratic form $q$ only takes even values.

If $t$ is an integer, we let $\L(t)$ be the group $\L$ with the quadratic form $tq$.\ Finally, for any integers $ r,s\ge 0$, we let $I_1$ be the lattice $\Z$ with the quadratic form
$q(x)=x^2$
and we let $I_{r,s}$ be the lattice $I_1^{\oplus r}\oplus I_1(-1)^{\oplus s}$.

The only unimodular lattice of rank $1$ is the lattice $I_1$.\ The only unimodular lattices of rank $2$ are the lattices $I_{2,0},I_{1,1},I_{0,2}$ and the hyperbolic plane $U$.\ There is a unique positive definite even unimodular lattice of rank 8, which we denote by $E_8$.\ The signature $\tau(\L)$  of an even unimodular lattice $\L$  is divisible by 8; if $\L$ is not definite (positive or negative), it is a direct sum of copies of $U$ and $E_8$ if $\tau(\L)\ge0$  (resp.\ of $U$ and $E_8(- 1)$ if $\tau(\L)<0$).

If $x$ is a nonzero element of $\L$, we define its divisibility $\gamma(x)$ as the positive generator of the subgroup $x\cdot\L$ of $\Z$.\ We also consider $ x/\gamma(x)$, a primitive (\ie, nonzero and nondivisible) element of $\L^\vee$, and its class $x_*=[x/\gamma(x)]\in D(\L)$, an element of order $\gamma(x)$.

{\em Assume now that the lattice $\L$ is even.}\ We extend the quadratic form to a $\Q$-valued quadratic form on $\L\otimes\Q$, hence also on $\L^\vee$.\ We then
define a quadratic form $\bar q\colon D(\L)\to \Q/2\Z$ as follows: let $x\in \L^\vee, y\in \L$; then $q(x+y)=q(x)+2x\cdot y+q(y)$, modulo $2\Z$, does not depend on $y$.\ We may therefore set
$$\bar q([x]):=q(x) \in \Q/2\Z.$$
The {\em stable orthogonal group} $\tO(\L)$ is the kernel of the canonical map
$$O(\L)\lra O(D(\L),\bar q).$$
 This map is surjective when $\L$ is indefinite and   $\rank(\L)$ is at least the minimal number of generators of the finite abelian group $D(\L)$ plus 2.\ 

We will  use  the following result very often (see \cite[Lemma~3.5]{ghs}).\ When $\L$ is even unimodular (and contains at least two orthogonal copies of  the hyperbolic plane $U$), it says that the $O(\L)$-orbit of a primitive vector is exactly characterized by the square of this vector.

\begin{theo}[Eichler's criterion]\label{eic}
Let $\Lambda$ be an even lattice that contains at least two orthogonal copies of $U$.\ The $\tO(\L)$-orbit of a  primitive vector  $x\in\L$ is determined by the integer $q(x)$ and the element $x_*$ of $D(\L)$.
\end{theo}


\subsection{The period map}\label{sec17}

The Torelli theorem (Theorem~\ref{torth}) says that a polarized K3 surface is determined by its Hodge structure.\ We want to express this as the injectivity of a certain morphism, the period map, which we  now construct.

Let $S$ be a complex K3 surface.\ The lattice $(H^2(S,\Z),\cdot)$ was shown in Section~\ref{sec11} to be even unimodular with signature $-16$; by the discussion in Section~\ref{seclat}, it is therefore isomorphic to the rank-22 lattice
$$\LKKK:= U^{\oplus 3}\oplus E_8(-1)^{\oplus 2}.$$
Since we will restrict ourselves to polarized K3 surfaces, we fix, for each positive integer $e$, a primitive vector $h_{2e}\in \LKKK$ with $h_{2e}^2=2e$ (by Eichler's criterion (Theorem~\ref{eic}), they are all in the same $O(\LKKK)$-orbit).\ For example, if $(u,v)$ is a standard basis for one copy of $U$, we may take $h_{2e}=u+ev$.\ We then have
\begin{eqnarray}\label{lke}
\L_{\KKK,2e}:=h_{2e}^\bot= U^{\oplus 2}\oplus E_8(-1)^{\oplus 2}\oplus I_1(-2e).
\end{eqnarray}

Let now $(S,L)$ be a polarized K3 surface of degree $2e$ and let $\phi\colon H^2(S,\Z)\isomto \LKKK$ be an isometry of lattices such that $\phi(L)=h_{2e}$ (such an isometry exists by Eichler's criterion).\ The ``period'' $p(S,L):= \phi_\C(H^{2,0}(S))\in \LKKK\otimes\C$ is then in $h_{2e}^\bot$; it also satisfies the Hodge--Riemann bilinear relations
$$p(S,L)\cdot p(S,L)=0\quad,\quad p(S,L)\cdot \overline{p(S,L)}>0.$$
This leads us to define the $19$-dimensional (nonconnected) complex manifold
$$\Omega_{2e}:=\{[x]\in \P( \L_{\KKK,2e}\otimes\C)\mid x\cdot x=0,\ x\cdot\overline{x}>0\}
,$$
so that $p(S,L)$ is in $\Omega_{2e}$.\ However, the point $p(S,L)$ depends on the choice of the isometry $\phi$, so we would like to consider the quotient of $\Omega_{2e}$ by the image 
of the (injective) restriction morphism
\begin{eqnarray*}
\{\Phi\in O(\LKKK)\mid \Phi(h_{2e}) =h_{2e}\}&\lra &O(\L_{\KKK,2e})\\
\Phi&\longmapsto&\Phi\vert_{h_{2e}^\bot}.
\end{eqnarray*}
 It turns out that 
 this image is equal to the special orthogonal group $\tO(\L_{\KKK,2e})$, so we set\footnote{\label{misl}This presentation is correct but a bit misleading:  $\Omega_{2e}$ has in fact two irreducible components (interchanged by complex conjugation) and one usually chooses one component $\Omega_{2e}^+$ and considers the subgroups of the various orthogonal groups that preserve this component (denoted by $O^+$ in \cite{marsur,ghs,ghssur}), so that $\cP_{2e} =\tO^+(\L_{\KKK,2e})\backslash \Omega_{2e}^+$.}
 $$\cP_{2e}:=\tO(\L_{\KKK,2e})\backslash \Omega_{2e}
.$$
Everything goes well: by the Baily--Borel theory, $\cP_{2e}$ is an irreducible quasiprojective normal noncompact variety of dimension 19 and one can define a {\em period map}
\begin{eqnarray*}
\wp_{2e}\colon \cK_{2e}&\lra&\cP_{2e}\\
{[}(S,L)]&\longmapsto& [p(S,L)]
\end{eqnarray*}
which is an algebraic morphism.\ The Torelli theorem now takes the following form.

\begin{theo}[Torelli theorem, second version]\label{torth2}
Let $e$ be a positive integer.\ The period map
$$\wp_{2e}\colon \cK_{2e}\lra\cP_{2e}
$$
is an open embedding.
\end{theo}

It is this description of $\cK_{2e} $ (as an open subset of the quotient of the Hermitian symmetric domain $\Omega_{2e}$ by an arithmetic subgroup acting properly and discontinuously) that Gritsenko--Hulek--Sankaran used to compute the Kodaira dimension of $ \cK_{2e}$ (\cite{ghs0,ghssur}).

\medskip

Let us discuss what the image of $\wp_{2e} $  is.\ Let $y\in \L_{\KKK,2e}$ be such that $y^2=-2$ (one says that $y$ is a {\em root} of $\L_{\KKK,2e} $).\ If a period $p(S,L)$ is orthogonal to $y$, the latter is, by the Lefschetz $(1,1)$-theorem, the class of a line bundle $M$ with $M^2=-2$ and $L\cdot M=0$.\ Since either $|M|$ or $|M^{-1}|$ is nonempty (Section \ref{newsecc}), this contradicts the ampleness of $L$ and implies that the image of the period map is contained in the complement  $\cP_{2e}^0$ of the image of
$$\bigcup_{y\in \L_{\KKK,2e},\ y^2=-2}y^\bot\subset  \Omega_{2e}
$$
 in $\cP_{2e}$.\
It turns out that the image of the period map is exactly $\cP_{2e}^0$ (\cite[Remark~6.3.7]{huyk3}).\ Let us describe  $\cP_{2e}^0$ using  Eichler's criterion.\

\begin{prop}\label{prop19}
Let $e$ be a positive integer.\ The image of the period map
$\wp_{2e}\colon \cK_{2e}\to\cP_{2e}
$
is the complement of one irreducible hypersurface if $e\not\equiv 1\pmod4$ and of two irreducible hypersurfaces if $e\equiv 1\pmod4$.
\end{prop}

\begin{proof}
By Eichler's criterion, the $\tO(\L_{\KKK,2e})$-orbit of a root $y$ of $\L_{\KKK,2e}$ is characterized by   $ y_*\in D(\L_{\KKK,2e} )$.\
Let $w$ be a generator of $I_1 (-2e)$ in a decomposition \eqref{lke} of $\L_{\KKK,2e}$ (if $(u,v)$ is a standard basis for one copy of $U$ in $\LKKK$, we may take $w=u-ev$).\ We then have $\div(w)=2e$ and $w_*$ generates $D(\L_{\KKK,2e})\isom \Z/2e\Z$, with $\bar q (w_*)=q\bigl(\frac{1}{2e}w\bigr)=-\frac{1}{2e} \pmod{2\Z}$.\ A root $y$ has divisibility $1$ or $2$; if $\div(y)=1$, we have $y_*=0$; if $\div(y)=2$, then $y_*=\frac12 y$ has order $2$ in $D(\L_{\KKK,2e}) $, hence $y_*=ew_*$.\ In the latter case, we have $\bar q (y_*)\equiv q(\frac12 y)=-\frac12 $, whereas $\bar q (ew_*)\equiv e^2\bigl(-\frac{1}{2e}\bigr)=-\frac{e}{2}$; this implies $e\equiv 1\pmod4$.\ Conversely, if this condition is realized, we write $e=4e'+1$ and we let $(u',v')$ be a standard basis for one copy of $U$ in $\L_{\KKK,2e}$.\ The vector
$$y:=w+2(u'+e'v') 
$$
is then a root with divisibility 2.\ This proves the proposition.
\end{proof}

These irreducible hypersurfaces, or more generally any hypersurface in $\cP_{2e}$ which is the (algebraic and irreducible) image $\cD_x$ of a hyperplane section of $\Omega_{2e} $ of the form $x^\bot$, for some  $x\in \L_{\KKK,2e}$ with $x^2<0$, is usually called a {\em Heegner divisor.}\ By Eichler's criterion, the Heegner divisors can be indexed by the integer $x^2$ and the element $x_*$ of $ D(\L_{\KKK,2e})$, for $x$  primitive in~$\L_{\KKK,2e}$.

\begin{rema}[Period map over the integers] The construction of the period domain and period map can be done over $\Q$, and even over $\Z[\frac12]$.\ We refer the reader to \cite[Section~8.6]{lie} for a summary of results and for references.
\end{rema}

\begin{rema}
As noted by Brendan Hassett, the results of Section~\ref{ssec22} imply, by the same reasoning as in the proof of Proposition~\ref{prop19}, that, inside $\cK_{2e}$, the locus where the polarization is globally generated (resp.\ very ample) is the complement of the union of finitely many Heegner divisors.
\end{rema}

\begin{rema}
A point of $\cP_{2e}$ which is very general in the complement of the image of $\wp_{2e}$ corresponds to a pair $(S,L)$, where $L$ is a ``nef and big'' line bundle on the K3 surface $S$.\  Theorem \ref{th13} still applies to these line bundles.\ When $e>1$, it implies that $L$ is   \ggs, hence defines a morphism $\phi_L\colon S\to \P^{e+1}$ that contracts a smooth rational curve to a point.

When $e=1$, the complement of the image of $\wp_{2e}$ has two irreducible components.\ On one component, the Picard lattice has matrix $\begin{pmatrix} 2&0\\0&-2\end{pmatrix}$ and 
the situation is as above: $L$ is   \ggs\ and defines a morphism $\phi_L\colon S\to \P^2$ branched over a sextic curve with a single node.\ On the other component, the Picard lattice has matrix $\begin{pmatrix} 2&1\\1&0\end{pmatrix}$ and   the fixed part of the linear system $|L|$ is a smooth rational curve $C$.\ The morphism $\phi_L=\phi_{L(-C)}\colon S\to \P^2$ is an elliptic fibration onto a smooth conic.
\end{rema}

\subsection{The Noether--Lefschetz locus}\label{nlsec}
Given a polarized K3 surface $(S,L)$ with period $p(S,L)\in \Omega_{2e} $, the Picard group of $S$ can be identified, by  the Lefschetz $(1,1)$-theorem, with the saturation of the subgroup of $\LKKK$ generated by $h_{2e}$ and 
$$p(S,L)^\bot \cap \L_{\KKK,2e}.
$$
This means that if the period of  $(S,L)$ is outside the countable union 
$$\bigcup_{x\textnormal{ primitive in }\L_{\KKK,2e},\ x^2<0}\!\! \cD_x$$ 
of Heegner divisors, the Picard number $\rho(S)$  is 1 (and $\Pic(S)$ is generated by $L$).\ The inverse image in $\cK_{2e}$ of this countable union is called the {\em Noether--Lefschetz locus.}

\subsection{Automorphisms}\label{sec210}

Many works deal with automorphism groups of  K3 surfaces and we will only mention a couple of related results.\ Let $S$ be a K3 surface.\ The first remark is that since $T_S\isom \Omega^1_S\otimes (\Omega^2_S )^\vee\isom \Omega^1_S$, we have
$$H^0(S,T_S)\isom H^0(S,\Omega^1_S)=0
.$$
In particular,   the group $\Aut(S)$ of biregular automorphisms of $S$ is discrete (note that since $S$ is a minimal surface which has a unique minimal model, any birational automorphism of $S$ is biregular).

\begin{prop}\label{prop110}
Let $S$ be a K3 surface.\ Any automorphism of $S$ that acts trivially on $H^2(S,\Z)$ is trivial.
\end{prop}

\begin{proof}[Sketch of proof]
We follow \cite[Section~15.1.1]{huyk3}.\ Let $\sigma$ be a nontrivial automorphism of $S$ of finite order $n$ that acts trivially on $H^2(S,\Z)$ (hence on $H^\bullet(S,\Z)$).\ By the Hodge decomposition, $\sigma$ also acts trivially on $H^0(S,\Omega^2_S)$, hence $\sigma^*\omega=\omega$.\ Around any fixed point   of $\sigma$, there are analytic local coordinates $(x_1,x_2)$ such that $\sigma(x_1,x_2)=(\lambda x_1, \lambda^{-1}x_2)$, where $\lambda$ is a primitive $n$th root of $1$ \cite[Lemma~15.1.4]{huyk3}.\ In particular, the
  number $N(\sigma)$ of fixed points of $\sigma$ is   finite.

The holomorphic Lefschetz fixed   point formula, which relates     $N(\sigma)$   to the trace of the action of $\sigma$ on the $H^i(S,\cO_S)$, implies $N(\sigma)\le 8$ (\cite[Corollary~15.1.5]{huyk3}).\ The topological Lefschetz fixed   point formula, which relates $N(\sigma)$ to the trace of the action of $\sigma$ on the $H^i(S,\Z)$, implies that since $\sigma$ acts trivially on $H^\bullet(S,\Z)$, we have $N(\sigma)=\chi_{\textnormal{top}}(S)=24$ (\cite[Corollary~15.1.6]{huyk3}).\ This gives a contradiction, hence  no nontrivial automorphism of $S$ of finite order  acts trivially on $H^2(S,\Z)$.

To prove that any automorphism $\sigma$  that acts trivially on $H^2(S,\Z)$ has finite order, one can invoke \cite[Theorem~4.8]{fuj} which says that the group of automorphisms of a compact K\"ahler manifold that fix a K\"ahler class has only finitely many connected components: this implies in our case that  $\sigma$ has finite  order, hence is trivial.\end{proof}

The conclusion of the proposition can be restated as the injectivity of the map
\begin{equation}\label{injp}
\Psi_{S}\colon\Aut(S)\lra O( H^2(S,\Z),\cdot).
\end{equation}
One may also consider the morphism
$$\overline\Psi_{S}\colon\Aut(S)\lra O( \Pic(S),\cdot).$$
Even if $S$ is projective, this morphism is not necessarily injective.\ An easy example is provided by polarized K3 surfaces $(S,L)$ of degree 2: we saw in Section~\ref{sec13} that $\phi_L\colon S\to \P^2$ is a double cover; the associated involution $\iota$ of $S$ satisfies $\iota^*L\isom L$ and when  $(S,L)$ is very general, $\Pic(S)=\Z L$ is acted on trivially by $\iota$.\ 

However, the kernel of  $ \overline\Psi_{S}$ is always finite (if $S$ is projective): if $S\subset \P^N$ is an embedding, 
 any automorphism $\sigma$ of $S$ in $\ker(\overline\Psi_{S})$ satisfies $\sigma^*\cO_S(1)\isom \cO_S(1)$.\ In follows that $\Aut(S)$ acts 
  linearly on $\P(H^0(\P^N,\cO_{\P^N}(1)))$ while globally preserving  $S$.\ The group $\ker(\overline\Psi_{S})$ is therefore a closed algebraic subgroup of the affine linear algebraic group $\PGL(N+1,\C)$ hence has finitely many components.\ Since it is discrete, it is finite.

Here a simple application of this result: the automorphism group of a K3 surface with Picard number 1 is not very interesting!

\begin{prop}\label{autk3}
Let $S$ be a K3 surface whose Picard group is generated by an ample class $L$.\ The automorphism group of $S$ is trivial when $L^2\ge 4$ and has order 2 when $L^2=2$.   
\end{prop}

\begin{proof}
Let $\sigma$ be an automorphism of $S$.\ One has $\sigma^*L=L$ and $\sigma^*$ induces a Hodge isometry of the {\em transcendental lattice} $L^\bot$.\ One can show that $\sigma^*\vert_{L^\bot}=\eps\Id$, where $\eps\in\{-1,1\}$.\footnote{If $(S,L)$ is very general, this follows from standard deformation theory; the general argument is clever and relies on Kronecker's theorem and  the fact that 21, the rank of $L^\bot$, is odd (\cite[Corollary~3.3.5]{huyk3}).}\ Choose as above an identification of the lattice $H^2(S,\Z)$ with $\LKKK$ such that $L=u+ev$, where $L^2=2e$ and $(u,v)$ is a standard basis of a hyperbolic plane $U\subset \LKKK$.\ One then has
$$\sigma^*(u+ev)=u+ev\qquad\textnormal{and}\qquad \sigma^*(u-ev)=\eps(u-ev).$$
This implies $2e\sigma^*(v)=(1-\eps)u+e(1+\eps)v$, so that $2e\mid (1-\eps)$.\ If $e>1$, this implies $\eps=1$, hence $\sigma^*=\Id$, and $\sigma=\Id$ by Proposition~\ref{prop110}.\ If $e=1$, there is also the possibility $\eps=-1$, and $\sigma$, if nontrivial, is a uniquely defined involution of $S$.\ But in that case,   such an involution always exists (Section~\ref{sec13}).
\end{proof}

To actually construct automorphisms of a K3 surface $S$, one needs to know the image of the map $\Psi_S$.\ This is provided by the Torelli theorem, or rather an extended version of Theorem \ref{torth} to the nonpolarized setting: any Hodge isometry of $H^2(S,\Z)$ that maps one K\"ahler class to a K\"ahler class is induced by an automorphism of $S$ (\cite[Theorem~7.5.3]{huyk3}).

Automorphisms get  more interesting when the Picard number increases.\ There is a huge literature on the subject and we will only sketch one construction.

\begin{theo}[Cantat, Oguiso]\label{autk32}
There exists a projective
K3 surface of Picard number 2 with a fixed-point-free automorphism of positive entropy.\footnote{\label{entro}If $\sigma$ is an automorphism of a metric space $(X,d)$, we set, for all positive integers $m$ and all   $x,y\in X$, 
$$d_m(x,y):=\max_{0\le i<m}d(\sigma^i(x),\sigma^i(y)).$$
For $\eps>0$, let $s_m(\eps)$ be the maximal number of disjoint balls in $X$ of radius $\eps/2$ for the distance $d_m$.\ The {\em topological entropy} of $\sigma$ is defined by
$$h(\sigma):=\lim_{\eps\to 0}\limsup_{m\to\infty}\frac{\log s_m(\eps)}{m}\ge 0
.$$
Automorphisms with positive entropy are the most interesting from a dynamical point of view.}
\end{theo}

\begin{proof}[Sketch of proof]
One first shows that there exists a projective K3 surface $S$ with Picard lattice isomorphic to the rank-2 lattice $K=\Z^2$ with intersection matrix\footnote{This can be deduced as in \cite[Lemma~4.3.3]{has} from the Torelli theorem and the surjectivity of the period map, because this lattice does not represent $-2$; the square-4 class given by the first basis vector is even very ample on $S$ since the lattice does not represent $0$ either.\ Oguiso gives in \cite{ogu1} a geometric construction of the surface $S$ as a quartic in $\P^3$.\ It was later realized  in \cite{vgee} that $S$ is a determinantal Cayley surface (meaning that its equation can be written as the determinant of a $4\times 4$-matrix of linear forms) and that the automorphism $\sigma$ had already been constructed by Cayley in 1870 (moreover, $\Aut(S)\isom \Z$, generated by $\sigma$); that article gives very clear explanations about the various facets of these beautiful constructions.}
$$\begin{pmatrix}  4 & 2 \\  2 &-4\end{pmatrix}.$$
 By Theorem \ref{amps} and the fact that the lattice $K$ does not represent $-2$, the ample cone of $S$ is one component of its positive cone.\  

The next step is to  check that the isometry $\phi$ of $K\oplus K^\bot$ that acts as the matrix 
$$\begin{pmatrix}  5 & 8 \\  8 &13\end{pmatrix}$$
 on $K$ and   as $-\Id$ on $K^\perp$ extends to an isometry of $\Lambda_{\KKK}$ (one can do it by hand as in the proof above).\
 This isometry obviously preserves the ample cone of $S$, and a variation of the Torelli Theorem \ref{torth} implies that it is induced by an antisymplectic automorphism $\sigma$ of $S$.
 
Since $1$ is not an eigenvalue of $\phi$, the fixed locus of $\sigma$ contains no curves: it is therefore a finite set of points, whose cardinality $c$ can be computed by the Lefschetz fixed point formula
\begin{eqnarray*}
c&=&\sum_{i=0}^4\Tr (\sigma^*\vert_{H^i(S,\Q)})\\
&=&2+\Tr (\sigma^*\vert_{H^2(S,\Q)})\\
&=&2+\Tr (\sigma^*\vert_{K\otimes \Q})+\Tr (\sigma^*\vert_{K^\bot\otimes \Q})=2+18-20=0.
 \end{eqnarray*}

The fact that the entropy $h(\sigma)$ is positive is a consequence of results of Gromov and Yomdim that say that $h(\sigma)$ is the logarithm of 
  the largest eigenvalue of $\sigma^*$ (here $9+4\sqrt{5}$) acting on $H^{1,1}(S,\R)$.
\end{proof}

 \begin{rema}
 It is known that for any K3 surface $S$, the group $\Aut(S)$ is finitely generated.\ The proof uses the injective map  $\Psi_S$ defined in \eqref{injp} (\cite[Corollary~15.2.4]{huyk3}).
  \end{rema}

\section{Hyperk\"ahler manifolds}\label{sect3}

We study in the section generalizations of complex K3 surfaces to higher (even) dimensions.

\subsection{Definition and first properties}\label{sec21}

\begin{defi}
A {\em \hKm} is a simply connected compact K\"ahler manifold $X$ such that  $H^0(X,\Omega^2_X)=\C \omega$, where $\omega$ is a holomorphic 2-form on $X$ which is nowhere degenerate (as a skew symmetric form on the tangent space).
\end{defi}

These properties imply that the canonical bundle is trivial, the dimension of $X$ is even, say $2m$, and the abelian group $H^2(X,\Z)$ is torsion-free.\footnote{The fact that $X$ is simply connected implies  $H_1(X,\Z)=0$, hence the torsion of $H^2(X,\Z)$ vanishes by the Universal Coefficient Theorem.}\ Hyperk\"ahler manifolds of dimension 2 are K3 surfaces.\ The following result follows from the classification of compact K\"ahler manifolds with vanishing real first Chern class (see \cite{beac10}).

\begin{prop}
Let $X$ be a \hKm\ of dimension $2m$ and let $\omega$ be a generator of $H^0(X,\Omega^2_X)$.\ For each $r\in\{0,\dots,2m\}$, we have
$$H^0(X,\Omega^{r}_X)=\begin{cases}
\C\omega^{\wedge (r/2)}&\textnormal{\it{ if $r$ is even;}}\\
0&\textnormal{\it{ if $r$ is odd.}}
\end{cases}$$
In particular, $\chi(X,\cO_X)=m+1$.
\end{prop}

Hyperk\"ahler manifolds of fixed dimension $2m\ge 4$ do not all have the same topological type.\ The various possibilities  (even the possible Betti numbers) are still unknown and it is also not known whether there are finitely  many deformation types.

A fundamental tool in the study of \hKm s is the {\em Beauville--Fujiki} form, 
 a canonical integral nondivisible quadratic form $q_X$ on the free abelian group $H^2(X,\Z)$.\ Its signature is $(3,b_2(X)-3)$, it is proportional to the quadratic form
$$x\longmapsto \int_X\sqrt{\td(X)}\, x^2,$$
 and it satisfies  
$$\forall x\in H^2(X,\Z)\qquad x^{2m}=c_Xq_X(x)^m,$$
where $c_X$ (the {\em Fujiki constant}) is a positive rational number  and $m:=\frac12\dim(X)$ (in dimension~2, $q_X$ is of course the cup-product).\ Moreover, one has $q_X(x)>0$ for all K\"ahler (\eg, ample) classes $x\in H^2(X,\Z)$.

\subsection{Examples}

A few families of \hKm s are known: in each even dimension $\ge 4$,  two deformations types, which we   describe  below, were found by Beauville (\cite[Sections 6 and~7]{bea}), and two other types (in dimensions 6 and 10) were later found   by O'Grady (\cite{ognew1, ognew2}).

\subsubsection{Hilbert powers of K3 surfaces}\label{k3m}
Let $S$ be a K3 surface and let $m$ be a positive integer.\ The Hilbert--Douady space  $S^{[m]}$ parametrizes analytic subspaces of $S$ of length $m$.\ It is a smooth (Fogarty) compact K\"ahler  (Varouchas) manifold of dimension $2m$ and Beauville proved in \cite[Th\'eor\`eme~3]{bea} that it is a \hKm.\ When $m\ge 2$, he also computed the   Fujiki constant $c_{S^{[m]}}=\frac{(2m)!}{m!2^m}$ and   the second cohomology group
$$H^2(S^{[m]} ,\Z)\isom H^2(S,\Z)\oplus \Z\delta,$$
where $2\delta$ is the class of the divisor in $S^{[m]}$ that parametrizes nonreduced subspaces.\ This decomposition is orthogonal for the Beauville form $q_{S^{[m]}}$, which restricts to the intersection form on $H^2(S,\Z)$ and satisfies
$q_{S^{[m]}}(\delta)=-2(m-1)$.\ In particular, we have 
\begin{eqnarray}
(H^2(S^{[m]} ,\Z),q_{S^{[m]}})&\isom& \LKKK\oplus I_1(-2(m-1))\nonumber\\
&\isom &  U^{\oplus 3}\oplus E_8(-1)^{\oplus 2}\oplus I_1(-2(m-1))=:\Lkkk[m].\label{lk3m}
\end{eqnarray}
The second Betti number of $S^{[m]}$ is therefore 23.\ This is the maximal possible value for all \hK\ fourfolds (\cite{gua}) and sixfolds (\cite[Theorem 3]{saw}).\footnote{Added on September 15, 2020: the author of that article informed me that there is a gap in his proof of this result, which should be corrected soon.} Odd Betti numbers of $S^{[m]}$ all vanish.

The geometric structure of $S^{[m]}$ is explained in \cite[Section~6]{bea}.\ It is particularly simple when $m=2$: the fourfold $S^{[2]}$ is the quotient by the involution that interchanges the two factors of the blow-up of the diagonal in 
$S^2$.

Finally, any line bundle $M$ on $S$ induces a line bundle on each $S^{[m]}$; we denote it by~$M_m$.

\subsubsection{Generalized Kummer varieties}
Let $A$ be a complex torus of dimension 2.\ The Hilbert--Douady space  $A^{[m+1]}$   again  carries a nowhere-degenerate holomorphic 2-form, but it is not simply connected.\ We consider instead the sum morphism
\begin{eqnarray*}
A^{[m+1]}&\lra&A\\
(a_1,\dots,a_m)&\longmapsto &a_1+\dots+a_m
\end{eqnarray*}
and the inverse image $K_m(A)$ of $0\in A$.\ Beauville proved in \cite[Th\'eor\`eme~4]{bea} that it is a \hKm\ (of dimension $2m$).\ When $m=1$, the surface $K_1(A)$ is isomorphic to the blow-up of the surface $A/\pm 1$ at its 16 singular points; this is the Kummer surface of $A$.\ For this reason, the $K_m(A)$ are called generalized Kummer varieties.\ When $m\ge 2$, we have $c_{K_m(A)}=\frac{(2m)!(m+1)}{m!2^m}$ and there is again a decomposition
$$H^2(K_m(A)  ,\Z)\isom H^2(A,\Z)\oplus \Z\delta,$$
which is orthogonal for the Beauville form $q_{K_m(A)}$, and $q_{K_m(A)}(\delta)=-2(m+1)$.\ 
 The second Betti number of $K_m(A)$ is therefore 7 and
$$
(H^2(K_m(A),\Z),q_{K_m(A)}) \isom   U^{\oplus 3}\oplus I_1(-2(m+1))=:\L_{K_m}.
$$
As for all \hKm s, the first Betti number vanishes,  but not all odd Betti numbers vanish (for example, one has $b_3(K_2(A))=8$).

\subsection{The Hirzebruch--Riemann--Roch theorem}

The Hirzebruch--Riemann--Roch theorem takes the following form on  \hKm s.

\begin{theo}[Fujiki]
Let $X$ be a \hKm\  of dimension $2m$.\ There exist rational constants $a_0,a_2,\dots,a_{2m}$ such that, for every line bundle  $M$ on $X$,  one has
$$\chi(X,M)=\sum_{i=0}^m a_{2i}q_X(M)^i
.$$
\end{theo}

The relevant constants have been computed for the two main series of examples:
\begin{itemize}
\item when $X$ is the $m$th Hilbert power of a K3 surface (or a deformation), we have  (Ellingsrud--G\"ottsche--M.~Lehn)
\begin{equation}\label{chism}
\chi(X,M) = \binom{\frac12 q_X(M)+m+1}{ m}   
;
\end{equation}
\item when $X$ is a generalized Kummer variety of dimension $2m$ (or a deformation), we have (Britze)
$$\chi(X,M)=(m+1)\binom{\frac12 q_X(M)+m}{m}
.$$
\end{itemize}

\subsection{Moduli spaces for polarized hyperk\"ahler manifolds}

Quasi-projective coarse moduli spaces for polarized \hKm s $(X,H)$ of fixed dimension $2m$ and fixed degree $H^{2m}$ can be constructed using the techniques explained in Section \ref{sec14}:
Matsusaka's Big Theorem implies that there is a positive integer $k(m)$ such that, for any \hKm\ $X$ of dimension $2m$ and any ample line bundle $H$ on $X$, the line bundle $H^{\otimes k}$ is very ample for all $k\ge k(m)$,\footnote{The integer $k(m)$ was made explicit (but very large) by work of  Demailly and Siu, but is still far from the  value $k(m)=2m+2$ conjectured by Fujita in general, or for the optimistic value $k(m)=3$ conjectured (for \hKm s) by Huybrechts (\cite[p.~34]{huyk3}) and O'Grady.} and Viehweg's theorem works in any dimension.

\begin{theo}
Let $m$ and $d$ be positive integers.\ There exists   a quasiprojective coarse moduli space  for 
polarized   complex \hKm s  of dimension $2m$ and  degree $d$.
\end{theo}

The dimension of the moduli space at a point $(X,H)$  is $h^1(X,T_X)-1=h^{1,1}(X)-1=b_2(X)-3$.\ The matter of its irreducibility  will be discussed in the next section.


\subsection{Hyperk\"ahler manifolds of $\kkk[m]$-type}\label{sec25}

A   \hKm\ $X$ is said to be {\em of $\kkk[m]$-type} if it is a smooth deformation of the $m$th Hilbert power of a K3 surface.\ This fixes the topology of $X$ and its Beauville form.\ In particular, we have
$$(H^2(X,\Z),q_X) \isom \Lkkk[m]:=U^{\oplus 3}\oplus E_8(-1)^{\oplus 2}\oplus I_1(-(2m-2)).$$
Let $\ell$ be a generator for  $I_1(-(2m-2))$.\ The lattice $\Lkkk[m]$ has discriminant group $\Z/(2m-2)\Z$, generated by $\ell_*=\ell/(2m-2) $, with $\bar q(\ell_* )=-1/(2m-2)\pmod{2\Z}$.

A {\em polarization type} is the $O(\Lkkk[m] )$-orbit of a primitive element $h$ with positive square.\ It certainly determines the positive integers $2n:=h^2$ and $\gamma:=\div(h)$; the converse is not true in general,
 but it is if and only if 
  $\gamma=2$ or   $\gcd(\frac{2n}{\gamma},\frac{2m-2}{\gamma},\gamma)=1$ (\cite[Proposition~3.6, Example~3.10 and Corollary~3.7]{ghs}), \eg, when $\gcd(n,m-1)$ is square-free and odd.
%

If we write  $h=ax+b\ell$, where $a,b\in\Z$ are relatively prime and $x$ is primitive in $\L_{\KKK}$,  we have $\div(h)=\gcd(a,2m-2)$.

Even fixing the polarization type does not always give rise to {\em irreducible} moduli spaces of polarized \hKm s   of $\kkk[m]$-type.\footnote{Song checked that this is the case if and only if the divisibility is either $p^s$ or $2p^s$, where $p$ is a prime number and $s\ge0$.}

 We still have  the following result.

 \begin{theo}[Gritsenko--Hulek--Sankaran, Apostolov]\label{ghsa}
Let $n$ and $m$ be integers with $m\ge 2$ and $n>0$.\ The quasiprojective moduli space ${}^m\!\!\cM^{(\gamma)}_{2n}$ which parametrizes  \hKm s  of $\kkk[m]$-type with a polarization of  square $2n$ and divisibility $\gamma $ is    irreducible of   dimension $20$ whenever $\gamma=1$, or $\gamma=2$ and $n+m\equiv 1 \pmod4$.
\end{theo}

 We can work out in those two cases what the lattice $h^\bot$ is.\ Let  $(u,v)$ be a standard basis for a copy of $U$ in $\Lkkk[m]$ and set  $M:=  U^{\oplus 2}\oplus E_8(-1)^{\oplus 2}$.
\begin{itemize}
\item When $\gamma=1$, we have $h_*=0$ and we may take $h= u+nv$,  so that 
\begin{equation}\label{eq1}
h^\bot \isom  M\oplus I_1(-(2m-2))\oplus I_1(-2n)=:\L^{(1)}_{\KKK^{[m]},2n},
\end{equation}
with discriminant group $ \Z/(2m-2)\Z\times \Z/2n\Z$.\footnote{\label{foo11n}One has $h^\bot \isom  M\oplus \Z\ell\oplus \Z(u-nv)$, hence  $D(h^\bot)\isom    \Z/(2m-2)\Z\times \Z/ 2n\Z$, with generators   $\frac{1}{2m-2}\ell$ and $\frac{1}{2n} (u-nv) $, and intersection matrix   $\left( \begin{smallmatrix} -\frac{ 1}{2m-2} & 0 \\ 0 &-\frac{ 1}{2n}\end{smallmatrix}\right) $.}
\item When $\gamma=2$ (and $n+m\equiv 1 \pmod4$),
we have $h_*=(m-1)\ell_*$ and we may take $h=2( u+\frac{n+m-1}{4}v)+\ell$, where   $(u,v)$ is a standard basis for a copy of $U$ inside $\Lkkk[m]$, so that
\begin{equation}\label{eq2}
h^\bot \isom  M \oplus  \begin{pmatrix}-(2m-2)&-(m-1)\\-(m-1)&-\frac{n+m-1}{2}\end{pmatrix}=:\L^{(2)}_{\KKK^{[m]},2n} ,
\end{equation}
 with discriminant group of   order $n(m-1)$ and  isomorphic to $   \Z/( m-1)\Z\times \Z/ n\Z$ when $n$ (hence also $m-1$) is odd, or when $n=m-1$ (and $n$ even).\footnote{\label{foo11}We have $\L^{(2)}_{\KKK^{[m]},2n}\isom M\oplus \langle e_1,e_2\rangle$, with $e_1=(m-1)v+\ell$ and $e_2:=-u+ \frac{n+m-1}{4}v$, with intersection matrix as desired.\ It contains $M\oplus \Z e_1\oplus \Z(e_1-2e_2)\isom \L^{(1)}_{\KKK^{[m]},2n}$ as a sublattice of index 2 and discriminant group $ \Z/(2m-2)\Z\times \Z/2n\Z$.\ This implies that $D(h^\bot)$ is  the quotient by an element of order 2 of a subgroup of index 2 of $ \Z/(2m-2)\Z\times \Z/2n\Z$.\ When $n$ is odd, it is therefore isomorphic to $   \Z/( m-1)\Z\times \Z/ n\Z$.\ One can choose as generators of each factor $\frac{1}{m-1}e_1$ and $\frac{1}{n}(e_1-2e_2)$, with intersection matrix   $\left( \begin{smallmatrix} -\frac{ 2}{m-1} & 0 \\ 0 &-\frac{ 2}{n}\end{smallmatrix}\right) $.\ When $n=m-1$ (which implies $n$ even), 
 the lattices $(\Z^2, \left( \begin{smallmatrix} -2n & -n \\ -n & -n\end{smallmatrix}\right))$ and $(\Z^2, \left( \begin{smallmatrix} - n & 0 \\ 0 & -n\end{smallmatrix}\right))$ are isomorphic, hence 
 $D(h^\bot)\isom    \Z/n\Z\times \Z/ n\Z$, with generators   $\frac{1}{n}(e_1-e_2)$ and $\frac{1}{n} e_2 $, and intersection matrix   $\left( \begin{smallmatrix} -\frac{ 1}{n} & 0 \\ 0 &-\frac{ 1}{n}\end{smallmatrix}\right) $.
 
More generally, one    checks that   $D(h^\bot)\isom \Z/ n\Z\times \Z/( m-1)\Z$ if and only if $v_2(n)=v_2(m-1)$.\ If $v_2(n)>v_2(m-1)$, one has $D(h^\bot)\isom \Z/ 2n\Z\times \Z/( (m-1)/2)\Z$, and analogously if $v_2(m-1)>v_2(n)$.}
\end{itemize}

 \begin{rema}\label{remsym}
Note in both cases the symmetry $\L^{(\gamma)}_{\KKK^{[m]},2n}\isom \L^{(\gamma)}_{\KKK^{[n+1]},2(m-1)}$: it is obvious when $\gamma=1$; when $\gamma=2$, the change of coordinates $(x,y)\leftrightarrow (-x,2x+y)$ interchanges the matrices $\left( \begin{smallmatrix}  2m-2 & m-1 \\ m-1 & \frac{n+m-1}{2}\end{smallmatrix}\right)$ and $\left(\begin{smallmatrix} 2(n+1)-2 &  (n+1)-1 \\  (n+1)-1 & \frac{(m-1)+(n+1)-1}{2}\end{smallmatrix}\right)$.

Another way to see this is to note that the lattice $\Lkkk[m]$ is the orthogonal in the larger even unimodular lattice  $\widetilde{\Lambda}_{\KKK}$ defined in \eqref{mml} of any primitive vector  of square $2m-2$.\ The lattice 
$h^\bot$ considered above is then the orthogonal in $\widetilde{\Lambda}_{\KKK}$ of a (possibly nonprimitive) rank-2 lattice $\Lambda_{m-1,n}$ with intersection matrix $\left( \begin{smallmatrix}  2m-2 & 0 \\ 0& 2n\end{smallmatrix}\right)$.\ The embedding of $\Lambda_{m-1,n}$ is primitive if and only if $\gamma=1$ and any two such embedding differ by an isometry of $\widetilde{\Lambda}_{\KKK}$;  the symmetry is then explained by the isomorphism $\Lambda_{m-1,n}\isom  \Lambda_{ n,m-1}$.\ The embedding of $\Lambda_{m-1,n}$ has index~2 in its saturation   if and only if $\gamma=2$; this explains the 
symmetry in the same way (\cite[Proposition~2.2]{apo}).
\end{rema}

\subsection{Projective models of \hKm s}\label{sec36}

In this section, we   consider exclusively \hK\ manifolds (mostly fourfolds) $X$ of $\kkk[m]$-type with a polarization $H$ of divisibility $\gamma\in\{1,2\}$ and $q_X(H)=2n$ (that is, the pair $(X,H)$ represents a point of the irreducible moduli space ${}^m\!\!\cM^{(\gamma)}_{2n}$; the case $\gamma=2$ only occurs when $n+m\equiv 1\pmod4$).\ We want to know whether $H$ is very ample (at least for general $(X,H)$) and describe the corresponding embedding of $X$ into a projective space.\ We  start with a general construction.

Let $S\subset \P^{e+1}$ be a  K3 surface.\
There is  a  morphism
$$\phi_2\colon \SS\lra \Gr(2,e+2)$$
that takes a pair of points to the line that it spans in $\P^{e+1}$.\ More generally, if the line bundle $L $ defining the embedding $S\subset \P^{e+1}$ is {\em $(m-1)$-very ample,}\footnote{A line bundle $L$ on $S$ is $m$-very ample if, for every 0-dimensional scheme $Z\subset S$ of length $\le m+1$, the restriction map $H^0(S,L)\to H^0(Z,L\vert_Z)$ is surjective (in particular, $1$-very ample is just very ample).} one can define a morphism
\begin{equation}\label{phim}
\phi_m\colon S^{[m]}\lra \Gr(m,e+2)
\end{equation}
sending a 0-dimensional subscheme of $S$ of length $m$ to its linear span in $\P^{e+1}$, and $\phi_m$ is an embedding if and only if $L $ is $m$-very ample.\ The pull-back by $\phi_m$ of the Pl\"ucker line bundle on the Grassmannian has class $L_m-\delta$ on $S^{[m]}$ (with the notation of Section~\ref{k3m}).

\begin{prop}\label{pk}
Let $(S,L) $ be a polarized K3 surface  with $\Pic(S)\isom \Z L$ and $L^2=2e$.\ The line bundle  $ L^{\otimes a}$ is $k$-very ample if and only if either $a=1$ and $k\le e/2$, or $a\ge 2$ and $k\le 2(a-1)e-2$.
\end{prop}

\begin{proof}
We follow the proof of \cite[Proposition 3.1]{bcns} and use  the numerical characterization of $k$-ample line bundles on $S$ given in \cite[Theorem 1.1]{knu}.\ Set $H:=L^{\otimes a}$; that theorem implies the following.

If $a=1$ and  $k\le e/2$, the line bundle $H$ is $k$-very ample unless there exist a positive integer $n$ and a nonzero divisor $D\in |L^{\otimes n}|$ satisfying various properties, including $2D^2\le HD$, which is absurd.\ If $a\ge 2$ and $k\le 2e(a-1)-2$, we have $H^2=2ea^2\ge 2e(4a-4)>4k $, hence $H$ is $k$-very ample unless there exist a nonzero divisor $D\in |L^{\otimes n}|$ satisfying various properties, including $2D^2\le HD$, \ie, $2n\le a/2$, and
$HD\le D^2+k+1$, \ie, $2ane\le 2en^2+k+1$.\ These two inequalities imply
$ 2e(a-1)\le  2en(a-n) \le  k+1 
$, which contradicts our hypothesis.\ The divisor $D$ therefore does not exist and this proves the proposition.

The proof that these conditions are optimal is left to the reader.\end{proof}

 \begin{coro}\label{cor17}
Let $(S,L) $ be a polarized K3 surface  with $\Pic(S)\isom \Z L$ and $L^2=2e$.\ The line bundle  $aL_m-\delta$ on $S^{[m]}$ is base-point-free if
  and only if either $a=1$ and $m-1\le e/2$, or $a\ge 2$ and $m\le 2(a-1)e-1$; it is very ample 
  if and only if either $a=1$ and $m\le e/2$, or $a\ge 2$ and $m\le 2(a-1)e-2$.
\end{coro}

Let us restrict ourselves to the case   $a=1$ (resp.\ $a=2$).\ The class $aL_m-\delta$ then has divisibility $a$ and square $2(e-m+1)$ (resp.\ $2(4e-m+1) $); it is very ample when $e\ge 2m$ (resp.\ $e\ge (m+2)/2 $).\

 \begin{coro}\label{cpk}
Let $m$, $n$, and $\gamma$ be  integers with $m\ge 2$, $n\ge1$, and $\gamma\in\{1,2\}$.\
Let $(X,H)$ be a polarized  \hK\ $2m$-fold corresponding to a general point of the (irreducible) moduli space ${}^m\!\!\cM^{(\gamma)}_{2n}$.\ 
\begin{itemize}
\item When $\gamma=1$, the line bundle $H$ is base-point-free if $n\ge m-1$ and very ample if  $n\ge m+1$. 
\item When $\gamma=2$, the line bundle $H$ is base-point-free if  $n\ge m+3$ and very ample if  $n\ge m+5$. 
\end{itemize}
When $H$ is  very ample,
it defines an embedding  
$$X\lhra \P^{\binom{ n+m+1}{ m} -1}.$$
\end{coro}

\begin{proof}
Assume $\gamma=1$ (the proof in the case $\gamma=2$ is completely similar and left to the reader).\ If $(S,L )$ is a very general K3 surface of degree $2e$, the class $L_m-\delta$ on $S^{[m]}$ has divisibility 1 and square $2e-(2m-2)=:2n$.\ By Corollary~\ref{cor17}, it is base-point-free  if and only if $2(m-1)\le e=n+m-1$.\ It then defines a morphism
$$S^{[m]}\xrightarrow{\ \phi_m\ }   \Gr(m,n+m+1 )\xhookrightarrow{\ \textnormal{Pl\"ucker}\ } \P^{\binom{ n+m+1}{ m} -1}$$
which is the morphism associated with the complete linear system $|L_m-\delta|$ (compare with~\eqref{chism}).\ By 
Corollary~\ref{cor17} again, this morphism is a closed embedding 
if and only if $2m\le e=n+m-1$.

 Since base-point-freeness and very ampleness are open properties, they still hold for a general deformation of $(S^{[m]}, L_m-\delta)$, that is, for a general element of  ${}^m\!\!\cM^{(1)}_{2n}$.
\end{proof}

 \begin{exam}\label{ex19}
A general polarized K3 surface $(S,L )$ of degree 4 is a smooth quartic surface   in $\P^3$.\ Points $Z_1$ and $Z_2$ of $\SS$   have same image by $\phi_2\colon \SS\to \Gr(2,4)$  if and only if they span the same line.\ 
  If $(S,L )$ is   general, $S$ contains no lines and $\phi_2$ is finite  of degree $\binom{4}{ 2}$ (so that the class $L_2-\delta$ is ample on $\SS$, of square 2).
\end{exam}

 \begin{exam}\label{ex20}
A general polarized K3 surface $(S,L )$ of degree 6 is the   intersection of a smooth quadric $Q$ and a cubic $C$ in $\P^4$.\ Two   points $Z_1$ and $Z_2$ of $\SS$ have same image by $\phi_2\colon \SS\to \Gr(2,5)$ if and only if they span the same line.\ If $Z_1\ne Z_2$, this line   lies in $Q$.\ Conversely, 
if $(S,L )$ is very general, $S$ contains no lines,
any line contained in $Q$ (and there is a $\P^3$  of such lines)  meets $C$ in~3 points and   gives rise to 3 points of $\SS$ identified by $\phi_2$.\ The morphism $\phi_2$ is therefore finite, birational onto its image, but not an embedding  (so that the class $L_2-\delta$ is ample, but not very ample, on $\SS$).

The rational map $\phi_3\colon \sss[3] \dra \Gr(3,5)$ is not a morphism; it is dominant of degree   $\binom{6}{3}$.
\end{exam}

\begin{rema}\label{remaog}
Let $S\subset \P^{e+1} $ be a K3 surface of degree $2e$.\ Assume that $S$ contains no lines and that its ideal is generated by quadrics (by  \cite[Theorem~7.2]{sd}, this happens when $e\ge4$ and $(S,L)$ is a general polarized K3 surface of degree $2e$).\ Let $I_{S}(2)$ be the vector space of quadratic equations of $S$.\
Consider the morphism
$$\psi_2\colon S^{[2]}\lra \P(I_{S}(2)^\vee)\isom \P^{e(e-3)/2}$$
that sends a 0-dimensional subscheme $Z$ of $S$ of length 2 to the hyperplane of quadratic equations that vanish on the line spanned by $Z$.\ By \cite[Claim~5.16]{og8}, one has $\psi_2^*\cO_{\P(I_{S}(2)^\vee)}(1)=L_2-2\delta$, where  $L:=\cO_S(1)$.\ In particular, this   line bundle   is base-point-free on $S^{[2]}$; it is ample when $e\ge 5$.
 \end{rema}

\begin{exer}\label{32}
\footnotesize\narrower

Let $\phi\colon S\to \P^2$ be a double cover branched along a smooth sextic curve, so that $(S,L)$, with $L:=\phi^*\cO_{\P^2}(1)$, is a   polarized K3 surface of degree 2. 

\noindent (a) Show that the map $\phi_{L_2-\delta}$ is the rational map $ S^{[2]}\dra \P^{2\vee}$ that takes a pair of points of $S$ to the line in $\P^2$ spanned by their images by $\phi$ and that its indeterminacy locus is the image   of the embedding $\P^2\hra S^{[2]}$ given by the double cover $\phi$.

\noindent (b) Let $S^{[2]}\dra X$ be the {\em Mukai flop} of $\P^2$ (that is, the blow up of $\P^2$ in $ S^{[2]}$ followed by the contraction of its exceptional divisor in another direction).\ Show that $\phi_{L_2-\delta}$ induces a morphism $ X\to \P^{2\vee}$.\footnote{This example was presented by Mukai in \cite[Example~0.6]{muk0} as one of the first examples of \hKm s (in that article, the Mukai flop is called  {\em elementary transformation}).\ The \hK\ fourfold $X$ is the neutral component of the compactified Picard scheme of the family of curves in $|L|$.}

\noindent (c)
Show that $3L$ is very ample and that the image of $\phi_{3L}\colon S\to \P^{10}$ is the intersection of a quadric with a cone over the image of the Veronese map $\phi_{\cO_{\P^2}(3)}\colon \P^2\to \P^{9}$.

\noindent (d) Show that the line bundle $3L_2-2\delta$ on $S^{[2]}$ is base-point-free, hence nef, but not ample: the associated morphism   contracts the image of the canonical embedding $\P^2\hra S^{[2]}$, factors through the involution of $S^{[2]}$ induced by the involution of $S$ associated with $\phi$, and has degree $2$.\ In particular, $3L_2-2\delta$ is also base-point-free on $X$.
  \end{exer}

\subsubsection{Hyperk\"ahler fourfolds of low degrees}\label{sec261}
We review the known descriptions of the morphism $\phi_H\colon X\to \P^{}$  for $(X,H)$ general in $ {}^2\!\!\cM^{(\gamma)}_{2n}$ and small $n$ (recall that the case $\gamma=2$ only occurs when $n\equiv -1\pmod4$).\ By  \eqref{chism} and Kodaira vanishing, we have
$$h^0(X,H)=\chi(X,H) = \binom{n+3}{ 2} 
.$$

\noindent{$\mathbf{q(H)=2}$.}  O'Grady showed that for $(X,H)$ general in $ {}^2\!\!\cM^{(1)}_2$, the map $\phi_H\colon X\to \P^5$ is a morphism (as predicted by Corollary \ref{cpk}) which is a double ramified cover of a  singular sextic hypersurface called an {\em EPW sextic} (for Eisenbud--Popescu--Walter).\ Any such smooth {\em double EPW sextic} is a polarized \hK\ fourfold of   degree $2$.\ We saw in Example~\ref{ex19}  (nongeneral)  examples where the morphism $\phi_H$ is 6:1 onto the quadric $\Gr(2,4)\subset \P^5$.

\noindent{$\mathbf{q(H)=4}$.}  There is no geometric description of  general elements $(X,H)$ of $\,{}^2\!\!\cM^{(1)}_4$.\ In particular, it is not known whether $\phi_H\colon X\to \P^9$ is an embedding (\eg, whether $H$ is very ample).\ Example~\ref{ex20} shows that $\phi_H$ is birational onto its image, which therefore has degree~48.\ The pairs $(X,H)=(\SS,L_2-\delta)$, where $(S,L)$ is a polarized K3 surface of degree $6$, form a hypersurface in ${}^2\!\!\cM^{(1)}_4$; for these pairs, $\phi_H(X)$ is a nonnormal fourfold in $\Gr(2,5)\subset \P^9$.\ The general elements $(X,H)$ of another hypersurface in ${}^2\!\!\cM^{(1)}_4$ where described in \cite{ikkr}; for those pairs, the morphism $\phi_H$ is a double cover of a singular quartic section of the cone over the Segre embedding $\P^2\times \P^2\hra \P^8$ (a fourfold of degree 24).

Note that $\dim(\Sym^2H^0(X,H))=\binom{9+2}{2}$, whereas $H^0(X,H^{\otimes 2})= \binom{8+3}{ 2} $ by \eqref{chism}.\ One would expect the canonical map
$$\Sym^2H^0(X,H)\lra H^0(X,H^{\otimes 2})
$$
to be an isomorphism for $(X,H)$ general in $\,{}^2\!\!\cM^{(1)}_4$, but this does not hold for the two families of examples we just described.

\noindent{$\mathbf{q(H)=6, \gamma=1}$.}  There is  no geometric description of  general elements $(X,H)$ of $\,{}^2\!\!\cM^{(1)}_6$.\ The morphism $\phi_H\colon X\to \P^{14}$ is a closed embedding by Corollary \ref{cpk}.\ Moreover, the pairs $(\SS,L_2-\delta)$, where $(S,L)$ is a polarized K3 of degree $8$, form a hypersurface in $\,{}^2\!\!\cM^{(1)}_6$.

\noindent{$\mathbf{q(H)=6, \gamma=2}$.}   General elements $(X,H)$ of $\,{}^2\!\!\cM^{(2)}_6$ can be described as follows.\
Let $W\subset \P^5$ be a smooth cubic hypersurface.\ Beauville--Donagi showed in \cite{bedo} that the family $F(W)\subset \Gr(2,6)\subset \P^{14}$ of lines  contained in $W$ is a \hK\ fourfold and that the Pl\"ucker polarization $H$ has square 6 and divisibility 2.\ General elements of $\,{}^2\!\!\cM^{(2)}_6$ are  of the form $(F(W),H)$.\ We have $h^0(X,H)=\binom{6}{2}=15$, and $\phi_H$ is the closed embedding 
$F(W)\subset \Gr(2,6)\subset \P^{14}$; in particular, $H$ is very ample.

Pairs $(\SS,2L_2-\delta)$, where $(S,L)$ is a polarized K3 surface of degree $2$ with $\Pic(S)=\Z L$, form a family of codimension 1 in $\,{}^2\!\!\cM^{(2)}_6$ which is   disjoint from the family of varieties of lines in cubic fourfolds described above.\footnote{The line bundle $2L_2-\delta$ is ample (Example \ref{exa217} or Exercice~\ref{32}) but
it  has base-points:  the  morphism $\phi_{2L}\colon S\to \P^5$ is the composition of the double cover $\phi_{L}\colon S\to \P^2$ with the Veronese embedding and the map $\phi_2\colon S^{[2]}\dra \Gr(2,6)$ is not   defined along the image of the embedding $\P^2\hra S^{[2]}$ given by the double cover $\phi_L$.\ Therefore, $(\SS,2L_2-\delta)$ cannot be the variety of lines of a smooth cubic fourfold.}\ It is interesting to mention that Beauville--Donagi proved that $F(W)$ is a \hK\ fourfold by exhibiting a codimension-1 family of cubics $W$ for which $(F(W),H)$ is isomorphic to  $(\SS,2L_2-5\delta)$, where $(S,L)$ is a polarized K3 surface of degree 14 (we check $q_{\SS}(2L_2-5\delta)=2^2\cdot 14+5^2(-2)=6$).

\noindent{$\mathbf{q(H)=22, \gamma=2}$.}   General elements $(X,H)$ of $\,{}^2\!\!\cM^{(2)}_{22}$ were   described by Debarre--Voisin in \cite{devo} as follows.\ Let $\cU$ be the tautological rank-6 subbundle on the Grassmannian $\Gr(6,10)$.\ Any smooth  zero locus $X\subset \Gr(6,10)$ (of the expected dimension 4) of a section $s$ of $\cE:=\bw3\cU^\vee$ is a \hK\ fourfold, the Pl\"ucker polarization   has square   22 and divisibility 2, and general elements of $\,{}^2\!\!\cM^{(2)}_{22}$ are  of this form. 

One checks using the Koszul complex for $\cE$ that $X\subset \Gr(6,10)\subset  \P(\bw{6}{\C^{10}})$ is actually contained in a linear subspace $\P^{90}\subset  \P(\bw{6}{\C^{10}})$.\footnote{We need to compute $h^0(\Gr(6,10),\cI_X(1))$.\ Only two terms from  the Koszul complex contribute: $H^0(\Gr(6,10),\cE^\vee(1))\isom \bw3\C^{10\vee}$, of dimension $\binom{10}{3}=120$, and 
 $H^0(\Gr(6,10),\bw2\cE^\vee(1))\isom H^0(\Gr(6,10),\bw6\cU(1))\isom H^0(\Gr(6,10),\cO_{\Gr(6,10)})$, of dimension 1.\ Hence $H^0(\Gr(6,10),\cI_X(1))$ has dimension $120-1=\dim(\bw{6}{\C^{10}})-91$ (thanks to L.\ Manivel for doing these computations).
 
 More intrinsically, one sees that $X$ is the linear section of $\Gr(6,10)$   by the projectivization of the kernel of the contraction map ${}\lrcorner\,s\colon \bw6 \C^{10}\lra \bw3 \C^{10}
$, whose image is the hyperplane defined by $s$ (here $s$ is the element of $H^0(\Gr(6,10), \cE)\isom \bw3 \C^{10\vee}$ that defines $X$).}\ By \eqref{chism}, we have $h^0(X,H)=\binom{14}{2}=91$ and   $\phi_H$ is the closed embedding $X\hra \P^{90}$; in particular, $H$ is very ample, as predicted by Corollary~\ref{cpk}.\ 
 
 Pairs $(S^{[2]}, 2 L_2-\delta)$, where~$(S,L)$ is a general polarized K3 surface of degree 6, form an irreducible hypersurface in $\,{}^2\!\!\cM^{(2)}_{22}$.

\noindent{$\mathbf{q(H)=38, \gamma=2}$.}   General elements $(X,H)$ of $\,{}^2\!\!\cM^{(2)}_{38}$ can be described as follows: for a general cubic polynomial $P$  in 6 variables,
$$\VSP(P,10):=\overline{\{(\ell_1,\dots,\ell_{10})\in \Hilb^{10}(\P^5)\mid P\in \langle \ell_1^3,\dots,\ell_{10}^3  \rangle\}}
$$
is a (smooth) \hK\ fourfold with a natural embedding into $\Gr(4,\bw4\C^6)$  (\cite{ir1}).\ It was checked in \cite[Proposition~1.4.16]{mon} that the Pl\"ucker polarization restricts to a polarization $H$ of square $38$ and divisibility 2 (by checking that this polarization is $2L_2-3\delta$ when $\VSP(P,10)$ is isomorphic to the Hilbert square of a very general polarized K3 surface $(S,L)$ of degree~14).\ A general element of $\,{}^2\!\!\cM^{(2)}_{38}$ is of the form $(\VSP(P,10),H)$. 

The line bundle $H$ is very ample by Corollary \ref{cpk}, and $h^0(X,H)=\binom{22}{2}=231$.\ It is likely that the embedding $X\hra  \Gr(4,\bw4\C^6)\xhookrightarrow{\textnormal{Pl\"ucker}}  \P(\bw4(\bw4\C^6)) $ factors as
 $X\xhookrightarrow{\phi_H} \P^{230} \subset\P(\bw4(\bw4\C^6)) $.

\subsubsection{An example in dimension 6 {\rm($\mathbf{m=3, q(H)=4, \gamma=2}$)}}\label{sec262a}
General elements $(X,H)$ of $\,{}^3\!\!\cM^{(2)}_{4}$ are described in \cite[Theorem 1.1]{ikkr1} as double covers of certain degeneracy loci   $D\subset \Gr(3,6)\subset \P^{19}$.\ By \eqref{chism}, we have $h^0(X,H)= 20$, and $\phi_H$ factors as   
$$\phi_H\colon X\xrightarrow{\ 2:1\ } D\lhra \Gr(3,6)\xhookrightarrow{\textnormal{Pl\"ucker}}  \P(\bw3\C^6) .$$

\subsubsection{An example in dimension 8 {\rm($\mathbf{m=4, q(H)=2, \gamma=2}$)}}\label{sec262}
 Let again $W\subset \P^5$ be a smooth cubic hypersurface that contains no planes.\ The moduli space $\cM_3(W)$ of generalized twisted cubic curves on $W$ is a smooth and
irreducible projective variety of dimension~10, and there is a contraction $\cM_3(W)\to X(W)$, where $X(W)$ is a projective  \hKm\ of type $\kkk[4]$ (\cite{lls,adle}).\ The maps that takes a cubic curve to its span defines a morphism from $\cM_3(W)$ to $\Gr(4,6)$ which factors through a surjective rational map
$ X(W)\dra \Gr(4,6)$ 
of degree 72.\ The Pl\"ucker polarization on $\Gr(4,6)$ pulls back to a polarization $H$ on $X(W)$ with   $q_{X(W)}(H)=2$.\  By \eqref{chism}, we have $h^0(X(W),H)= 15$, and $\phi_H$ is the rational map
 $$\xymatrix    @C=40pt
{\phi_H\colon  X(W)\ar@{-->>}[r]^-{72:1}& \Gr(4,6)\ar@{^(->}[r]^-{\textnormal{Pl\"ucker}} & \P(\bw4\C^6) .}$$
It follows from \cite{adle}\footnote{The main result of \cite{adle} is that $X(W)$ is a deformation of $S^{[4]}$, where $(S,L)$ is a very general polarized K3 surface of degree 14.\ One can therefore write the polarization  on $X(W)$ as $H=aL_4-b\delta$, with $2=H^2=14a^2-6b^2$, hence $a^2\equiv b^2 -1\pmod4$.\ This implies that $a$ is even, hence so is $\gamma$.\ Since $\gamma \mid H^2$, we get $\gamma=2$ (note that the condition $n+m\equiv 1\pmod4$ of Theorem~\ref{ghsa} holds).} that the divisibility $\gamma$ is  2  and a general element of $\,{}^4\!\!\cM^{(2)}_{2}$ is of the form $(X(W),H)$ (we will actually see in Proposition~\ref{prop:BayerMongardi} that {\em all} elements of $\,{}^4\!\!\cM^{(2)}_{2}$ are of this type).

\begin{rema}
It follows from  the above descriptions that the moduli spaces $\,{}^2\!\!\cM^{(1)}_2$, $\,{}^2\!\!\cM^{(2)}_6$, $\,{}^2\!\!\cM^{(2)}_{22}$,  $ \,{}^2\!\!\cM^{(2)}_{38}$, $ \,{}^3\!\!\cM^{(2)}_{4}$, and $ \,{}^4\!\!\cM^{(2)}_{2}$ are unirational.\ It was proved in \cite{ghs} that $\,{}^2\!\!\cM^{(1)}_{2n}$ is of general type for all $n\ge 12$.
\end{rema}

 \subsection{The nef and movable cones}\label{sec37}
 
 Let $X$ be a projective \hKm.\ We define its {\em positive cone}  $\Pos(X)$ as that of the two components of
 the cone $\{x\in \Pic(X)\otimes\R\mid q_X(x)> 0 \}$ that contains one (hence all) ample classes.

 The (closed) {\em movable cone} 
  $$\Mov(X)\subset \Pic(X)\otimes \R$$
   is the closed convex cone generated by classes of   line bundles on $X$ whose base locus has codimension $\ge 2$ (no fixed divisor).\ It is not too difficult to prove the inclusion\footnote{One may follow the argument in the proof of \cite[Theorem~7]{hast2} and  compute explicitly the Beauville--Fujiki form on a resolution of the rational map induced by the complete linear system of the movable divisor.}
   $$\Mov(X)\subset \overline{\Pos(X)}   .$$
We
defined in Section~\ref{newsecc}  the  nef cone $\Nef(X)\subset \Pic(X)\otimes \R$; we have of course
 \begin{equation}\label{inc}
\Nef(X)\subset \Mov(X).
\end{equation}
The importance of the movable cone (which, for K3 surfaces, is just the nef cone) stems from the following result.

\begin{prop}\label{prop110b}
Let $X$ and $X'$ be  \hKm s.\ Any birational isomorphism   $\sigma\colon X\isomdra X'$ induces a Hodge isometry    $\sigma^*\colon (H^2(X',\Z),q_{X'})\isomlra (H^2(X,\Z),q_X)$.

Assume moreover that $X$ and $X'$ are moreover projective.\footnote{Since any compact K\"ahler manifold which is Moishezon is projective, $X$ is projective if and only if $X'$ is projective.\ The general statement here is that if $\sigma^*$ maps a K\"ahler class of $X'$ to a K\"ahler class of $X$, the map $\sigma$ is an isomorphism (\cite[Proposition~27.6]{ghj}).} 
We have then,
$$\sigma^*(\Mov(X'))=\Mov(X),$$
and if $\sigma^*(\Nef(X'))$ meets  $\Amp(X)$, the map $\sigma$ is an isomorphism and $\sigma^*(\Nef(X'))=\Nef(X)$.
\end{prop}

\begin{proof}[Sketch of proof]
Let $U\subset X$ (resp.\ $U'\subset X'$) be the largest open subset on which $\sigma$ (resp.\ $\sigma^{-1}$) is defined.\ We have $\codim_X(X\moins U)\ge 2$ 
hence, since $X$ is normal and $\Omega_X^2$ is locally free, restriction induces an isomorphism $H^0(X, \Omega_X^2)\isomto H^0(U, \Omega_U^2)$.\ These vector spaces are spanned by the symplectic form $\omega$ and, since $(\sigma\vert_U)^*\omega'$ is nonzero, it is a nonzero multiple of $\omega\vert_U$.\ Since this 2-form is nowhere degenerate, $\sigma\vert_U$ is quasifinite and, being birational, it is an open embedding by Zariski's Main Theorem.\ This implies that $\sigma$ induces an isomorphism between $U$ and $U'$.\ Since $\codim_X(X\moins U)\ge 2$, the restriction   $H^2(X, \Z)\isomto H^2(U, \Z)$ is an isomorphism\footnote{If $N:=\dim_\C(X)$ and $Y:=X\moins U$, this follows from example from the long exact sequence
$$   H_{2N-2}(Y,\Z) \to H_{2N-2}(X,\Z)\to H_{2N-2}(X,Y,\Z)\to H_{2N-3}(Y,\Z),
$$
the fact that $H_j(Y,\Z)=0$ for $j>2(N-2)\ge \dim_\R(Y)$, and the duality isomorphisms $H_{2N-i}(X,\Z) \isom H^i(X,\Z)$ 	and $H_{2N-i}(X,Y,\Z) \isom H^i(X\moins Y,\Z)$.}
 and we get an isomorphism $\sigma^*\colon H^2(X',\Z)\isomto H^2(X,\Z)$ of Hodge structures.
 
For the proof  that $\sigma^*$ is an isometry, we refer to \cite[Section~27.1]{ghj}.\ Given a line bundle  $M'$ on $X'$, we have isomorphisms
\begin{equation}\label{iso21}
H^0(X',M')\isomto H^0(U',M')\isomto H^0(U,\sigma^*M')\isomtol  H^0(X,\sigma^*M')
\end{equation}
and it is clear that $\sigma^*$ maps $\Mov(X')$ to $\Mov(X)$.

Finally, if $\sigma^*$ maps an (integral) ample class $H'$ on $X'$ to an ample class $H$ on $X$, we obtain by \eqref{iso21}  isomorphisms
$  H^0(X',H^{\prime\otimes k})\isomto H^0(X,(\sigma^*H')^{\otimes k}))=H^0(X,H^{\otimes k})$  for all $k\ge 0$.\ Since  $X=\Proj(\bigoplus_{k=0}^{+\infty}H^0(X,H^{\otimes k}))$ and $X'=\Proj(\bigoplus_{k=0}^{+\infty}H^0(X',H^{\prime\otimes k}))$, this means that $\sigma$ is an isomorphism.
\end{proof}

   If $X'$ is a  (projective) \hKm\ with a birational map $\sigma\colon X\isomdra X'$, one  identifies $H^2(X',\Z)$ and $  H^2(X,\Z)$ using $\sigma^*$ (Proposition~\ref{prop110b}).\ By \cite[Theorem~7]{hast2},     we have 
\begin{equation}\label{chamb}
\Mov(X)=\overline{\bigcup_{\sigma\colon X\isomdra X'}\sigma^*(\Nef(X'))},
\end{equation}
      where the various cones $\sigma^*(\Nef(X'))$ are either equal or have disjoint interiors.\ It is known that for any \hKm\ $X$ of $\KKK^{[m]}$-type, the set of isomorphism classes of \hKm s  
birationally isomorphic to $X$
 is finite 
 (\cite[Corollary~1.5]{mayo}).

\subsubsection{Nef and movable cones of \hK\ fourfolds of  $\KKK^{[2]}$-type}
 
 The nef and movable cones of \hKm s of  $\KKK^{[m]}$-type are  known, although their concrete descriptions can be quite complicated.\  They are best explained in terms of the Markman--Mukai lattice, which we now define (see~\cite[Section 9]{marsur} and~\cite[Section~1]{bht} for more details).
 
 We define the \emph{extended K3 lattice} 
\begin{equation}\label{mml}
 \widetilde{\Lambda}_{\KKK} := U^{\oplus 4}\oplus E_8(-1)^{\oplus 2}.
\end{equation}
It is even, unimodular of signature $(4,20)$ and the  lattice $\Lkkk[m]$ defined in \eqref{lk3m} is the orthogonal of any primitive vector of square $2m-2$.\ 

Given a \hKm\ $X$ of  $\KKK^{[m]}$-type, there 
  is a canonical  extension 
  $$\theta_X\colon H^2(X,\Z)\hra\widetilde{\Lambda}_X$$
  of lattices and weight-2 Hodge structures, where the lattice $\widetilde{\Lambda}_X$  is isomorphic to $\widetilde{\Lambda}_{\KKK}$.\  A generator $\bv_X$ (with square $2m-2$) of $H^2(X,\Z)^\bot$ is of type $(1,1)$.\ 
 We denote by $\widetilde{\Lambda}_{\textnormal{alg},X}$ the algebraic (\ie, type $(1,1)$) part of $\widetilde{\Lambda}_X $, so that $\Pic(X)= \bv_X^\bot\cap \widetilde{\Lambda}_{\textnormal{alg},X}$.

Given a class $  \bs\in \widetilde{\Lambda}_X$, we define a hyperplane
$$H_{  \bs}:=\{ x\in \Pic(X)\otimes\R \mid x\cdot  {\bs}=0 \}
.$$

In order to keep things simple, we will limit ourselves to the descriptions of the nef and movable cones of \hKm s of  $\KKK^{[2]}$-type ($m=2$).\ The general case will be treated in Theorem~\ref{thm:BHT}  (see also Example~\ref{exa215}).

\begin{theo}\label{thm:NefConeHK4}
Let $X$ be a hyperk\"ahler fourfold of $\KKK^{[2]}$-type.

\noindent {\rm (a)} The interior  $\Int(\Mov(X))$ of the movable cone is the connected component of
\[
\Pos(X) \smallsetminus  \bigcup_{\kappa\in \Pic(X)\atop \kappa^2=-2}   H_\kappa
\]
that contains the class of an ample divisor.

\noindent {\rm (b)} The ample cone $\Amp(X)$ is the connected component of
\[
\Int(\Mov(X)) \smallsetminus \bigcup_{\kappa\in \Pic(X)\atop {\kappa^2=-10 \atop  \div_{H^2(X,\Z)}(\kappa)=2}}\hskip-5mm  H_\kappa
\]
that contains the class of an ample divisor.
\end{theo}

\begin{proof}
Statement (a)  follows from the general result~\cite[Lemma~6.22]{marsur} (see also \cite[Proposition~6.10, Proposition~9.12, Theorem~9.17]{marsur}).\ More precisely, this lemma says that the interior of the movable cone is the connected component that contains the class of an ample divisor of the complement in $\Pos(X)$ of the union of the hyperplanes $H_{ \bs}$, where $ \bs\in \widetilde{\Lambda}_{\textnormal{alg},X}$ is such that
  $ \bs^2= -2$ and $ \bs\cdot \bv_X=0$.\ The second relation means $ \bs\in \Pic(X)$, hence (a) is proved.

As to (b), the dual statement of~\cite[Theorem 1]{bht} says\footnote{This still requires some work, and reading \cite[Sections 12 and 13]{bama} might help.} that the ample cone is the connected component containing the class of an ample divisor of the complement in $\Pos(X)$ of the union of the hyperplanes $H_{  \bs}$, where $  \bs\in \widetilde{\Lambda}_{\textnormal{alg},X}$ is such that
  $  \bs^2=-2$   and $  \bs\cdot \bv_X\in\{0,1\}$.\ 
 
Writing $  \bs=a\kappa+ b\bv_X$, with $\kappa\in\Pic(X)$ primitive and $a,b\in\frac12\Z$, we get $-2=  \bs^2=a^2\kappa^2+2b^2$ and $  \bs\cdot \bv_X=2b$.\ Inside $\Int(\Mov(X))$, the only new condition  therefore corresponds to $b=\frac12$ and $a^2\kappa^2= -\frac{5}{2}$, hence $a= \frac12$ and $\kappa^2= -10$.\ Moreover, $\kappa\cdot H^2(X,\Z)=2   \bs\cdot \widetilde{\Lambda}_X=2\Z$, hence $\div_{H^2(X,\Z)}(\kappa)=2$.\ Conversely, given any such $s$, the element $\kappa+\bv_X$ of $ \widetilde{\Lambda}_X$ is divisible by 2 and  $  \bs:=\frac12(\kappa+\bv_X)$ satisfies $  \bs^2=-2$, $  \bs\cdot \bv_X=1$, and $H_\kappa=H_{  \bs}$.\ This proves (b).\end{proof}

\begin{rema}\label{rmk:NefConeHK4}
We can make the descriptions in  Theorem~\ref{thm:NefConeHK4} more precise.

\noindent (a) As explained in~\cite[Section 6]{marsur}, it follows from~\cite{mar3} that there is a group of reflections $W_{\Exc}$ acting on $\Pos(X)$ that acts faithfully and transitively on the set of connected components of 
\[
\Pos(X) \smallsetminus  \bigcup_{\bs\in \Pic(X)\atop \bs^2=-2}   H_\bs
\]
In particular,   $\overline{\Mov(X)} \cap \Pos(X)$ is a fundamental domain for the action of $W_{\Exc}$ on $\Pos(X)$.

\noindent (b) By~\cite[Proposition 2.1]{mayo} (see also~\cite[Theorem 7]{hast2}), each connected component of
\[
\Int(\Mov(X)) \smallsetminus \bigcup_{\bs\in \Pic(X)\atop {\bs^2=-10 \atop  \div_{H^2(X,\Z)}(\bs)=2}}\hskip-5mm  H_\bs
\]
corresponds to the ample cone of a hyperk\"ahler fourfold $X'$ of $\KKK^{[2]}$-type  via   a
  birational isomorphism $X\isomdra X'$ (compare with \eqref{chamb}).
\end{rema}

  \subsubsection{Nef and movable cones of punctual Hilbert schemes of K3 surfaces}\label{sec210a}
 
 In this section, $(S,L)$ is a polarized K3 surface  of degree $2e$ with $\Pic(S)=\Z L$  and  $m$ is an integer such that $m\ge 2$.\ We describe the nef and movable cones of the $m$th Hilbert power $S^{[m]}$.\ The line bundle $L_m$ on $S^{[m]}$ induced by $L$ is nef and big and spans a ray  which is extremal for  both cones $\Mov(\sss[m])$ and $\Nef(\sss[m])$.\ Since the Picard number of $ \sss[m] $ is 2, there is just one ``other'' extremal ray for each cone.

   \begin{exam}[Nef and movable cones of $S^{[2]}$]\label{exa217}
  By Theorem~\ref{thm:NefConeHK4}, cones of divisors on $S^{[2]}$ can be described as follows (see Appendix~\ref{secpell} for the notation).
\begin{itemize}
\item[(a)] The other extremal ray of the (closed) movable cone $\Mov(S^{[2]})$ is spanned by $L_2-\mu_e\delta$, where
\begin{itemize}
\item[$\bullet$] if $e$ is a perfect square, $\mu_e=\sqrt{e}$;
\item[$\bullet$] if $e$ is not a perfect square and $(a_1,b_1)$ is the minimal solution of the equation $\cP_e(1)$, $\mu_e=e\frac{b_1}{a_1}$.
\end{itemize}
\item[(b)] The other extremal ray of the nef cone $\Nef(S^{[2]})$ is spanned by  $L_2-\nu_e\delta$, where
\begin{itemize}
\item[$\bullet$] if the equation $\cP_{4e}(5)$ is not solvable, $\nu_e=\mu_e$;
\item[$\bullet$] if the equation $\cP_{4e}(5)$ is solvable and 
  $(a_5,b_5)$ is its minimal solution, $\nu_e =  2e\frac{b_5}{a_5}$.\footnote{\label{newf}There is a typo in~\cite[Lemma 13.3(b)]{bama}: one should replace $d$ with $2d$.\ Also, the   general theory developed in \cite{bama} implies that this ray does lie  inside the movable cone; this is also a consequence of the elementary inequalities \eqref{slope2} and \eqref{slope3}.}
\end{itemize}
\end{itemize}

\begin{exer}\label{lbound}
\footnotesize\narrower
Prove that in all cases, one has $\nu_e\ge \lfloor \sqrt{e}\rfloor$, with equality if and only if $e$ is a perfect square and $e>1$.
%
\end{exer}

The walls of the decomposition  \eqref{chamb} of the movable cone $ \Mov(S^{[2]}) $ correspond exactly to the solutions of the equation $\cP_{4e}(5)$ that give rays inside that cone and there are only three possibilities (see Appendix~\ref{secpell} and in particular Lemma \ref{le46}, which says that the equation $\cP_{4e}(5)$ has at most two classes of solutions when $e$ is not a perfect square): 
\begin{itemize}
\item either   the equation $\cP_{4e}(5)$ is not solvable, in which case the nef and movable cones are equal; 
\item or   the equation $\cP_{4e}(5)$ has a   minimal solution   $(a_5,b_5)$, and, if $(a_1,b_1)$ is the minimal solution to  the equation $\cP_{e}(1)$ (set $b_1=1$ when $e=1$),
\begin{itemize}
\item either $b_1$ is odd, or $b_1$ is even and $5\mid e$, in which cases there are two chambers in the movable cone  and the middle wall is spanned by  $ a_5L_2-2eb_5\delta$;
\item or $b_1$ is even and $5\nmid e$,  in which case  there are three chambers   in the movable cone and the middle walls are spanned by  $ a_5L_2-2eb_5\delta$ and $(a_1a_5-2eb_1b_5)L_2-e(a_5b_1-2a_1b_5)\delta$ respectively.
\end{itemize}
\end{itemize} 
In the last case, for example, there are three ``different'' birational isomorphisms $S^{[2]}\dra X$, where $X$ is a \hKm\ of dimension 4.

The nef and movable cones of  $\sss[2]$  are computed in the table below for $1\le e \le 13$.\ The first two lines give the minimal solution of the Pell-type equations $\cP_{e}(1)$ and $\cP_{4e}(5) $  (when they exists).\ A~$\star$ means that $e$ is a perfect square.\ The last two lines indicate the ``slope'' $\nu$ of the ``other'' extremal ray of the cone (that is, the ray is generated by $L_2-\nu\delta$).\ The  $=$ sign means that the two cones are equal.
$$\tiny{
 \setlength{\extrarowheight}{1ex}
 \begin{array}{|c|c|c|c|c|c|c|c|c|c|c|c|c|c|c|c|c|c|}
 \hline 
e&1&2&3&4&5&6&7&8&9&10&11&12&13    \\
 \hline 
\cP_{e}(1)&\star&(3,2)&(2,1)&\star& (9,4)&(5,2)&(8,3)&(3,1)&\star&(19,6)&(10,3)&(7,2)& (649,180)\\
\cP_{4e}(5)&\star&-&-&\star&(5,1)&-&-&-&\star&-&(7,1)&-&-
\\
 \Mov(S^{[2]}) 
 &1&\frac43&\frac32&2& \frac{20}{9}&\frac{12}5&\frac{21}8&\frac83&3&\frac{60}{19}&\frac{33}{10}&\frac{24}7&\frac{2340}{649}  
\\
 \Nef(S^{[2]}) 
 &\frac23&=&=&=&2&=&=&=&=&=&\frac{22}7&=&= 
\\[5pt]
 \hline
 \end{array}}$$
  
 This table shows that when $e=1$, the fourfold $\sss[2]$ has a nontrivial birational model.\ This model was described in Exercise~\ref{32}: it is the Mukai flop of $\sss[2]$ along a $\P^2$.
 
 Let $(S,L)$ be a general polarized K3 surface of degree 22.\ It was shown in \cite[Proposition~3.4]{devo} that there is a a birational isomorphism $\sigma\colon S^{[2]}\dra X$, where $(X,H)$ is a nodal degeneration of polarized \hKm s   of  degree 22 and divisibility 2 (see Section~\ref{sec261}), such  that $\sigma^*H=10L_2-33\delta $ (this implies that this line bundle is movable).\ The table above shows that this class is on the boundary of the movable cone of $S^{[2]}$.\ The corresponding birational model of $S^{[2]}$ (via the decomposition \eqref{chamb}) is the minimal desingularization of $X$, on which the class $H$ is nef but not ample. 
 
 In the next table, we restrict ourselves to   integers $1<e\le 71$  for which the equation 
 $\cP_{4e}(5)$ is solvable (in other words, to those $e$ for which the nef and movable cones are distinct).\ We indicate all solutions of that equation that give rise to rays that are in the movable cone (that is, with slope smaller than that of the ``other'' ray of the movable cone): they are the walls for the decomposition \eqref{chamb}.\ As expected, there are two rays if and only if $b_1$ is even and $5\nmid e$.
 $$\tiny{
 \setlength{\extrarowheight}{1ex}
 \begin{array}{|c|c|c|c|c|c|c|c|c|c|c|c|c|c|c|c|c|c|}
 \hline 
 e&5&11&19&29&31&41 &55&71 
  \\
 \hline 
\cP_{e}(1)& (9,4)& (10,3)&(170,39)& (9801 ,1820)&(1520,273)&(2049,320)&(89,12)&(3480,413)\\
\cP_{4e}(5)&(5,1)& (7,1)&(9,1)&(11,1), (2251,209)&(657,59)&(13,1),(397,31)&(15,1)&(17,1)
\\
 \Mov(S^{[2]}) 
 & \frac{20}{9}& \frac{33}{10}&\frac{741}{170}&\frac{52780}{9801}  &\frac{8463}{1520}&\frac{13120}{2049}&\frac{660}{89}&\frac{29323}{3480}
\\
\textnormal{Walls for \eqref{chamb}} 
 &2 &\frac{22}7&\frac{38}{9}&\frac{58}{11} ,\frac{12122}{2251}&\frac{ 3658}{657}&\frac{82 }{13},\frac{2542}{397}&\frac{ 22}{3}&\frac{ 142}{17}
\\[5pt]
 \hline
 \end{array}}$$
\end{exam}

  \begin{exam}[Nef and movable cones of $S^{[m]}$]\label{exa215}
 A complete description of the cone $\Mov(\sss[m])$ is given in \cite[Proposition~13.1]{bama}.\footnote{There is a  typo in   \cite[(33)]{bama} hence also in statement (c) of that proposition.} One extremal ray is spanned by $L_m$ and generators of the other extremal ray are given as follows:
\begin{itemize}
\item if $e(m-1)$ is a perfect square, a generator is   $(m-1)L_m-\sqrt{e(m-1)}\delta$ (with square~$0$);
\item if $e(m-1)$ is not a perfect square and  the equation $(m-1)a^2-eb^2=1$ has a minimal   solution $(a_1,b_1)$, a generator is  $(m-1)a_1L_m-eb_1\delta$ (with square $2e(m-1)$);
\item otherwise, a generator is $a'_1L_m-eb'_1\delta$ (with square $2e$), where  $(a'_1,b'_1)$ is the minimal positive solution of the equation $a^2-e(m-1)b^2=1$ that satisfies 
$a'_1\equiv\pm 1\pmod{{m-1}}$.\footnote{\label{foo16}If $(a,b)$ is a solution to that equation, its ``square'' $(a^2+e(m-1)b^2,2ab) $ is also a solution and $a^2+e(m-1)b^2\equiv a^2\equiv 1\pmod{m-1}$.\ So there always exist solutions with the required property.

Also, we always have the inequalities $  \frac{eb_1}{a_1(m-1)}\le \frac{eb'_1}{a'_1}<\sqrt{\frac{e}{m-1}}$ between slopes and the first inequality is strict when $m\ge 3$ (when the equation $(m-1)a^2-eb^2=1$ has a minimal positive solution $(a_1,b_1)$ and $m\ge 3$, the minimal solution of the equation $\cP_{e(m-1)}(1)$ is $(2(m-1)a_1^2-1,2a_1b_1)$ by Lemma \ref{lemmpell} and this is $(a'_1,b'_1)$; hence $\frac{eb'_1}{a'_1}>\frac{eb_1}{a_1(m-1)}$).}
\end{itemize}
 
 A complete description of the cone $\Nef(\sss[m])$ theoretically follows from~\cite[Theorem~1]{bht}, although it is not as simple as for the movable cone.\  Here are a couple of facts:
 \begin{itemize}
\item when $m\ge \frac{e+3}{2}$, the ``other'' extremal ray of $\Nef(S^{[m]})$ is spanned by  $ (m+e)L_m-2e\delta $ (\cite[Proposition 10.3]{bama1}) and the   movable  and  nef cones are  different, except when $m=e+2$;\footnote{If they are equal, we are in one of the following cases described in Example~\ref{exa215} (we set $g:=\gcd({m+e},2e)$)
\begin{itemize}
\item either $(m+e)L_m-2e\delta$ has square 0 and $2e(m+e)^2=4e^2(2m-2)$, which implies  $ (m+e)^2=2e (2m-2)$, hence $ (m-e)^2=-4e$, absurd;
\item or $ m+e =g(m-1)a_1$ and $2e=geb_1$, which implies $b_1\le 2 $ and
$$ a_1^2\, \frac{e+1}{2}\le a_1^2(m-1)=1+eb_1^2+1\le 1+4e+1
 $$
 hence either $a_1=b_1=2$, absurd, or $a_1=1$, and $m+e=g(m-1)$ implies $g=2$, 
 $b_1=1$,  and $m=e+2$, the only case when the cones are equal;
\item or $ m+e =g a'_1$ and $2e=geb'_1$, which implies $g\le 2$ and
$$(m+e)^2 =g^2a_1^{\prime 2}=g^2(e(m-1)b_1^{\prime 2}+1) =4e(m-1)+g^2 
$$
hence $(m-e)^2 = -4e +g^2\le -4e+4\le 0$ and $e=m=1$, absurd.
\end{itemize}
}
\item when $ e=(m-1)b^2$, with $b\ge 2$,    the ``other'' extremal ray of $\Nef(S^{[m]})$ is spanned by  $ L_m-b\delta $ (\cite[Theorem 10.6]{bama1}) and the movable   and   nef cones are equal. 
\end{itemize}
\end{exam}

 \subsubsection{Nef and movable cones for some other  \hK\ fourfolds with Picard number~2}\label{sec27a}
 
 We now give an example of cones with irrational slopes.\  
Let $n$ and $e'$ be a positive   integers such that $n\equiv -1\pmod4$.\ Let $(X,H)$ be a polarized hyperk\"ahler fourfold of $\KKK^{[2]}$-type with $H$ of divisibility 2 and   $\Pic(X)=\Z H\oplus \Z L$, with intersection matrix   $\left(\begin{smallmatrix}2n&0\\0&-2e'\end{smallmatrix}\right)$.\footnote{In the notation of Section~\ref{sec29}, these are fourfolds whose period point is  very general in one component of the hypersurface $\cD_{2n,2e'n}^{(2)}$ and we will prove in Theorem~\ref{imper} that they exist if and only if $n >0$ and $e'>1$.}\   By Theorem~\ref{thm:NefConeHK4}, cones of divisors on $X$ can be described as follows.\footnote{Before  the general results of \cite{bht} were available, the case $n=3$ and $e'=2$ had been worked out in \cite[Proposition 7.2]{hast3}  using   beautiful geometric arguments.}
\begin{itemize}
\item[(a)] The extremal rays  of the  movable cone $\Mov(X)$ are spanned by $H-\mu_{n,e'} L$ and $H+\mu_{n,e'} L$, where
\begin{itemize}
\item[$\bullet$] if the equation $\cP_{n, e'}(-1)$ is not solvable, $\mu_{n,e}=\sqrt{n/e'}$;
\item[$\bullet$] if the equation $\cP_{n, e'}(-1)$ is solvable and $(a_{-1},b_{-1})$ is its minimal solution, $\mu_{n,e'}=\frac{ na_{-1}}{e'b_{-1}}$.
\end{itemize}
\item[(b)] The extremal rays of the nef cone $\Nef(X)$ are spanned by $H-\nu_{n,e'} L$ and $H+\nu_{n,e'} L$, where
\begin{itemize}
\item[$\bullet$] if the equation $\cP_{n, 4e'}(-5)$ is not solvable, $\nu_{n,e'}=\mu_{n,e'}$;
\item[$\bullet$] if the equation $\cP_{n, 4e'}(-5)$ is solvable and $(a_{-5},b_{-5})$ is its minimal solution, $\nu_{n,e'}=\frac{ na_{-5}}{2e'b_{-5}}$.
\end{itemize}
\end{itemize}
%

The new element here is that we may have cones with irrational slopes (when $e'n $ is not a perfect square).\ As in Section~\ref{sec210a}, the walls of the chamber decomposition \eqref{chamb} correspond to the rays determined by the solutions to the equation $\cP_{ n,4e'}(-5)$ that sit inside the movable cone.\ This is only interesting when the nef and movable cones are different, so we assume that the equation  $\cP_{ n,4e'}(-5)$ is   solvable, with 
 minimal solution  $(a_{-5},b_{-5})$, and the extremal rays of the nef cone have  rational slopes $\pm\frac{n a_{-5}}{2e'b_{-5}}$.\ This implies that $ne' $ is not a perfect square\footnote{\label{lab}Since  the equation  $\cP_{n, 4e'}(-5)$ is   solvable,  $d:=\gcd(n,e')\in\{1,5\}$.\ If  $ne'$ is a perfect square, we can write $n=du^2$ and $e'=dv^2$.\ This is easily checked to be incompatible with the equality $ (ua+2vb)(ua-2vb)=-5/d$.} and the equation $\cP_{n, 4e'}(-5)$ has infinitely many solutions.\ By Lemma~\ref{le46}, these solutions form two conjugate classes if $5\nmid e'$, and one class if $5\mid e'$.
 
 This means that all the solutions $(a,b)$ to the equation  $\cP_{ n,4e'}(-5)$  are given by  
\begin{equation}\label{solm}
na +b\sqrt{4ne'}=\pm (na_{-5} \pm b_{-5}\sqrt{4ne'})x_1^m\ , \quad m\in\Z,
\end{equation}
   where $x_1=a_1+ b_1 \sqrt{4ne'} $ corresponds to the minimal solution $(a_1,b_1)$ to the equation $\cP_{ 4ne'}(1)$.\  There are two cases:
\begin{itemize}
\item[a)] either the equation $\cP_{n ,e'}(-1)$ is not  solvable and the extremal rays of the movable cone have  slopes $\pm\sqrt{n/e'}$, which is irrational by footnote \ref{lab};
\item[b)] or the equation $\cP_{n ,e'}(-1)$ is   solvable, with minimal solution $(a_{-1},b_{-1})$, and $x_1=na_{-1}^{2}+e'b_{-1}^2+a_{-1} b_{-1} \sqrt{4ne'} $  by Lemma \ref{lemmpell}.
\end{itemize}


In  case a), all positive solutions $(a,b)$ of the equation  $\cP_{ n,4e'}(-5)$ satisfy 
$$\frac{ na}{2e'b}<\frac{ n}{2e'}\sqrt{\frac{ 4e'}{n}}=\sqrt{\frac{ n}{e'}}.$$
Therefore, there are infinitely many walls in the chamber decomposition \eqref{chamb}: they correspond to the
  (infinitely many) solutions  $(a,b)$ of the equation  $\cP_{ n,4e'}(-5)$ with $a>0$.
  
  In  case (b), the interior walls of the chamber decomposition \eqref{chamb} correspond to the
 solutions  $(a,b)$ of the equation  $\cP_{ n,4e'}(-5)$ (as given in \eqref{solm}) with $a>0$ that land in the 
interior of the  movable cone, that is, that satisfy $\frac{ na }{2e'|b| }< \frac{n a_{-1}}{ e'b_{-1}}$.\ We saw in \eqref{sols4} that there are only two such solutions, $(a_{-5},b_{-5})$ and $(a_{-5},-b_{-5})$, hence three chambers.

\subsection{The Torelli theorem}\label{sec27}

A general version of the Torelli theorem for \hKm s was proven by Verbitsky (\cite{ver}).\ It was later reformulated by Markman in terms of the {\em monodromy group} of a \hKm\ $X$ (\cite{marsur}): briefly, this is the subgroup $\Mon^2(X)$  of
$O(H^2(X,\Z),q_X)$ generated by  monodromy operators   associated with all smooth deformations of  $X$.\ 

When $X$ is of $\kkk[m]$-type, Markman showed that  $\Mon^2(X)$ is the subgroup of the group $O(H^2(X,\Z),q_X)$ generated by reflections about $(-2)$-classes and the negative of reflections about $(+2)$-classes (\cite[Theorem~9.1]{marsur}).\ Combined with a result of Kneser, he obtains that $\Mon^2(X)$ is the subgroup $\widehat O^+(H^2(X,\Z),q_X)$ of   elements of  $O^+(H^2(X,\Z),q_X)$ that act as $\pm\Id$ on the  discriminant group $D(H^2(X,\Z))$ (\cite[Lemma~9.2]{marsur}).\footnote{As in footnote \ref{misl}, $O^+(H^2(X,\Z),q_X)$ is the index-2 subgroup of $O (H^2(X,\Z),q_X)$ of isometries that preserve  the positive cone $\Pos(X)$.}

This implies
 that the index of $\Mon^2(X)$ in  $O^+(H^2(X,\Z),q_X)$ is  $2^{\max\{\rho(m-1)-1,0\}}$, where $\rho(d)$ is the number of distinct prime divisors of an integer $d$;\footnote{One has $D(H^2(X,\Z))\isom \Z/(2m-2)\Z$ and $O(D(H^2(X,\Z)),\bar q_X)\isom \{x\pmod{2m-2}\mid x^2\equiv 1\pmod{4t} \}\isom (\Z/2\Z)^{\rho(m-1)}$ (\cite[Corollary~3.7]{ghs}).\  One then uses the surjectivity of the canonical map $O^+(H^2(X,\Z),q_X)\to O(D(H^2(X,\Z)),\bar q_X) $ (see Section~\ref{seclat}).}
 in particular, these two groups are equal if and only if $m-1$ is a prime power (including $1$).\ For the sake of simplicity, we will only state the first version of the Torelli theorem in that case (\cite[Theorem~1.3]{marsur}); compare with Theorem \ref{torth}).

\begin{theo}[Verbitsky, Markman; Torelli theorem, first version]\label{torthhk}
Let $m$ be a positive integer such that $m-1$ is a prime power.\
Let $(X,H)$ and $(X',H')$ be polarized     \hKm s of $\kkk[m]$-type.\ If there exists an isometry of lattices 
$$\phi\colon (H^2(X',\Z),q_{X'})\isomlra (H^2(X,\Z),q_X)$$
 such that $\phi(H')=H$ and $\phi_\C(H^{2,0}(X'))=H^{2,0}(X)$, there exists an isomorphism $\sigma\colon X\isomto X'$ such that $\phi=\sigma^*$.
\end{theo}
 
We will state a version of this theorem valid for all $m$ in the next section (Theorem~\ref{torthhk2}).\ In this present form, when $m-1$ is not a prime power, the isometry $\phi$ needs to satisfy additional conditions for it to be induced by an isomorphism between $X$ and $X'$.\ For example, when $X=X'$, the isometry $\phi$ needs to be in the group $\widehat O(H^2(X,\Z),q_X)$ defined above.

\subsection{The period map}\label{sec28}
Let  ${}^m\!\!\cM_\tau$ be the  (not necessarily irreducible!) 20-dimensional quasiprojective moduli space of polarized \hKm s of $\kkk[m]$-type with fixed  polarization type $\tau$, that is, $\tau$ is the $O(\Lkkk[m] )$-orbit of a primitive element $h_\tau$ with positive square (see Section~\ref{sec25}).\ Fix such an element $h_\tau\in \Lkkk[m]$ and define as in Section \ref{sec17} a 20-dimensional (nonconnected) complex manifold
\begin{eqnarray*}
\Omega_{h_\tau}&:=&\{[x]\in \P( h_\tau^\bot \otimes\C)\mid x\cdot x=0,\ x\cdot\overline{x}>0\}\\
&=&\{[x]\in \P( \Lkkk[m] \otimes\C)\mid x\cdot h_\tau = x\cdot x=0,\ x\cdot\overline{x}>0\}
.
\end{eqnarray*}
Instead of taking the quotient by the action of the group
$O(\Lkkk[m] ,h_\tau):=\{\phi\in O(\Lkkk[m] )\mid \phi(h_\tau) =h_\tau\}
$ as we did in the K3 surface case,    we consider, according to the discussion in Section~\ref{sec27}, the (sometimes) smaller group
$$\widehat O(\Lkkk[m] ,h_\tau):=\{\phi\in O(\Lkkk[m] )\mid \phi(h_\tau) =h_\tau \textnormal{ and $\phi$ acts   as $\pm\Id$ on } D(\Lkkk[m])\}.
$$
The quotient
$$\cP_\tau:=\widehat O(\Lkkk[m] ,h_\tau)\backslash \Omega_{h_\tau}$$
  is again an irreducible\footnote{As in footnote \ref{misl},  one usually chooses one component $\Omega_{h_\tau}^+$ of  $\Omega_{h_\tau} $, so that  $\cP_\tau =\widehat O^+(\Lkkk[m] ,h_\tau)\backslash \Omega_{h_\tau}^+$ (see \cite[Theorem 3.14]{ghssur}).} quasiprojective normal variety of dimension 20 and one can define an algebraic  period map
\begin{eqnarray*}
\wp_\tau\colon
 {}^m\!\!\cM_\tau\lra\cP_\tau.
\end{eqnarray*}
When the polarization type is determined by its degree $2n$ and its divisibility $\gamma$, we will also write $ {}^m\!\cP^{(\gamma)}_{2n}$ instead of $\cP_\tau$ and $ {}^m\wp^{(\gamma)}_{2n}$ instead of $\wp_\tau$.

 The Torelli theorem now takes the following form (\cite[Theorem~3.14]{ghssur}, \cite[Theorem~8.4]{marsur}).
 
\begin{theo}[Verbitsky, Markman; Torelli theorem, second version]\label{torthhk2}
Let $m$ be an integer with $m\ge 2$ and let $\tau$ be a polarization type of $\Lkkk[m]$.\ The restriction of the period map $\wp_\tau$  to any irreducible component of ${}^m\!\!\cM_\tau$ is an open embedding.
\end{theo}

As we did in Section \ref{sec17} for K3 surfaces (Proposition~\ref{prop19}), we will determine the image of the period map in Section \ref{sec210} (at least when $m=2$).

\begin{rema}
Recall the following facts about the irreducible components of ${}^m\!\!\cM_\tau$:
\begin{itemize}
\item when the divisibility $\gamma$  satisfies $\gamma=2$ or $\gcd(\frac{2n}{\gamma},\frac{2m-2}{\gamma},\gamma)=1$, the polarization type $\tau$ is determined by its square $2n$ and $\gamma$ (in other words, ${}^m\!\!\cM_\tau={}^m\!\!\cM^{(\gamma)}_{2n}$);
\item when  $\gamma=1$, or $\gamma=2$ and $n+m\equiv 1 \pmod4$, the spaces ${}^m\!\!\cM_\tau={}^m\!\!\cM^{(\gamma)}_{2n}$ are irreducible   (Theorem \ref{ghsa}).
\end{itemize} 
\end{rema}

 \begin{rema}[Strange duality]\label{stdu}\hskip-2mm\footnote{This strange duality was already noticed  in \cite[Proposition 3.2]{apot}.}
We   assume here $\gamma=1$, or 
$\gamma=2$ and $n+m\equiv 1 \pmod4$, 
so that the moduli spaces $ {}^m\!\!\cM^{(\gamma)}_{2n}$ are irreducible.\ Recall from Remark~\ref{remsym} the isomorphism of lattices $\L^{(\gamma)}_{\KKK^{[m]},2n}\isom \L^{(\gamma)}_{\KKK^{[n+1]},2m-2}$.\ It translates\footnote{\label{foot21}The compatibility of the group actions were checked by J.~Song; see also \cite{apot}.}
 into an isomorphism
 \begin{equation}\label{ppp}
  {}^m\!\cP^{(\gamma)}_{2n}\isomlra   {}^{n+1}\!\cP^{(\gamma)}_{2m-2} 
  \end{equation} 
hence into a birational isomorphism
\begin{equation}\label{mmm}
 {}^m\!\!\cM^{(\gamma)}_{2n}\isomdra   {}^{n+1}\!\cM^{(\gamma)}_{2m-2}. 
\end{equation}
 So there is a way to associate with a general polarized \hKm\ $(X,H)$ of dimension $2m$, degree $2n$, and divisibility $\gamma\in\{1,2\}$ another polarized \hKm\ $(X',H')$ of dimension $2n+2$, degree $2m-2$, and same divisibility.\ Note that by formula \eqref{chism}, we have $h^0(X,H)=h^0(X',H')$.\ A complete strange duality in the sense of Le Potier would also require a canonical isomorphism 
 \begin{equation}\label{caniso}
 H^0(X,H)\isomlra H^0(X',H')^\vee.
\end{equation}
 \end{rema}

\begin{exam}
There is a very nice interpretation of the birational isomorphism \eqref{mmm} when $m=2 $, $n=3$, and $\gamma=2$.\ Recall from Section~\ref{sec261} that a general element of $ {}^2\!\!\cM^{(2)}_{6}$ is the variety $F(W)\hra  \Gr(2,6)\subset \P^{14}$ of lines contained in a smooth cubic hypersurface $W\subset \P^5$.\ In Section~\ref{sec262}, we explained the construction of an element 
$\xymatrix    @C=25pt
{X(W) \ar@{-->>}[r]^-{72:1} &\Gr(4,6)\subset\P^{14}}$  of $\, {}^4\!\!\cM^{(2)}_{2}$ associated with $W$; this is the strange duality correspondence.\footnote{The fact that the period points of $F(W)$ and  $X(W)$ are identified via the isomorphism~\eqref{mmm} is proved in \cite[Proposition~1.3]{lpz} (see  footnote~\ref{idea}).}\
Geometrically, one can recover $W$ (hence also $F(W)$)  from $X(W)$ as one of the components of the fixed locus of the canonical involution   on $X(W)$  (footnote~\ref{foot24}).\ There is also a canonical rational map
\begin{equation}\label{fwxw}
F(W)\times F(W)\dra X(W) 
\end{equation}
of degree 6 described geometrically in   \cite[Proposition~4.8]{voi} and the canonical involution on $X(W)$ is induced by   the involution of $F(W)\times F(W) $ that interchanges the two factors.

Similarly, a Debarre--Voisin fourfold $X\subset \Gr(6,10) $ (a general element of $\,{}^2\!\!\cM^{(2)}_{22}$; see Section \ref{sec261})  has an associated \hKm\ $(X',H')\in   {}^{12}\!\!\cM^{(2)}_{2}$ of dimension $24$ with $(H')^{24}= 2^{12}\frac{24!}{12!2^{12}}=\frac{24!}{12! }$.\   By analogy with the construction above, could it be that there is a dominant  rational map $\xymatrix    @C=25pt
{X' \ar@{-->>}[r]
&\Gr(4,10)}$?

Finally, a general element $(X,H)$ of $\,{}^2\!\!\cM^{(2)}_{38}$ (geometrically described in Section \ref{sec261}) should have an associated \hKm\ $(X',H')\in   {}^{20}\!\!\cM^{(2)}_{2}$ of dimension $40$, but I have again no idea how to construct it geometrically.
 \end{exam}
 
 \begin{exam}
Let $(S,L)$ be a polarized K3 surface with $\Pic(S)=\Z L$ and $L^2=2e$.\ In \cite[Section~5.3]{og8}, O'Grady considers, for each positive integer $r$,
\begin{itemize}
\item on the one hand, the moduli space $\cM(r,L,r)$, a \hKm\ of type $\KKK^{[m]}$, where $m:=1-r^2+e$, and a line bundle $H$ on $\cM(r,L,r)$, of square 2;
\item on the other hand, the fourfold $S^{[2]}$,  with the line bundle $L_2-r\delta$, of square $2e-2r^2=2m-2$
\end{itemize}
(see Remark~\ref{rema45}).\ We assume $e>r^2$, so that $m>1$.\

By \cite[Corollary~4.15]{og8}, the line bundle $H$ is ample whenever $ e\le r^2+2r-1$.\footnote{Note that $H$ is   movable and nef on $\cM(r,L,r)$ for all $e>r^2$ (see footnotes~\ref{nf} and~\ref{nf2}).}\ The line bundle $L_2-r\delta$ is ample whenever $r$ is less than the slope $\nu_e$ of the nef cone of $S^{[2]}$, which we computed in Example~\ref{exa217}.\ This is the case since $r^2<e$ (Exercice~\ref{lbound}). 

Whenever $r^2< e\le r^2+2r-1$, the pair $( \cM(r,L,r),H)$ therefore defines a point of the moduli space  ${}^m\!\!\cM^{(1)}_2$, the pair $( S^{[2]}, L_2-r\delta)$   defines a point of the moduli space  ${}^2\!\!\cM^{(1)}_{2m-2}$, and it follows from the very construction of these objects that they correspond under the strange duality isomorphism \eqref{mmm} (see \cite[Section~5.3]{og8}).\ Note that one obtains in this fashion infinitely many Heegner divisors in both moduli spaces that correspond under the isomorphism \eqref{mmm}  (just take any $r\ge m/2$ and $e:=m+r^2-1$).

Regarding the existence of the canonical isomorphism \eqref{caniso}, O'Grady constructed in \mbox{\cite[(5.3.8)]{og8}} a canonical linear map
\begin{equation*}\label{sdua}
H^0(\cM(r,L,r), H)^\vee \isomlra H^0(S^{[2]}, L_2-r\delta)
\end{equation*}
and conjectures in 
 \cite[Statement~5.15]{og8} that it is an isomorphism.\ He   gave in \cite[Claim~5.16]{og8} a geometric proof of this conjecture   when $r=2$ and $e\in\{5,6,7\}$.\ This result was vastly extended in \cite[Theorems~1 and~1A]{maop} to an isomorphism
 \begin{equation*} 
H^0(\cM(r,L,s), H)^\vee \isomlra H^0(\cM(r',L,s'), H')
\end{equation*}
where $r$ and $s$ are integers that satisfy the following conditions
\begin{itemize}
\item $r\ge 2$, $r'\ge 2$, $r+s\le 0$, $r'+s'\le 0$;
\item $2e+rs'+r's=0$;
\item $e-rs\ge (r-1)(r^2+1)$, $e-r's'\ge (r'-1)(r^{\prime2}+1)$
\end{itemize}
(the \hKm s $\cM(r,L,s)$ and $\cM(r',L,s')$ are strange duals, of respective dimensions  $2(e-rs+1)$ and $2(e-r's'+1)$). But they do not seem to prove that the (theta)  divisors $H$ and $H'$ are ample.
\end{exam}

\begin{rema}\label{rem220}
There is a chain of subgroups\footnote{\label{foot20}Let $G$ be the subgroup $\Lkkk[m] / (\Z h_\tau\oplus h_\tau^\bot)$ of $D(\Z h_\tau)\times D(h_\tau^\bot)$.\ An element of $O(h_\tau^\bot)$ is in $O(\Lkkk[m] ,h_\tau)$ if and only if it induces the identity on $p_2(G)\subset D(h_\tau^\bot)$.\ This is certainly the case if it is in $\widetilde O(h_\tau^\bot)$ and the lift is then in $\widetilde O(\Lkkk[m] ,h_\tau)$ since it induces the identity on $D(\Z h_\tau)\times D(h_\tau^\bot)$, hence on its subquotient $D(\Lkkk[m])$.} of finite index
$$\widetilde O(h_\tau^\bot)\xhookrightarrow{\ \iota_1\ }\widetilde O(\Lkkk[m] ,h_\tau)\xhookrightarrow{\ \iota_2\ } \widehat O(\Lkkk[m] ,h_\tau)\xhookrightarrow{\ \iota_3\ }   O(\Lkkk[m] ,h_\tau)\xhookrightarrow{\ \iota_4\ } O(h_\tau^\bot)
$$
(the  lattices $h_\tau^\bot$ are described in \cite[Proposition~3.6.(iv)]{ghs}) hence  a further finite  morphism
$$\cP_\tau=\widehat O(\Lkkk[m] ,h_\tau)\backslash \Omega_{h_\tau}\lra O(h_\tau^\bot)\backslash \Omega_{h_\tau}$$
which is sometimes nontrivial.\footnote{\label{foo27}The following holds:
\begin{itemize}
\item the  inclusion $\iota_1$ is an equality if $\gcd(\frac{2n}{\gamma},\frac{2m-2}{\gamma},\gamma)=1$ (\cite[Proposition~3.12(i)]{ghs});
\item the  inclusion $\iota_2$ is an equality if $m=2$, and has index 2 if $m>2$ (\cite[Remark~3.15]{ghssur});
\item the inclusion $\iota_3$ defines a normal subgroup and is an equality if $m-1$ is a prime power (Section~\ref{sec27});
\item the inclusion $\iota_3\iota_2$ defines a normal subgroup and, if $\gcd(\frac{2n}{\gamma},\frac{2m-2}{\gamma},\gamma)=1$, the corresponding quotient  is the group   $(\Z/2\Z)^\alpha$, where  $\alpha= \rho\bigl( \frac{m-1}{\gamma}\bigr)$ when $\gamma$ is odd, and $\alpha= \rho\bigl( \frac{2m-2}{\gamma}\bigr)+\eps$ when $\gamma$ is even, with  $\eps=1$ if $ \frac{2m-2}{\gamma}\equiv 0\pmod8$, $\eps=-1$ if $ \frac{2m-2}{\gamma}\equiv 2\pmod4$, and $\eps=0$ otherwise (\cite[Proposition~3.12(ii)]{ghs}).
\end{itemize}}

Assume $\gamma\in\{1,2\}$.\ When $n=m-1$ (and $n$ even when $\gamma=2$), \eqref{ppp} gives an involution  of $ {}^m\!\cP^{(\gamma)}_{2m-2}$ which corresponds to the element of $O(h_\tau^\bot)$ that switches the two factors of $D(h_\tau^\bot)$.\ This isometry is not in  $O(\Lkkk[m] ,h_\tau)$ (hence the involution of $ {}^m\!\cP^{(\gamma)}_{2m-2}$ is nontrivial), except when $n=\gamma=2$, where the involution is trivial.

Assume now $m=2$.\
The inclusions $\iota_1$, $\iota_2$, and $\iota_3$ are then equalities
and   the group $O(h_\tau^\bot)/\widetilde O(h_\tau^\bot)\isom O(D(h_\tau^\bot))$ acts on $\cP_\tau$ (where  $-\Id$   acts trivially).\ Here are some examples:
\begin{itemize}
\item if $\gamma=1$ and $n$ is odd, we obtain a generically free action of the group $(\Z/2\Z)^{\rho(n)-1}$ when $n\equiv -1\pmod4$, and of the group $(\Z/2\Z)^{\max\{\rho(n),1\}}$ when $n\equiv 1\pmod4$.\footnote{By \eqref{eq1}, we have   $D(h_\tau^\bot)\isom \Z/2n\Z\times \Z/2\Z$, with $q(1,0)=-\frac{1}{2n}$ and $q(0,1)=-\frac12$ in $\Q/2\Z$.\ 
When $n\equiv -1\pmod4$, one checks that any isometry must leave each factor of $D(h_\tau^\bot)$ invariant, so that $O(D(h_\tau^\bot))\isom O(\Z/2n\Z)$.\ When $n\equiv 1\pmod4$, there are extra isometries $(0,1)\mapsto (n,0)$, $(1,0)\mapsto (a,1)$, where $a^2+n\equiv 1\pmod{4n}$.}
\item if $\gamma=2$ and $n\equiv -1\pmod4$, we obtain a generically free action of the group $(\Z/2\Z)^{\rho(n)-1}$.
\end{itemize}
These free actions translate into the existence of nontrivial (birational) involutions on the dense open subset ${}^2\!\!\cM^{(\gamma)}_{2n}$ of ${}^2\!\cP^{(\gamma)}_{2n}$.\ When $\gamma=n=1$, O'Grady gave in \cite{og2} a geometric description of the corresponding involution (general elements of ${}^2\!\!\cM^{(1)}_{2}$ are  double EPW sextics).\ When $\gamma=2$, all cases where we have a geometric description of  general elements of ${}^2\!\!\cM^{(2)}_{2n}$   have $n$ prime (Section~\ref{sec261}), so there are no involutions.

  The quotient of ${}^2\!\cP^{(1)}_{2 }$ by its nontrivial involution is isomorphic to   ${}^3\!\cP^{(2)}_{4}$.\   The authors of \cite{ikkr1}   speculate that given a Lagrangian $A$, the image of the period point of the double EPW sextic $\widetilde Y_A\in {}^2\!\cP^{(1)}_{2 }$  is the period point of the EPW cube associated with $A$ (private discussion).
\end{rema}

\subsection{The Noether--Lefschetz locus}\label{sec29}

Let $(X,H)$ be a polarized \hKm\ of $\kkk[m]$-type with period  $p(X,H)\in \Omega_{h_\tau} $.\ 
As in the case of K3 surfaces, the Picard group of $X$ can be identified with the subgroup of $\Lkkk[m]$ generated by  
$$p(X,H)^\bot \cap \Lkkk[m]
$$
(which contains $h_\tau$).\
This means that if the period of  $(X,H)$ is outside the countable union 
$$\bigcup_{K} \P(K^\bot\otimes\C)\subset \Omega_{h_\tau}$$ 
of hypersurfaces, where $K$ runs over the countable set of primitive, rank-2, signature-$(1,1)$ sublattices of $\Lkkk[m]$ containing $h_\tau$,
 the group $\Pic(X)$ is generated by $H$.\ We let $\cD_{\tau,K} \subset\cP_\tau$ be the image of $  \P(K^\bot\otimes\C)$   and set, for each positive integer $d$,
\begin{equation}\label{unionh}
\cD_{\tau,d}:=\hskip-5mm\bigcup_{K,\ \disc(K^\bot)=-d} \hskip-5mm\cD_{\tau,K}\subset\cP_\tau
\end{equation}
(these are {\em Heegner divisors}; see Section~\ref{sec17}).\
The inverse image in ${}^m\!\!\cM_\tau$  of $\bigcup_d \cD_{\tau,d}$ by the period map is called the {\em Noether--Lefschetz locus.}\ It consists of  (isomorphism classes of) \hKm s of  $\KKK^{[m]}$-type with a polarization of type $\tau$ whose Picard group has rank at least 2.

\begin{lemm}\label{le324}
Fix a polarization type $\tau$ and a positive integer $d$.\ The subset  $\cD_{\tau,d}$  is a finite union of (algebraic) hypersurfaces of $\cP_\tau $.
\end{lemm}

\begin{proof}
Let $K$ be a primitive, rank-2, signature-$(1,1)$ sublattice  of $\Lkkk[m]$ containing $h_\tau$ and let $\kappa$ be a generator of $K\cap h_\tau^\bot$.\  Since $K$ has signature $(1,1)$, we have $\kappa^2<0$ and  the formula   from~\cite[Lemma~7.5]{ghssur}   reads
\begin{equation*}
d=\left|\disc(K^\bot)\right|=\left|\frac{\kappa^2\disc(h_\tau^\bot)}{s^2}\right|   ,
\end{equation*}
where $s $ is the divisibility of $\kappa$ in $ h_\tau^\bot$.\ Since $\kappa$ is primitive in $h_\tau^\bot$, the integer $s$ is also the order of  the element $\kappa_*$ in the discriminant group $D(h_\tau^\bot)$  (Section \ref{seclat}).\ In particular, we have $s\le \disc(h_\tau^\bot)$, hence
$$
|\kappa^2|= ds^2/ \disc(h_\tau^\bot) \le d\disc(h_\tau^\bot).
$$
Since $\tau$ is fixed, so is the (isomorphism class of the) lattice $h_\tau^\bot$, hence $\kappa^2$ can only take finitely many values.\ Since the element $\kappa_*$ of the (finite) discriminant group $D(h_\tau^\bot)$ can only take finitely many values, Eichler's criterion (Section \ref{seclat}) implies that $\kappa$ belongs to finitely many 
$ \widetilde O(h_\tau^\bot) $-orbits, hence to finitely many 
$\widehat O(\Lkkk[m] ,h_\tau) $-orbits.\ Therefore, the images in $\cP_\tau $ of the hypersurfaces $\cD_{\tau,K}$ form finitely many hypersurfaces.
\end{proof}

Describing the irreducible components of the loci $\cD_{\tau,d}$  is a lattice-theoretic question (which we answer   below in the case $m=2$).\ When the divisiblity $\gamma$ is   1 or 2, the polarization type only depends on $\gamma$ and the integer $h_\tau^2=:2n$ (Section \ref{sec25}).\ We use the notation ${}^m\cD^{(\gamma)}_{2n,d}$ instead of 
$\cD_{\tau,d}$ (and $ {}^m\!\cP^{(\gamma)}_{2n}$ instead of $\cP_\tau$) and we let ${}^m\cC^{(\gamma)}_{2n,d}$ be the inverse image in ${}^m\cM^{(\gamma)}_{2n}$ of ${}^m\cD^{(\gamma)}_{2n,d}$  by the period map.

\begin{prop}[Debarre--Macr\`i]\label{propirr}
Let $n$ and $d$ be a positive integers and let $\gamma\in\{1,2\}$.\ If the locus ${}^2\cD_{2n,d}^{(\gamma)}$ is nonempty, the integer $d$ is even; we set $e:=d/2$.

\noindent {\rm (1)\hskip 6mm  (a)} The locus ${}^2\cD_{2n,2e}^{(1)}$ is nonempty if and only if   either $e$ or $e-n$ is a  square modulo~$4n$.
 
 {\rm (b)}  If $n$ is {\em square-free} and $e$ is divisible by $n$ and   satisfies the conditions in {\rm (a)}, 
 the locus ${}^2\cD_{2n,2e}^{(1)}$ is  an irreducible hypersurface, except when 
 \begin{itemize}
\item either $n\equiv 1\pmod4$ and $e\equiv  n\pmod{4n}$,
\item or   $n\equiv -1\pmod4$ and $e\equiv 0\pmod{4n}$,
\end{itemize}
  in which cases ${}^2\cD_{2n,2e}^{(1)}$ has two irreducible components.

{\rm (c)} If  $n$ is {\em prime} and $e$ satisfies the conditions in {\rm (a)}, ${}^2\cD_{2n,2e}^{(1)}$ is  an irreducible hypersurface, except when $n\equiv 1\pmod4$ and $e\equiv  1  \pmod{4}$,  or when $n\equiv -1\pmod4$ and  $e\equiv  0  \pmod{4}$,  in which cases  ${}^2\cD_{2n,2e}^{(1)}$ has two irreducible components.

\noindent {\rm (2)} Assume moreover $ n\equiv -1\pmod4$.

{\rm (a)}   The locus ${}^2\cD_{2n,2e}^{(2)}$ is nonempty if and only if    $e$ is a  square modulo~$n$.

{\rm (b)} If $n$ is {\em square-free and $ n\mid e$,} the locus ${}^2\cD_{2n,2e}^{(2)}$ is  an irreducible hypersurface.

{\rm (c)} If $n$ is {\em prime} and $e$ satisfies the conditions in {\rm (a)}, ${}^2\cD_{2n,2e}^{(2)}$ is  an irreducible hypersurface.
\end{prop}

In cases (1)(b) and (1)(c), when $ {}^2\cD_{2n,d}^{(1)}$ is reducible, its the  two components  are exchanged  by one of the involutions of the period space $\cP_\tau$ described at the end of Section~\ref{sec28} when $n\equiv 1\pmod4$, but not when $n\equiv -1\pmod4$ (in that case, these involutions  are in fact trivial when $n$ is prime).

\begin{proof}[Proof of the proposition]
\noindent{\bf Case $\gamma=1$.}\ 
 Let  $(u,v)$ be a standard basis for a hyperbolic plane $U$ contained in $\L_{\KKK^{[2]}}$ and let~$\ell$ be a basis for the $I_1(-2)$ factor.\ We may take  
$h_\tau:=u+nv
$ (it has the correct square and divisibility), in which case  $h_\tau^\bot=\Z( u-nv)\oplus\Z\ell\oplus M$, where  $M:=\{u,v,\ell\}^\bot=U^{\oplus 2}\oplus E_8(-1)^{\oplus 2}$ is unimodular.\ The   discriminant group  $D(h_\tau^\bot)\isom \Z/2\Z\times \Z/2n\Z$ is generated by $\ell_*=\ell/2$ and $(u-nv)_*=(u-nv)/2n$, with $\bar q(\ell_*)=-1/2$ and  $\bar q((u-nv)_*)=-1/2n$.

Let $\kappa$ be a generator of $K\cap h_\tau^\perp$.\ We write
$$\kappa=a(u-nv)+b\ell +cw,
$$
where $w\in M$ is primitive.\ Since $K$ has signature $(1,1)$, we have $\kappa^2<0$ and  the formula   from~\cite[Lemma~7.5]{ghssur}   reads
\begin{equation}\label{eqd}
d=\left|\disc(K^\bot)\right|=\left|\frac{\kappa^2\disc(h_\tau^\bot)}{s^2}\right|  = \frac{8n(na^2+b^2+rc^2)}{s^2}\equiv \frac{8n(na^2+b^2)}{s^2}\pmod{8n},
\end{equation}
where $r:=-\frac12 w^2$ and $s: =\gcd(2na,2b,c) $ is the divisibility of $\kappa$ in $ h_\tau^\bot$.\ If $s\mid b$, we obtain $d\equiv 2 \bigl( \frac{2na}{s}\bigr)^2\pmod{8n}$, which is the first case of the conclusion: $d$ is even and $e:=d/2$ is a square modulo $4n$.\ Assume  $s\nmid b$ and, for any nonzero integer $x$, write $x=2^{v_2(x)}x_{\textnormal{odd}}$, where $x_{\textnormal{odd}}$ is odd.\ One has then
$\nu_2(s)=\nu_2(b)+1
$ and
$$
d\equiv 2 \Bigl( \frac{2na}{s}\Bigr)^2+2n\Bigl( \frac{b_{\textnormal{odd}}}{s_{\textnormal{odd}}}\Bigr)^2 \equiv 2 \Bigl( \frac{2na}{s}\Bigr)^2+2n   \pmod{8n},
$$
which is the second case of the conclusion: $d$ is even and $d/2 -n$ is a square modulo $4n$.\ It is then easy, taking suitable integers $a$, $b$,~$c$, and vector $w$, to construct examples that show that these necessary conditions on $d$ are also sufficient.

We now prove (b) and (c).\ 
 
Given a lattice $K$ containing $h_\tau$ with $\disc(K^\perp)=-2e$, we let as above $\kappa$ be a generator of $K\cap h_\tau^\perp$.\  By Eichler's criterion (Theorem~\ref{eic}),   the group $\widetilde O(h_\tau^\bot)$ acts transitively on the set of primitive vectors $\kappa\in h_\tau^\bot$ of given square and fixed $\kappa_*\in D(h_\tau^\bot )$.\ Since $\kappa$ and $-\kappa$ give rise to the same lattice $K$ (obtained as the saturation of $\Z h_\tau\oplus \Z \kappa$), the locus $ \cD^{(1)}_{2n,2e} $ will be irreducible (when nonempty) if we show that
  the integer $e$ determines $\kappa^2$, and $\kappa_*$ up to sign.
 
We write as above $\kappa=a(u-nv)+b\ell +cw\in h_\tau^\bot$, with $\gcd(a,b,c)=1$ and    $s =\div_{h_\tau^\bot}(\kappa)=\gcd(2na,2b,c)$.\ From \eqref{eqd}, we get
\begin{equation}\label{kap}
\kappa^2=-es^2/2n=-2(na^2+b^2+rc^2)\quad\textnormal{and}\quad
\kappa_*= (2na/s,2b/s)\in \Z/2n\Z\times \Z/2\Z.
\end{equation}

If  $s=1$, we have $e \equiv 0\pmod{4n}$ and $\kappa_*=0$.
 
 If $s=2$, the integer $c$ is even and $a$ and $b$ cannot be both even (because $\kappa$ is primitive).\ We have $e=n(na^2+b^2+rc^2)$ and
 $$ \begin{cases} 
 e \equiv  n^2 \pmod{4n}\textnormal{\quad and\ }\kappa_*=(n,0) &\textnormal{if $b$ is even (and $a$ is odd);}\\
  e \equiv  n \pmod{4n}\textnormal{\quad and\ }\kappa_*=(0,1) &\textnormal{if $b$ is odd and $a$ is even;}\\
   e \equiv  n(n+1) \pmod{4n}\textnormal{\quad and\ }\kappa_*=(n,1) &\textnormal{if $b$ and $a$ are odd.}
 \end{cases}
 $$ 
 
 Assume now that $n$ is square-free and $ n\mid e$.\  From \eqref{eqd}, we get $n\mid \bigl( \frac{2na}{s}\bigr)^2$, hence $s^2\mid 4na^2$, and 
 $s\mid 2a$ because $n$ is square-free.\ This implies $s=\gcd(2a,2b,c)\in\{1,2\}$.\ 
 
 When $n$ is even (that is, $n\equiv 2\pmod4$),  we see from the discussion above that  both $s$ (hence also $\kappa^2$) and $\kappa_*$ are determined by $e$, so the corresponding hypersurfaces $ \cD^{(1)}_{2n,2e} $ are  irreducible.\ 
 
 If $n$ is odd, there are coincidences: 
 \begin{itemize}
 \item when $n\equiv 1\pmod4$, we have $ n\equiv  n^2 \pmod{4n}$, hence $ \cD^{(1)}_{2n,2e} $ is  irreducible when $e\equiv 0$ or $2n\pmod{4n}$,   has two irreducible components (corresponding to $\kappa_*=(n,0)$ and $\kappa_*=(0,1)$) when $e\equiv  n\pmod{4n}$, and is empty otherwise; 
 \item when $n\equiv -1\pmod4$, we have $ n(n+1)\equiv 0 \pmod{4n}$, hence $ \cD^{(1)}_{2n,2e} $ is  irreducible when $e\equiv  -n$ or $ n\pmod{4n}$,   has two irreducible components (corresponding to $\kappa_*=0$ and $\kappa_*=(n,1)$) when $e\equiv 0\pmod{4n}$, and is empty otherwise.\ 
 \end{itemize}
 This proves (b).
 
  We now assume  that $n$ is prime 
and
prove (c).\ Since $s\mid 2n$, we have $s\in\{1,2,n,2n\}$; the cases $s=1$ and $s=2$ were explained above.\ 
 If $s=n$ (and $n$ is odd), we have $n\mid b$, $n\mid c$, $n\nmid a$, and
  $$  
 e \equiv 4a^2 \pmod{4n}\textnormal{\quad and\quad }\kappa_*=(2a,0). $$ 
  If $s=2n$, the integer $c$ is even, $a$ and $b$ cannot be both even, $n\mid b$, and $n\nmid a$.\ We have
$$ \begin{cases} 
 e \equiv  a^2 \pmod{4n}\textnormal{\quad and\ }\kappa_*=(a,0) &\textnormal{if $2n\mid b$  (hence $a$ is odd);}\\
 e \equiv  a^2+ n \pmod{4n}\textnormal{\quad and\ }\kappa_*=(a,1) &\textnormal{if $b$ is odd (and $n$ is odd);}\\
 e \equiv  a^2+2 \pmod{8}\textnormal{\quad and\ }\kappa_*=(a,1) &\textnormal{if $4\nmid b$ is odd and $n=2$.}
 \end{cases}
 $$  
 
 When $n=2$, one checks that the class of $e$ modulo $8$ (which is in $\{0, 1,2,3, 4,6\}$) completely determines $s$, and $\kappa_*$ up to sign.\ The corresponding divisors $ \cD^{(1)}_{4,2e} $ are therefore all irreducible.
 
When $n\equiv 1\pmod4$, we have $ n\equiv  n^2 \pmod{4n}$ and $ a^2 \equiv  (n-a)^2+ n\pmod{4n}$ when $a$ is odd (in which case $ a^2 \equiv 1\pmod{4}$).\ When  $n\equiv -1\pmod4$, we have $ n(n+1)\equiv 0 \pmod{4n}$ and 
 $ a^2 \equiv  (n-a)^2+ n\pmod{4n}$ when $a$ is even (in which case $ a^2 \equiv 0\pmod{4}$).\ Together with changing $a$ into $-a$ (which does not change the lattice $K$), these are the only coincidences: the corresponding divisors  $ \cD^{(1)}_{2n,2e} $  therefore have two components and the others are irreducible.\ This proves (c).

\smallskip
\noindent{\bf Case $\gamma=2$}  (hence $n\equiv -1\pmod4$).\ We may take 
 $h_\tau:=2\bigl(u+\frac{n+1}{4}v\bigr)+\ell
$, in which case $h_\tau^\bot=\Z w_1\oplus \Z w_2\oplus M$, with    $w_1:= v+\ell$ and $w_2:=-u+\frac{n+1}{4}v$.\ The  intersection form on $\Z w_1\oplus \Z w_2$ has matrix    $\left(\begin{smallmatrix}-2&-1\\-1&-\tfrac{n+1}{2}\end{smallmatrix}\right)$ as in \eqref{eq2} and the discriminant group  $D(h_\tau^\bot)\isom \Z/n\Z$ is generated by $(w_1-2w_2)_*=(w_1-2w_2)/n$, with $\bar q((w_1-2w_2)_*)=-2/n$.

Let $(h_\tau,\kappa')$ be a basis for $K$, so that $\disc(K)=2n\kappa^{\prime 2}-(h_\tau\cdot \kappa')^2$.\ Since  $\div(h_\tau)=\gamma=2$, the integer $ h_\tau\cdot \kappa'$ is even and since $\kappa^{\prime 2}$ is also even (because $\L_{\KKK^{[2]}}$ is an even lattice), we have   $4\mid \disc(K) $ and $-\disc(K)/4$ is a square modulo $n$.\ Since the discriminant of $\L_{\KKK^{[2]}}$ is~2,  the integer $d=|\disc(K^\bot)|$ is either $2\,|\disc(K)|$ or $\tfrac12\,|\disc(K)|$, hence it is even and $e=d/2$ is a square modulo $n$, as desired.

Conversely, it is   easy to construct examples that show that these necessary conditions on $d$ are also sufficient.\ This proves (a).
 
We now prove (b) and (c).\ To prove that the loci $ \cD^{(2)}_{2n,2e} $ are irreducible (when nonempty), we need to show that $ e$ determines $\kappa^2$, and $\kappa_*$ up to sign (where $\kappa$ is a generator of $K\cap h_\tau^\bot$).

With the notation above, we have $\kappa= ((h_\tau\cdot \kappa')h_\tau-2n \kappa')/t$, where $t:=\gcd(h_\tau\cdot \kappa',2n)$ is even and $\kappa^2=\frac{2n}{t^2}\disc(K)$.\ 
Formula~\eqref{eqd} then gives
\begin{equation*}
2e=\left|\disc(K^\bot)\right|=\left|\frac{\kappa^2\disc(h_\tau^\bot)}{\div_{h_\tau^\bot}(\kappa)^2}\right|  = \left|\frac{2n^2\disc(K)}{t^2\div_{h_\tau^\bot}(\kappa)^2 }\right|.
\end{equation*}
Since $n$ is odd and $t$ is even, and, as we saw above, $\disc(K)\in\{-e,-4e\}$, the only possibility is $\disc(K)=-4e$ and $t\div_{h_\tau^\bot}(\kappa) =2n$.\ 

Assume that $n$ is square-free and $ n\mid e$.\ Since $-4e=\disc(K)=2n\kappa^{\prime 2}-(h_\tau\cdot \kappa')^2$, we get $2n\mid (h_\tau\cdot \kappa')^2$ hence, since $n$ is square-free and odd, $2n\mid h_\tau\cdot \kappa'$.\ This implies $t=2n$ and $\div_{h_\tau^\bot}(\kappa)=1$; in particular,   $\kappa_*=0$ and $\kappa^2=-2e/n$ are uniquely determined.\ This proves~(b).

We now assume that $n$ is prime.\
 Since $t\div_{h_\tau^\bot}(\kappa) =2n$ and $t$ is even, 
 \begin{itemize}
 \item either 
$(t,\div_{h_\tau^\bot}(\kappa),\kappa^2)=(2n,1,-2e/n)$ and  $n\mid  e$; 
\item or
$(t,\div_{h_\tau^\bot}(\kappa),\kappa^2)=(2,n,-2ne)$ and $n\nmid e$ (because $n\nmid h_\tau\cdot \kappa'$ and $d=-\frac12 \disc(K)\equiv \frac12(h_\tau\cdot \kappa')^2\pmod{n}$).
 \end{itemize}
Given $e= a^2+ nn'$, the integer $\kappa^2$ is therefore uniquely determined by $e $: 
\begin{itemize}
\item either $  n\mid a$,  $\kappa^2= -2e/n$, and $\kappa_*=0$;
\item or $n\nmid a$, $\kappa^2= -2ne$, $\kappa_*=\kappa/n$, and $\bar q(\kappa_*)= -2a^2/n\pmod{2\Z}$.
\end{itemize}
In the second case, $\kappa_*=\pm a(w_1-2w_2)_* $; it follows that in all cases, $\kappa_*$ is also uniquely defined, up to sign, by $e$.\ This proves (c).
\end{proof}

\begin{exam}[Case $n=\gamma=1$]\label{exa221}
The  moduli space ${}^2\!\!\cM^{(1)}_{2}$ is irreducible and contains a dense open subset $\cU_2^{(1)}$ whose 
  points correspond  to double EPW sextics (Section~\ref{sec261}).\ The  complement ${}^2\!\!\cM^{(1)}_{2}\smallsetminus \cU_2^{(1)}$ contains the irreducible hypersurface $\cH_2^{(1)}$ whose general points correspond to     pairs $(S^{[2]}, L_2-\delta)$, where~$(S,L)$ is a polarized K3 surface of degree 4 (Example~\ref{ex19}).

The loci ${}^2\cC^{(1)}_{2,d}$ (for $d>0$) were  studied  in \cite{dims}.\ Assuming $d\equiv 0, 2,$ or $4\pmod8$ (which is, according to Proposition \ref{propirr}, the condition for ${}^2\cD^{(1)}_{2,d}$ to be nonempty), they were shown to be nonempty if and only if  $d\notin\{2,8\}$ (see Remark \ref{rema227} for a more precise statement).\ The locus $  {}^2\cC_{2,4}^{(1)}$ is the hypersurface $\cH_2^{(1)}$.

There is a way to associate  a double EPW sextic with any smooth {\em Gushel--Mukai fourfold}  (these are by definition complete intersections of $\Gr(2,V_5)\subset \P(\bw2V_5)=\P^9$ with a hyperplane and a quadric hypersurface; see \cite{DK}).\ Double EPW sextics with small discriminant $d$
  correspond  to Gushel--Mukai fourfolds  with special geometric features (\cite{dims}).
\end{exam}

\begin{exam}[Case $n=3$ and $\gamma=2$]\label{exa222}
The  moduli space ${}^2\!\!\cM^{(2)}_{6}
$ is irreducible and contains a dense open subset $\cU_6^{(2)}$ whose 
  points  correspond  to the varieties of lines $X:=F(W)$ contained in a  cubic fourfold $W\subset \P^5$ (Section~\ref{sec261}).\ The  complement ${}^2\!\!\cM^{(2)}_{6}\smallsetminus \cU_6^{(2)}$ contains the irreducible hypersurface $\cH_6^{(2)}$ whose general points correspond to     pairs $(S^{[2]}, 2 L_2-\delta)$, where~$(S,L)$ is a polarized K3 surface of degree 2.

  The loci ${}^2\cC^{(2)}_{6,d}$ (for $d>0$) were originally introduced and studied by Hassett in \cite{has}.\ 
Assuming $d\equiv 0$ or $2\pmod6$ (which, according to Proposition \ref{propirr}, is the condition for ${}^2\cD^{(2)}_{6,d}$ to be nonempty), he showed that they are nonempty if and only if $d\ne 6$ (see also Remark \ref{rema227}).\ The locus $  {}^2\cC_{6,2}^{(2)}$ is the hypersurface $\cH_6^{(2)}$ (see \cite{vdd} for a very nice description of the geometry involved).

Small $d$ correspond  to cubics $W$ with special geometric features:   $W$ contains a plane ($d=8 $), $W$ contains a cubic scroll ($d=12$), $W$ is a Pfaffian cubic ($d=14$), $W$ contains a Veronese surface ($d=20$).
 \end{exam}

\subsection{The image of the period map}\label{sec210e}

In this section, we study the image of the period map
\begin{eqnarray*}
\wp_\tau\colon
 {}^m\!\!\cM_\tau\lra\cP_\tau 
\end{eqnarray*}
defined in Section~\ref{sec28} (here $\tau$ is a polarization type for polarized \hKm s of $\kkk[m]$-type).\ For K3 surfaces ($m=1$), we saw in Section~\ref{sec17} that this image is precisely the complement of the (images in $\cP_\tau$ of the)  hypersurfaces $y^\bot$, for all $y\in h_\tau^\bot$ with $y^2=-2$.\ 

It follows from the results of  \cite{av} that for all $m\ge 2$ (which we assume for now on), the image of the period map is the complement of the union of   finitely many Heegner divisors (that is,   of the type studied in Section~\ref{sec29}).\ We will explain how to determine explicitly (in principle) these divisors.\ We will only do here the case $m=2$; the general case will be treated in Appendix~\ref{imagep}.\

  \begin{theo}\label{imper}
Let $n$ be a positive integer and let $\gamma\in\{1,2\}$.\ The   image of the period map
  \begin{equation*} 
 {}^2\wp^{(\gamma)}_{2n}\colon  {}^2\!\!\cM_{2n}^{(\gamma)}\lra  O(\Lambda_{\KKK^{[2]}},h_\tau)\backslash \Omega_{\tau}
 \end{equation*}
  is exactly the complement of the union of finitely many explicit Heegner divisors.\ More precisely, these Heegner divisors are
\begin{itemize}
\item if $\gamma=1$, 
\begin{itemize}
\item  some irreducible components of the hypersurface ${}^2\cD^{(1)}_{2n,2n}$ (two components if $n\equiv 0$ or $1\pmod4$, one component if $n\equiv 2$ or $3\pmod4$);
\item one irreducible component of the hypersurface ${}^2\cD^{(1)}_{2n,8n}$;
\item one irreducible component of the hypersurface ${}^2\cD^{(1)}_{2n,10n}$;
\item and, if $ n \equiv  \pm 5\pmod{25}$, some irreducible components of the hypersurface $\cD^{(1)}_{2n,2n/5}$;
\end{itemize}
\item if $\gamma=2$ (and $n\equiv -1\pmod4$), one irreducible component of the hypersurface ${}^2\cD^{(2)}_{2n,2n}$.
\end{itemize}
\end{theo}

\begin{rema}\label{rema227}
Assume that $n$ is square-free (so in particular $n\not\equiv 0\pmod4$).\ We proved in Proposition~\ref{propirr} that   
\begin{itemize}
\item the hypersurface $\cD^{(1)}_{2n,2n}$ has two components if $n\equiv  1\pmod4$, one component otherwise;
\item the hypersurface $\cD^{(1)}_{2n,8n}$ has two components if $n\equiv  -1\pmod4$, one component otherwise;
\item the hypersurface $\cD^{(1)}_{2n,10n}$ has two components if $n\equiv   1\pmod4$, one component otherwise;
\item the hypersurface $\cD^{(1)}_{2n,2n/5}$ has two components if $n\equiv   1\pmod4$, one component otherwise;
\item the hypersurface $\cD^{(2)}_{2n,2n}$ is irreducible (when $n\equiv   -1\pmod4$).
\end{itemize}
Furthermore, it follows from the proof of the theorem that when moreover $ n \equiv  \pm 5\pmod{25}$,
\begin{itemize}
\item the hypersurface $\cD^{(1)}_{2n,2n/5}$ is irreducible.
\end{itemize}
\end{rema}

\begin{proof}[Proof of Theorem~\ref{imper}]
Take a  point $x\in\! {}^2\cP_{2n}^{(\gamma)}$.\ Since the period map for smooth compact (not necessarily projective) hyperk\"ahler fourfolds is surjective (\cite[Theorem~8.1]{huyinv}), there exists a compact hyperk\"ahler fourfold $X'$ with the given period point $x$.\  Since the class $h_\tau$ is algebraic and has positive square, $X'$ is projective by~\cite[Theorem~3.11]{huyinv}.\ Moreover, the class $h_\tau$ corresponds to the class of an integral divisor $H$ in the positive cone of $X'$.\ By Remark~\ref{rmk:NefConeHK4}(a), we can let an element in the group $W_{\Exc}$ act and assume that the pair $(X',H)$, representing the period point $x$ and the class $h_\tau$, is such that $H$ is in $\overline{\Mov(X')}\cap\Pos(X')$.\  By Remark~\ref{rmk:NefConeHK4}(b), we can find a projective hyperk\"ahler fourfold $X$  which is birational to $X'$ (and so still has period $x$), such that the divisor $H$, with class $h_\tau$, is nef and big on $X$ and has divisibility $\gamma$.\  By \cite[Theorem~4.6]{huyinv}, the fourfold $X'$ is deformation equivalent to $X$, hence still of type $\kkk[2]$.

To summarize, the point $x$ is in the image of the period map ${}^2\wp^{(\gamma)}_{2n}$ if and only if $H$ is actually ample on $X$.\ We now  use Theorem~\ref{thm:NefConeHK4}: $H$ is ample if and only if it is not orthogonal to any algebraic class either with square $-2$, or with square $-10$ and divisibility~2. 

{\em If $H$ is orthogonal to an algebraic class $v$ with square $-2$,} the Picard group of $X$ contains a rank-2 lattice $K$ with intersection matrix  $\bigl(\begin{smallmatrix}2n&0\\0&-2\end{smallmatrix}\bigr)$; the fourfold $X$ is therefore special of discriminant $2e:=-\disc(K^\perp)$ (its period point is in the hypersurface ${}^2\cD^{(\gamma)}_{2n,K} $).\ In the notation of the proof of Proposition~\ref{propirr}, $v$ is the class~$\kappa$.
 
If $\gamma=1$,   the   divisibility $s:=\div_{K^\perp}(\kappa)$ is either 1 or 2.\ By \eqref{kap}, we have   $es^2=-2n\kappa^2=4n$, hence
\begin{itemize}
\item either $s=1$, $e=4n$, and $\kappa_*=0$: the period point is then in one irreducible component of the hypersurface ${}^2\cD^{(1)}_{2n,8n}$;
\item or $s=2$, $e= n$, and 
\begin{itemize}
\item either $\kappa_*=(0,1)$;
\item or $\kappa_*=(n,0)$ and $n\equiv 1\pmod4$;
\item or $\kappa_*=(n,1)$ and $n\equiv 0\pmod4$.
\end{itemize}
\end{itemize}
The period point $x$ is  in one irreducible component of the hypersurface ${}^2\cD^{(1)}_{2n,2n}$ if $n\equiv 2$ or $3\pmod4$, or in the union of two such components otherwise.

If $\gamma=2$, we have $e=-\disc(K)/4=n$, $t =\sqrt{2n\disc(K)/\kappa^2}=2n $, and $\div(\kappa)=2n/t=1$, hence $\kappa_*=0$: the period point $x$ is in one irreducible component of the hypersurface ${}^2\cD^{(2)}_{2n,2n}$.

{\em If $H$ is orthogonal to an algebraic class with square $-10$ and divisibility 2,} the Picard group of $X$ contains a rank-2 lattice $K$ with intersection matrix  $\bigl(\begin{smallmatrix}2n&0\\0&-10\end{smallmatrix}\bigr)$, hence $X$ is   special of discriminant $2e:=-\disc(K^\perp)$.\ Again, we distinguish two cases, keeping the same notation.

If $\gamma=1$,   the   divisibility $s:=\div_{h_\tau^\perp}(\kappa)$ is even (because the divisibility in $H^2(X,\Z)$ is~2) and divides $\kappa^2=-10$, hence it is
 either  2 or 10.\ Moreover, $es^2=-2n\kappa^2=20n$, hence
\begin{itemize}
\item either  $s=2$, $e= 5n$, and $\kappa_*=(0,1)$: the period point $x$ is then in one irreducible component of the hypersurface ${}^2\cD^{(1)}_{2n,10n}$;
\item or $s=10$ and $e= n/5$: the period point is then in   the hypersurface ${}^2\cD^{(1)}_{2n,n/5}$.
\end{itemize}
In the second case, since the divisibility of $\kappa$ in $H^2(X,\Z)$ is 2,
 $a$ and $c$ are even, so that $b$ is odd and $\kappa_*=(a,1)$.\ We have $10=s=\gcd(2na,2b,2c)$, hence $b$ and $c$ are divisible by $5$, but not $a$, because $\gcd( a, b, c)=1$.\ We have $e \equiv  a^2+ n \pmod{4n}$, hence $ e\equiv    a^2\equiv  \pm1\pmod{5}$.\

In general, there are many possibilities for $a=2a'$, with $a^{\prime 2}\equiv e\pmod {5e}$.\ However, if $n$ is square-free, $e$ divides $a'$ and   $  (a'/e)^2\equiv 1 \pmod{5}$, so that $a \equiv  \pm 2e\pmod{2n}$.\ It follows that
 $\pm a$ (hence also $\pm \kappa_*$) is well determined (modulo $2n$), so we have a single component of $\cD^{(1)}_{2n,n/5}$.

If $\gamma=2$, we have $e=-\disc(K)/4=5n$ and $t^2 = {2n\disc(K)/\kappa^2}=n^2/10 $, which is impossible.

Conversely, in each case  described above, it is easy to construct a class $\kappa$ with the required square and divisibility which is orthogonal to $H$.
\end{proof}
 
 \begin{exam}[Case $n=\gamma=1$]\label{exa221z}
We keep the notation of Example \ref{exa221}.\  O'Grady proved that the image of $\cU_2^{(1)}$ in the period space does not meet  $\cD_{2,2}^{(1)} $, $\cD_{2,4}^{(1)} $,  $\cD_{2,8}^{(1)} $, and one component of $\cD_{2,10}^{(1)} $ (\cite[Theorem~1.3]{og6}\footnote{O'Grady's hypersurfaces $\mathbb{S}'_2\cup \mathbb{S}''_2$, $\mathbb{S}^\star_2$, $\mathbb{S}_4$ are our $\cD_{2,2}^{(1)} $, $\cD_{2,4}^{(1)} $,  $\cD_{2,8}^{(1)} $.}); moreover, by~\cite[Theorem~8.1]{dims},  this image does meet   all the other components of the  nonempty hypersurfaces  $\cD_{2,d}^{(1)} $.\  The hypersurface $\cH_2^{(1)}$ maps to $\cD_{2,4}^{(1)} $.\ 

These results agree with Theorem~\ref{imper} and  Remark~\ref{rema227}, which say that the image of $\cM_2^{(1)}$ in the period space is the complement of the union of the two components of $\cD^{(1)}_{2 ,2 }$, the irreducible $\cD^{(1)}_{2 ,8}$, and one of the two components of $\cD^{(1)}_{2 ,10}$.\ However, our theorem says nothing about the image of 
 $\cU_2^{(1)}$.\ O'Grady conjectures that it is the complement of the   hypersurfaces $\cD_{2,2}^{(1)} $, $\cD_{2,4}^{(1)} $,  $\cD_{2,8}^{(1)} $, and one component of $\cD_{2,10}^{(1)} $; this would follow if one could prove $\cM_2^{(1)}=\cU_2^{(1)}\cup \cH_2^{(1)}$.
\end{exam}

\begin{exam}[Case $n=3$ and $\gamma=2$]\label{exa222z}
We keep the notation of Example \ref{exa222}.\  Theorem~\ref{imper} and Remark~\ref{rema227} say that the image of $\cM_6^{(2)}$ in the period space is the complement of the irreducible hypersurface $\cD^{(2)}_{6 ,6}$.\ This (and much more) was first proved by Laza in~\cite[Theorem~1.1]{laza}, together with the fact that $\cM_6^{(2)}=\cU_6^{(2)}\cup \cH_6^{(2)}$; since $\cH_6^{(2)} $ maps onto $\cD^{(2)}_{6 ,2}$,  the image of $\cU_6^{(2)}$  is the complement of $\cD^{(2)}_{6 ,2}\cup \cD^{(2)}_{6 ,6}$.
 \end{exam}
 
 \begin{exam}[Case $n=11$ and $\gamma=2$]\label{exa222zz}
 Theorem~\ref{imper} and Remark~\ref{rema227} say that the image of $\cM_{22}^{(2)}$ in the period space is the complement of the irreducible hypersurface $\cD^{(2)}_{22,22}$.\ In particular, the loci $  {}^2\cC_{22,2e}^{(2)}$ are irreducible hypersurfaces in the moduli space $  {}^2\cM_{22}^{(2)}$ for all positive integers $e$ which are squares modulo 11 but different from 11.\ Let us try to identify the general elements  of the hypersurfaces $  {}^2\cC_{22,2e}^{(2)}$ for small  $e$  (we will do that more systematically in  Section~\ref{sect5}).
 
 Let $(S,L)$ be a polarized K3 surface of degree 2e.\ A class of divisibility 2 on $\sss[2]$ can be written as $2bL_2-a\delta$ and it has square 11 if and only if $a^2-4eb^2=-11$.\ The lattice $K=\Z L\oplus \Z\delta$ is contained in $\Pic(\sss[2])$ and the discriminant of $K^\perp$ is $-d$ (see Section~\ref{sect5}).

If $e=1$, we find $b=3$ and $a=5$, but the class $6L_2-5\delta$ is not nef on $\sss[2]$, only movable (Example~\ref{exa217}).\ It is however ample on the Mukai flop $X$ of  $\sss[2]$ (see Exercise~\ref{32}).\ 
A  general point of the hypersurface $  {}^2\cC_{22,2}^{(2)}$ corresponds to   a pair$(X,6L_2-5\delta)$.

If $e=3$, we find $b=a=1$ and the class $2L_2-\delta$ is ample on $\sss[2]$ (Example~\ref{exa217}).\ 
A  general point of the hypersurface $  {}^2\cC_{22,6}^{(2)}$ corresponds to   a pair $(\sss[2],2L_2-\delta)$.

If $e=4$, the equation $a^2-16b^2=-11$ has no solutions, so a general point of the hypersurface $  {}^2\cC_{22,8}^{(2)}$ is not birational to the Hilbert square of a K3 surface.

If $e=5$, we find $a=3$ and $b=1$, and the class $2L_2-3\delta$ is ample on $\sss[2]$ (Example~\ref{exa217}).\ A  general point of the hypersurface $  {}^2\cC_{22,10}^{(2)}$ corresponds to   a pair $(\sss[2],2L_2-3\delta)$.

If $e=9$, we find $a=5$ and $b=1$, and the class $2L_2-5\delta$ is ample on $\sss[2]$ (Example~\ref{exa217}).\ A  general point of the hypersurface $  {}^2\cC_{22,18}^{(2)}$ corresponds to   a pair $(\sss[2],2L_2-5\delta)$.

Finally, we noted in Example~\ref{exa217} that nodal degenerations of elements of $\cM_{22}^{(2)}$ are birationally isomorphic to Hibert squares of polarized K3 surfaces of degree 22.\ One can extend the period map to these nodal degenerations and they then dominate the irreducible hypersurface $\cD^{(2)}_{22,22}$. 
 \end{exam}

\section{Automorphisms of \hKm s}\label{sect4}

We determine  the group $\Aut(X)$ of biregular automorphisms  and the group $\Bir(X)$ of birational automorphisms  for some   \hKm s $X$ of $\KKK^{[m]}$-type with Picard number 1 or 2.

\subsection{The orthogonal representations of the automorphism groups}\label{secpsi}

Let $X$ be a \hKm.\ As in the case of K3 surfaces, we have   $H^0(X,T_X)\isom H^0(X,\Omega^1_X)=0
$ and  the group $\Aut(X)$ of biholomorphic automorphisms of $X$ is discrete.\ We introduce the two representations
\begin{equation}\label{reppica}
\Psi_X^A\colon\Aut(X)\lra O( H^2(X,\Z),q_X)\quad\textnormal{and}\quad  \overline\Psi_X^A\colon\Aut(X)\lra O( \Pic(X),q_X).
\end{equation}
By Proposition~\ref{prop110b}, the group $\Bir(X)$ of birational automorphisms of $X$ also acts on $O( H^2(X,\Z),q_X)$.\ We may therefore define two other representations
\begin{equation}\label{reppicb}
\Psi_X^B\colon\Bir(X)\lra O( H^2(X,\Z),q_X)\quad\textnormal{and}\quad  \overline\Psi_X^B\colon\Bir(X)\lra O( \Pic(X),q_X).
\end{equation}
We have of course $\Ker( \Psi_X^A) \subset \Ker(\overline\Psi_X^A )$ and $\Ker( \Psi_X^B) \subset \Ker(\overline\Psi_X^B )$.

\begin{prop}[Huybrechts, Beauville]\label{prop111}
Let $X$ be a \hKm.\ The kernels of $\Psi_X^A$ and $\Psi_X^B$ are equal and finite and, if $X$ is projective,  the kernels  of $\overline\Psi_X^B$ and $\overline\Psi_X^A$ are equal and finite.

If $X$ is   of $\kkk[m]$ type, $\Psi_X^A$ and $\Psi_X^B$ are injective.
\end{prop}

\begin{proof}[Sketch of proof]
If $X$ is projective and $H$ is an ample line bundle on $X$, and if $\sigma$ is a birational automorphism of $X$ that acts trivially on $\Pic(X)$, we have $\sigma^*(H)=H$ and this implies that $\sigma$ is an automorphism (Proposition \ref{prop110b}).\  We have therefore 
$\Ker( \overline\Psi_X^B) = \Ker(\overline\Psi_X^A )$ (and $\Ker( \Psi_X^B) = \Ker(\Psi_X^A )$) and the latter group is a discrete subgroup of a general linear group hence is finite.

For the  proof in the  general case (when $X$ is only assumed to be K\"ahler), we refer to \cite[Proposition~9.1]{huyinv} (one replaces $H$ with a K\"ahler class).

When $X$ is a punctual Hilbert scheme of a K3 surface,  Beauville proved that $\Psi^A_X$ is injective (\cite[Proposition 10]{beac10}).\ It was then proved in \cite[Theorem~2.1]{hast4} that the kernel of $\Psi_X^A$ is invariant by smooth deformations.\end{proof}

   Via the representation $\overline\Psi_X^A$, any automorphism of $X$ preserves the cone $\Nef(X)$ and via the representation $\overline\Psi_X^B$, any birational automorphism of $X$ preserves the cone $\Mov(X)$.\    The Torelli theorem (Theorem~\ref{torthhk}) implies that any element of 
   $\widehat O(H^2(X,\Z),q_X)$ which is a Hodge isometry  and preserves an ample class is induced by an automorphism of $X$. 
   
 \begin{rema}
 Let $X$ be a \hKm.\
Oguiso proved in \cite{ogu4} and \cite{ogu5}  that  when $X$ is not projective, both groups $\Aut(X)$ and $\Bir(X)$ are almost abelian  finitely generated, hence finitely presented.\ When 
 $X$ is projective, the groups $\Bir(X)$ and $\Aut(X)$ are also  finitely presented (\cite{bs}, \cite[Theorem~1.6]{cf}).
  \end{rema}
   
   \subsection{Automorphisms of very general polarized \hKm s}\label{sec42}
   
   The Torelli theorem allows us to ``read'' biregular automorphisms of a \hKm\ on its second cohomology lattice.\ For birational automorphisms, this is more complicated but there are still necessary conditions: a birational automorphism must preserve the movable cone.\ The next proposition describes the automorphism groups, both biregular and birational, for a very general polarized \hKm\ of $\KKK^{[m]}$-type; at best, we get groups of order~2 (as in the case of polarized K3 surfaces; see Proposition~\ref{autk3}).

\begin{prop}\label{prop27}
Let $(X,H)$ be a polarized \hKm\ corresponding to a very general point in any component of  a moduli space ${}^m\!\!\cM^{(\gamma)}_{2n}$ (with $m\ge 2$).\ The group $\Bir(X)$ of birational automorphisms of $X$ is trivial, unless 
\begin{itemize}
 \item   $n=1$ (so that $\gamma\in\{1,2\}$);
 \item   $n=m-1= \gamma$ and $-1$ is a square modulo $m-1$.
\end{itemize}
In each of these cases,  $\Aut(X)=\Bir(X)\isom\Z/2\Z$ and the corresponding involution of $X$ is antisymplectic.
\end{prop}

\begin{proof} 
As we saw in Section~\ref{sec29}, the Picard group of $X$ is generated by the class $H$.\ Any birational automorphism leaves this class fixed, hence is in particular biregular of finite order.\  Let $\sigma$ be a nontrivial automorphism of $X$.\ Since $\sigma$ extends to small deformations of $(X,H)$, the  restriction of $\sigma^*$ to $H^\bot$ is a homothety\footnote{
The argument is classical:  let $X$ be a \hKm, let $\omega$ be a symplectic form on $X$, let $\sigma$ be an automorphism of $X$, and write $\sigma^*\omega=\xi \omega$, where $\xi\in\C^\star$.\ Assume that $(X,\sigma)$ deforms along a subvariety of the moduli space; the image of this subvariety by the period map consists of period points which are eigenvectors for the action of $\sigma^*$   on $H^2(X,\C)$ and the eigenvalue is necessarily $\xi$.\ In our case, the span of the image by the period map is $H^\bot$, which is therefore contained in the eigenspace $H^2(X,\C)_\xi$.} whose ratio is, by~\cite[Proposition 7]{bea}, a root of unity; since it is real and nontrivial (because $\Psi_X^A$ is injective by Proposition \ref{prop111}), it must be $- 1$.\ We will study under which conditions  such an isometry of $\Z H\oplus H^\bot$  extends to an isometry $\phi\in \widehat O(H^2(X,\Z),q_X)$.

Choose an identification $H^2(X,\Z)\isom \Lkkk[m]$ and write, as in Section \ref{sec25}, $H=ax+b\delta$, where $a,b\in\Z$ are relatively prime, $x$ is primitive in $\L_{\KKK}$,  and $\gamma:=\div(H)=\gcd(a,2m-2)$.\ Using Eichler's criterion, we may write $x=u+cv$, where $(u,v)$ is a standard basis for a hyperbolic plane in $\L_{\KKK}$ and $c=\frac12 x^2\in\Z$.\ We then have $\phi(a(u+cv)+b\delta)=a(u+cv)+b\delta$ and
\begin{eqnarray}
\phi(u-cv)&=&-u+cv,\label{eqq1}\\
\phi(b(2m-2)v+a\delta)&=&-b(2m-2)v-a\delta\label{eqq2}
\end{eqnarray}
(these two vectors are in $H^\bot$).\ From these identities, we obtain
\begin{equation}\label{phid}
(2a^2c-b^2(2m-2))\phi(v)=2a^2u+b^2(2m-2)v+2ab\delta.
\end{equation}
This implies that the number $2a^2c-b^2(2m-2)$, which is equal to $H^2=2n$, divides $\gcd(2a^2,b^2(2m-2)v,2ab)=2\gcd(a,m-1)$.\ Since $H^2$ is obviously divisible by $2\gcd(a,m-1)$, we obtain $n=\gcd(a,m-1)\mid \gamma$ hence, since $\gamma\mid 2n$,
$$n\mid m-1 \qquad\textnormal{and}\qquad \gamma\in\{n,2n\}.
$$
One checks that conversely, if these  conditions are satisfied, one can define an involution $\phi$ of $\Lkkk[m] $ that has the desired properties 
$\phi(H)=H$ and $\phi\vert_{H^\bot}=-\Id_{H^\bot}$: use \eqref{phid} to define $\phi(v)$, then \eqref{eqq1} to define $\phi(u)$, and finally \eqref{eqq2} to define 
$$\phi(\delta)=-\frac{ab(2m-2)}{n}u-\frac{abc(2m-2)}{n}v-\frac{b^2(2m-2)+n}{n}\delta
.$$
The involution $\phi$ therefore acts on $D(H^2(X,\Z))=\Z/(2m-2)\Z=\langle \delta_*\rangle$ by multiplication by $-\frac{b^2(2m-2)+n}{n}$.\ For $\phi$ to be in $\widehat O(H^2(X,\Z),q_X)$, we need to have 
\begin{itemize}
\item either $-\frac{b^2(2m-2)+n}{n}\equiv -1\pmod{2m-2}$, \ie, $n\mid b^2$; since $n\mid a$ and $\gcd(a,b)=1$, this is possible only when $n=1$;
\item or $-\frac{b^2(2m-2)+n}{n}\equiv 1\pmod{2m-2}$, which implies $m-1\mid n$, hence $n=m-1$.
\end{itemize}
When the conditions 
\begin{equation}\label{cond}
n\in\{1,m-1\} \qquad\textnormal{and}\qquad  \gamma\in\{n,2n\}
\end{equation}
are realized, $\phi$ acts on $D(H^2(X,\Z))$ as $-\Id$ in the first case and as $\Id$ in the second case (the relations $m-1=n=a^2c-b^2(m-1)$ and $n\mid a$ imply $b^2\equiv-1\pmod{m-1} $).\ By the Torelli theorem, $\phi$ lifts to an involution of $X$.

Finally, we use the numerical conditions on $m$, $\gamma$, and $n$ given in  \cite[Proposition~3.1]{apo} to show that when  \eqref{cond} holds, the moduli space ${}^m\!\!\cM^{(\gamma)}_{2n} $ is empty except in the cases given in the proposition.
\end{proof}

 \begin{rema}\label{remar44}
 The degeneration argument used in \cite[Section~2]{og8}, and Proposition~\ref{prop27}, imply that   any \hKm\ of type $\KKK^{[m]}$ (where $m\ge2$) with  a line bundle $H$ with square 2 carries a birational involution $\sigma$ that acts in cohomology as the symmetry about the line spanned by $H$.\ It is in particular antisymplectic, and it is biregular if and only if either $H$ or $H^{-1}$ is ample.\footnote{\label{nf}If $H^{\pm 1}$ is ample, the relation   $\sigma^*H^{\pm 1}=H^{\pm 1}$ implies that $\sigma$ is biregular.\ Conversely, if $\sigma$ is biregular, consider an ample line bundle $A$ on $X$.\ Then $A\otimes \sigma^*A$ is ample and is proportional to $H$,  hence either $H$ or $H^{-1}$ is ample.\ This reasoning shows that either $H$ or $H^{-1}$ is always movable.}
\end{rema}

\begin{rema}\label{rema45}
When $m=2$ and $n=1$ (and $\gamma=1$), the involution $\sigma$ of $X$ is the canonical involution on double EPW sextics  and the morphism $\phi_H\colon X\to \P^5$ factors as $X\to X/\sigma\hra  \P^5$  (see Section~\ref{sec261}).\ 

When $m=3$, $n=2$, and $\gamma=2$, the involution $\sigma$ has a geometric description and  $\phi_H\colon X\to \P^{19}$ is a morphism that factors  as $X\to X/\sigma\hra  \P^{19}$(see Section~\ref{sec262a}).\ 

When $m=4$, $n=1$, and $\gamma=2$, the involution $\sigma$ again has a geometric description\footnote{\label{foot24}Given a twisted cubic $C$ contained in the cubic fourfold $W$, with span $\langle C\rangle\isom\P^3$, take any quadric $Q\subset \langle C\rangle$ containing $C$; the intersection $Q\cap W$ is a curve of degree 6 which is the union of $C$ and another twisted cubic $C'$.\ The curve $C'$ depends on  the choice of $Q$, but not its image in $X$ (thanks to Ch.~Lehn for this description).\ The map $[C]\mapsto [C']$ defines an  involution $\sigma$ on $X$; the fact that it is biregular is not at all clear from this description, but follows from the fact that $\sigma^*H=H$.\ The fixed locus of $\sigma$ has two (Lagrangian) connected components and one is the  image of the canonical embedding $W\hra X$ (\cite[Theorem~B(2)]{lls}).}  and the map $\phi_H\colon X\dra \P^{14}$ also factors through the quotient  $X\to X/\sigma$, but it is not a morphism and has degree 72 onto its image (see Section~\ref{sec262}).

This last example  shows that, for a general polarized \hKm\ $X$ of $\KKK^{[m]}$-type with a polarization $H$ of degree 2, and contrary to what   the L conjecture\footnote{\label{f45}This conjecture states that if $(X,H)$ is a polarized \hKm\ corresponding to a   general point of a moduli space ${}^m\!\!\cM^{(\gamma)}_{2}$ (with $m\ge 2$ and $\gamma\in\{1,2\}$), the linear system $|H|$ is base-point-free and the morphism $\phi_H$ has degree 2 onto its image and factors through the quotient of $X$ by its involution from Proposition~\ref{prop27}.}
 from \cite[Conjecture~1.2]{og8} predicts, the map $\phi_H$ is not always a morphism of degree 2 onto its image.\ The question remains however of whether $\phi_H$ always factors through the quotient of $X$ by its involution $\sigma$ from Proposition~\ref{prop27}.

%
%

Let $(S,L)$ be a polarized K3 surface of degree $2e$ with $\Pic(S)=\Z L$ and $e\ge 2$.\ 
The class $H:= L_e-\delta$ on $S^{[e]}$  has square 2 and divisibility 1, and it is nef (Example~\ref{exa215}).\ 
There is (Example~\ref{exa37}) a geometrically defined birational  involution $\sigma$ on $S^{[e]}$  that acts in cohomology as the symmetry about the line spanned by $H$; it is therefore the involution of Remark~\ref{remar44} and it is biregular if and only if $e=2$ (precisely  when $H$ is ample).\ There is a factorization
$$\phi_H\colon S^{[e]} \stackrel{\alpha}{\dra}  S^{[e]}/\sigma\stackrel{\beta}{\dra} \Gr(e,e+2)\hra \P^{e(e+3)/2},
$$
where $\beta\circ \alpha$ is the rational map $\phi_e$   defined in \eqref{phim} and $\beta$ is dominant of degree  $\frac12\binom{2e}{e} $.\


 O'Grady works out in  \cite{og8} generalizations of this construction.\ More precisely, let $(S,L)$ be  a polarized K3 surface of degree $2e$, let $r$ be a positive integer, and let $\cM(r,L,r)$ be the moduli space of
Gieseker--Maruyama $L$-semistable torsion-free rank-$r$ sheaves $\cF$  with $c_1(\cF)=L$ and $c_2(\cF)=r[S]$.\ Yoshioka proved that under certain conditions (which are satisfied if $\Pic(S)=\Z L$; see \cite[Theorem~4.5]{og8}), $\cM(r,L,r)$ is a \hKm\ of type $\KKK^{[m]}$ with Picard number $\rho(S)+1$, where $m:=1-r^2+e$ (when $r=1$, this is just $S^{[e]}$). 

 Markman constructs geometrically a birational involution $\sigma$  on   $\cM(r,L,r)$ which acts  in cohomology as the symmetry about the line spanned by a  class $H$ of square 2 and divisibility~1   (\cite[Theorem~4.12 and Proposition~4.14]{og8}): it is the involution of Remark~\ref{remar44}.\ When $\Pic(S)=\Z L$ and   
$r^2<e\le r^2+2r-1$, the class $H$ is ample on  $\cM(r,L,r)$ and the involution $\sigma$ is biregular   (\cite[Corollary~4.15]{og8}).\ We will see in Example~\ref{ehm} that when $e\ge r^2+2r$, the class $H$ is nef but not ample on $\cM(r,L,r)$.
 \end{rema}

 \begin{exer}\label{er2}
\footnotesize\narrower

Let $r$ be a positive integer and let   $(S,L)$ be a polarized K3 surface of degree \mbox{$2(r^2+1)$} with $\Pic(S)=\Z L$.\ By the discussion above, the moduli space $X:=\cM(r,L,r)$ is a hyperk\"ahler fourfold with Picard number 2 that carries an ample class $H$ of square 2 and a biregular involution~$\sigma$.

\noindent (a) Show that the extremal rays of the movable cone of $ X$  are rational and interchanged by the involution $\sigma^*$ ({\em Hint:} in the notation of Example~\ref{ehm}, prove that the class $\bs:= (r-e,r+e,rL)$ is in $\cS_X$ and use the description of the movable cone given in Theorem~\ref{thm:NefConeHK4}(a)).


\noindent (b) Show that   the nef cone   of $X$ is equal to its movable cone and that
 $\Aut(X)=\Bir(X)=\{\Id,\sigma\}$ ({\em Hint:} prove that there are no $(-2)$-classes $\bs$ such that $\bs\cdot \bv_X=1$ and use the proof of Theorem~\ref{thm:NefConeHK4}(b)).
 

 \end{exer}
 
 These constructions allow us to prove part of the L conjecture from \cite{og8} (see footnote~\ref{f45}).
 
 \begin{coro}\label{coro46}
Let $(X,H)$ be a polarized \hKm\ of dimension $2m\ge 4$, degree~$2$ and divisibility 1, with biregular involution $\sigma$ (Remark~\ref{remar44}).\ The rational map
$$\phi_H\colon X\dra \P^{m(m+3)/2}$$
factors through the quotient $X\to X/\sigma$.
\end{coro}
 
 \begin{proof} 
The space of sections $H^0(X,H)$ has dimension  $\binom{m+2}{2}$ has decomposes as the direct sum of the space of $\sigma$-invariant sections and the space of $\sigma$-antiinvariant sections.\ Since the dimension  of each of these spaces is lower semicontinuous (they are the kernels of the endomorphisms $\Id\pm\sigma^*$ of a vector space of constant dimension), they are in fact locally constant.\ Since the moduli space ${}^m\!\!\cM^{(1)}_{2} $  is connected, it is therefore
enough to find one point where
  one of these   spaces vanishes.
  
  In Remark~\ref{rema45}, we have the examples   $(S^{[e]}, L_e-\delta)$, but the line bundle $H=L_e-\delta$ is in general only nef (and big).\ This is not a problem: the space $H^0(X,H)$ has the same dimension, and $H$ will become ample in a small deformation, providing the required example.
\end{proof}

\begin{rema}
In Proposition \ref{prop27}, the assumption that $X$ be very general in its moduli space is necessary and it is not enough in general to assume only that $X$ have Picard number~1.\ When $m=2$, the proof of~\cite[Theorem~3.1]{sbcomp}   implies that $\Bir(X)$ is trivial when $\rho(X)=1$, unless $n\in\{1,3,23\}$.\ These three cases are actual exceptions: all fourfolds corresponding to points of ${}^2\!\!\cM^{(1)}_2$ carry a nontrivial involution; there   is   a 10-dimensional subfamily of 
${}^2\!\!\cM^{(2)}_{6}$ whose elements consists of fourfolds that have an automorphism of order~3 and whose very general elements have  Picard number~1 (\cite[Section~7.1]{sbcomp}); there is
a
(unique) fourfold in ${}^2\!\!\cM^{(2)}_{46}$ with Picard number~1 and an automorphism of order 23 (\cite[Theorem~1.1]{sbisom}).
\end{rema}

  \subsection{Automorphisms of projective \hKm s with Picard number 2}

  Let $X$ be a \hKm.\  When the Picard number of $X$ is 2, the very simple structure
of its nef and movable cones    (they only have two extremal rays)  can be used to describe the groups $\Aut(X)$ and $\Bir(X)$.\ We get potentially more interesting groups.

   \begin{theo}[Oguiso]\label{thogui}
Let $X$ be a projective \hKm\ with Picard number~2.\ Exactly one of the following possibilities occurs:
\begin{itemize}
\item   the extremal rays of the cones  $\Nef(X)$ and $\Mov(X)$  are all rational and the groups $\Aut(X)$ and $\Bir(X)$ are  both finite;
\item both extremal rays of $\Nef(X)$ are rational and the group $\Aut(X)$ is finite, both extremal rays of $\Mov(X)$ are irrational and the group $\Bir(X)$ is infinite;
\item    the cones $\Nef(X)$ and $\Mov(X)$ are equal with irrational extremal rays and the groups $\Aut(X)$ and $\Bir(X)$ are equal and  infinite.
\end{itemize}
Moreover, when $X$ is of $\kkk[m]$-type, the group  $\Aut(X)$ (resp.\ $\Bir(X)$), when finite, is isomorphic to $(\Z/2\Z)^r$, with $r\le 2$. 
%
\end{theo}

We will see later that all three cases already occur when $\dim(X)=4$.

\begin{proof}
 One of the main ingredients of the proof in \cite{ogu2} is a result of Markman's (\cite[Theorem~6.25]{marsur}): there exists a  rational  polyhedral closed cone $\Delta\subset \Mov(X)$ such that 
\begin{equation}\label{cone}
\Mov(X)=\overline{\bigcup_{\sigma\in \Bir(X)}\sigma^*\Delta}
\end{equation}
(where $\sigma^*=\overline\Psi_X^B(\sigma)$) and $\sigma^*( \overset{\circ}{\Delta})\cap \overset{\circ}{\Delta}=\vide$   unless $\sigma^*=\Id$.\ This decomposition implies that if $\Bir(X)$ is finite, the cone $\Mov(X)$ is rational (\ie, both its extremal rays are rational).

Let $x_1$ and $x_2$ be generators  of the two extremal rays of $\Mov(X)$.\ Let $\sigma\in \Bir(X)$;  these extremal rays are either left stable by $\sigma^*$ or exchanged.\ We can therefore write $\sigma^{*2}(x_i)=\alpha_ix_i$, with $\alpha_i>0$; moreover, $\alpha_1+\alpha_2\in\Z$ and $\alpha_1\alpha_2=\det(\sigma^{*2})=1$, since $\sigma^*$ is defined over $\Z$.\ 

If the ray $\R_{\ge0}x_1$ is rational, we may choose $x_1$ integral primitive, hence $\alpha_1$ is a (positive) integer.\ Since $\alpha_1+\alpha_2\in\Z$, so is $\alpha_2$ and, since  $\alpha_1\alpha_2=\det(\sigma^{*2})=1$, we obtain $\alpha_1=\alpha_2=1$ and $\sigma^{*2}=\Id$.\ Thus, every element of $\overline\Psi_X^B(\Bir(X))$ has order 1 or 2.\ In the orthogonal group of a real plane, the only elements order 2 are $-\Id$, which cannot be in $\overline\Psi_X^B(\Bir(X))$ since it does not fix the movable cone, and isometries with determinant $-1$.\ It follows that $\overline\Psi_X^B(\Bir(X))$ has order 1 or 2.\ Proposition \ref{prop111}  implies that $\Bir(X)$ is finite, hence $\Mov(X)$ is rational.

The same argument shows that if one extremal ray $\R_{\ge0}y_1$ of the cone $\Nef(X)$ is rational, the group $\Aut(X)$ is finite and $\overline\Psi_X^A(\Aut(X))$ has order 1 or 2.\ To show that the other ray $ \R_{\ge0}y_2$ is also rational, we use a result of Kawamata (see \cite[Theorem~2.3]{ogu2}) that implies that if  $q_X(y_2)>0$, the ray   $ \R_{\ge0}y_2$ is rational.\ But if $q_X(y_2)=0$, the ray $ \R_{\ge0}y_2$ is also, by \eqref{inc}, an extremal ray for $\Mov(X)$.\ The decomposition \eqref{cone} then implies that for infinitely many distinct $\sigma_k\in \Bir(X)$, we have $\sigma_k^*\Delta\subset \Nef(X)$.\ Fix an integral class $x\in \overset{\circ}{\Delta}$; the integral class $a_k:=\sigma_k^*(x)$ is ample, and $\sigma_0^*\sigma_k^{*-1}(a_k)=a_0$.\ This implies  $\sigma_k^{-1}\sigma_0\in\Aut(X)$, which is absurd since $\Aut(X)$ is finite.

If one extremal ray $\R_{\ge0}y_1$ of the cone $\Nef(X)$ is irrational, the argument above implies that the other extremal ray $\R_{\ge0}y_2$ is also irrational, and $q_X(y_1)=q_X(y_2)=0$.\ The chain of inclusions \eqref{inc} implies $\Nef(X)=\Mov(X)$, hence $\Aut(X)=\Bir(X)$.

 To prove the last assertion, we use the classical (and rather elementary, see~\cite[Lemma~3.1]{huyk3}) fact that
the transcendental lattice   $\Pic(X)^\bot \subset H^2( X,\Z)$ carries a simple rational Hodge structure.\ Since its rank, 21, is   odd, $\pm1$ is, for any $\sigma\in\Bir(X)$,  an eigenvalue of the isometry $\sigma^*$ and the corresponding eigenspace is a sub-Hodge structure of $\Pic(X)^\bot$.\ This implies
   $\sigma^*\vert_{\Pic(X)^\bot}=\pm \Id$.\ When $X$ is of $\kkk[m]$-type, the morphisms $\Psi_X^A$ and $\Psi_X^B$ are injective and we just described their possible images when they are finite.\ This ends the proof of the theorem.
\end{proof}

One can go further when the nef and movable cones are explicitly known and obtain necessary conditions for nontrivial (birational or biregular) automorphisms to exist on a projective \hKm\ $X$.\
To actually construct a biregular automorphism is easy if the nef cone is known: the Torelli theorem (see Section~\ref{sec27}) implies that any isometry in $\widehat O(H^2(X,\Z),q_X)$ that preserves the nef cone is induced by an automorphism of~$X$.\ Given an isometry $\phi\in \widehat O(H^2(X,\Z),q_X)$ that preserves the movable cone, it is harder in general to construct a birational automorphism that induces that isometry.\ One way to do it is to identify, inside the movable cone of $X$, the nef cone of a birational model $X'$ of~$X$ (see~\eqref{chamb}) which is preserved by $\phi$, and construct the birational automorphism of $X$ as a biregular automorphism of $X'$ using Torelli.\ This is however not always possible and requires a good knowledge of the chamber decomposition~\eqref{chamb}.\ We will describe some situations when this can be done without too much effort, using the results of Sections \ref{sec210a} and \ref{sec27a}.

\subsubsection{Automorphisms of punctual Hilbert schemes of K3 surfaces}\hskip-2mm\footnote{Many of the results of this section can also be found in \cite{cat}.}\label{sec331}
Let $(S,L)$ be a polarized K3 surface such that $\Pic(S)=\Z L$ and let $m\ge 2$.\ The Picard group of $\sss[m]$ is isomorphic to $\Z L_m\oplus \Z\delta$.\ The ray spanned by the integral class $L_m\in\Pic(\sss[m])$     is  extremal  for both cones $\Nef(\sss[m]) $ and $\Mov(\sss[m])  $ (described in Example~\ref{exa215}).\ We are therefore in the first case of Theorem~\ref{thogui} and both groups $\Bir(\sss[m]) $ and $\Aut(\sss[m])  $ are 2-groups of order at most 4.\ More precisely, if $e:=\frac12 L^2$, 
\begin{itemize}
\item either $e=1$, the canonical involution on $S$ induces a canonical involution on $\sss[m]$ which generates the kernel of 
$$\overline\Psi_{\sss[m]}^B\colon\Bir(\sss[m])\lra O( \Pic(\sss[m]),q_{\sss[m]})$$
 and the groups $\Bir(\sss[m]) $ and $\Aut(\sss[m])  $ are 2-groups of order 2 or 4;
\item  or $e>1$, the morphism $\overline\Psi_{\sss[m]}^B$ is injective, and the groups $\Bir(\sss[m]) $ and $\Aut(\sss[m])  $ have order 1 or 2.
\end{itemize}

 As we explained earlier, it is hard  to determine precisely  the group $\Bir(S^{[m]})$, but we can at least formulate necessary conditions for this group to be nontrivial.\ When $e>1$ and $m\ge 3$ (the case $m=2$ will be dealt with in Proposition \ref{bbb2} below), one needs all the following properties to hold for $\Bir(S^{[m]})$ to be nontrivial:\footnote{\label{f35}Since $e>1$, the morphism $\overline\Psi_{\sss[m]}^B$ is injective and any nontrivial birational involution $\sigma$ of $ \sss[m] $  induces a nontrivial involution $\sigma^*$  of $\Pic(\sss[m])$ that preserves the movable cone; in particular, primitive generators of its two extremal rays need to have the same lengths.\ Since $m\ge 3$, this implies (Example~\ref{exa215}) that items~(a) and~(c) hold.\ The other extremal ray of $\Mov(S^{[m]})$ is then spanned by $a_1L_m-eb_1\delta$, where  $(a_1,b_1)$ is the minimal positive solution of the equation $a^2-e(m-1)b^2=1$ that satisfies 
$a_1\equiv\pm 1\pmod{m-1}$.

The   lattice $\Pic(S^{[m]})=\Z L_m\oplus \Z\delta$ has intersection matrix $\left(\begin{smallmatrix} 2e& 0\\0&-(2m-2)\end{smallmatrix}\right)$ and its orthogonal group is
\[
O(\Pic(S^{[m]}))=\left\{ \begin{pmatrix} a&  \alpha m'b\\e'b&\alpha a \end{pmatrix}  \Big|\ a,b\in \Z, \  a^2-e'm'b^2=1 ,\ \alpha=\pm 1 \right\},
\]
where $g :=\gcd(e,m-1)$ and we write $e=g e'$ and $m-1=g m'$.\ Since $\sigma^*(L_m)=a_1L_m-eb_1\delta$, we must therefore have $g=\gcd(e,m-1)=1$.\ 

The transcendental lattice   $\Pic(S^{[2]} )^\bot \subset H^2( S^{[2]},\Z)$ carries a simple rational Hodge structure (this is a classical fact found for example  in~\cite[Lemma~3.1]{huyk3}\label{fsimple}).\ Since the eigenspaces of the involution
$\sigma^*$ of $ H^2( S^{[2]},\Z)$ are sub-Hodge structures, the restriction of $\sigma^*$ to $\Pic(S^{[2]} )^\bot$ is $\eps \Id$, with $\eps\in\{\pm 1\}$.\ As in the proof of Proposition~\ref{prop27} and with its notation, we write $L_m=u+ev$; we have  
\begin{eqnarray*}
\sigma^* (u+ev)&=& a_1(u+ev)-eb_1\delta \\
\sigma^* (u-ev)&=& \eps(u-ev),
\end{eqnarray*}
which implies $2e\sigma^*(v) =  (a_1-\eps)u+e(a_1+\eps)v-eb_1\delta$,
hence $2e\mid a_1-\eps$ and $2\mid b_1$.}
\begin{itemize}
\item[(a)] $e(m-1)$ is not a perfect square;
\item[(b)] $e$ and $m-1$ are relatively prime;
\item[(c)] the equation $(m-1)a^2-eb^2=1$ is not solvable;
\item[(d)] if $(a_1,b_1)$ is the minimal positive solution of the equation $a^2-e(m-1)b^2=1$ that satisfies 
$a_1\equiv\pm 1\pmod{m-1}$, one has  $a_1\equiv\pm 1\pmod{2e}$ and $b_1$ even.
\end{itemize}
The group $\Bir(S^{[m]})$ is therefore trivial in the following cases:
\begin{itemize}
\item $m=e+1$, because item (a)  does not hold;
\item $m=e+2$,  because item~(c) does not hold;
\item $m=e+3$ ,  because item~(d) does not hold ($a_1=e+1$ and $b_1=1$);
\item $m=e-1$ ,  because item~(d) does not hold ($a_1=e-1$ and $b_1=1$).
\end{itemize}
  When $m=e\ge 3$, we will see in Example \ref{exa37} that the group $\Bir(S^{[m]})$ has order $2$:  items (a) and (b) hold, so does (d) ($a_1=2e-1$ and $b_1=2$), but I was unable to check  item (c) directly!

In some cases, we can be  more precise.

\begin{exam}[Case  $m\ge \frac{e+3}{2}$]\label{exa36}
Let $(S,L)$ be a polarized K3 surface such that $\Pic(S)=\Z L$ and $L^2=:2e\ge 4$.\ When $m\ge \frac{e+3}{2}$, {\em the automorphism group of $S^{[m]}$ is trivial.}\ Indeed, we saw in Example~\ref{exa215} that the ``other'' extremal ray of $\Nef(S^{[m]})$ is spanned by the primitive vector $((m+e)L_m-2e\delta)/g$, where $g:=\gcd (m+e,2e)$.\ Any nontrivial automorphism of  $S^{[m]}$ acts nontrivially on $\Pic(S^{[m]})$ and exchanges the two extremal rays.\ The primitive generator $L_m$ of the first ray (with square $2e$) is therefore sent to $((m+e)L_m-2e\delta)/g$, with square $(2e(m+e)^2-4e^2(2m-2))/g^2$, hence
$$ g^2= (m+e)^2-2e (2m-2)=(m-e)^2+4e.
$$
Since $g\mid m-e$, this implies $g^2\mid 4e$; but it also implies $g^2\ge 4e$, hence $g^2=4e$ and $m=e$.\ We obtain 
$g =\gcd (2e,2e)=2e$ and $e=1$, which contradicts our hypothesis.
\end{exam}

\begin{exam}[The Beauville involution (case $m=e$)]\label{exa37}
Let $S\subset \P^{e+1}$ be a K3 surface of degree $2e \ge 4$.\ By sending a general point $Z\in\sss[e]$ to the residual intersection $(\langle Z\rangle\cap S)\moins Z$, one defines a birational involution $\sigma$ of $\sss[e]$; it is biregular if and only if $e=2$ and $S$ contains no lines (\cite[Section~6]{beac10}).\ 

Assume $\Pic(S)=\Z L$ (so that $S$ contains no lines).\  The isometry $\sigma^*$ of $\Pic(\sss[e])$ must then exchange the two  extremal rays of the movable cone; in particular, a primitive generator of  the ``other'' extremal ray of $\Mov(S^{[e]})$ must have square $2e$.\ We must therefore be in the third case of Example~\ref{exa215}
and $\sigma^*(L_e)=(2e-1)L_e-2e\delta$ and $\sigma^*((2e-1)L_e-2e\delta)=L_e$, so that $\sigma^*(L_e-\delta)=L_e-\delta$: the axis of the reflection $\sigma^*$ is spanned by the square-$2$ class $ L_e-\delta$.

When $e=2$, the nef and movable cones are equal (see Example \ref{exa217}) and $\sigma$ is biregular.\ When $e\ge 3$, the class $ L_e-\delta$ spans the ``other'' extremal ray of the nef cone (Example~\ref{exa36}).\ It is therefore not ample, which confirms the fact that $\sigma$ is not biregular (either by the Beauville result mentioned above or by Example~\ref{exa36}).\ Note also that $ L_e-\delta$ is not
  base-point-free by Corollary~\ref{cor17}.\ The chamber decomposition~\eqref{chamb} is
$$\Mov(S^{[e]})=\Nef(S^{[e]})\cup \sigma^*\Nef(S^{[e]}).$$
In particular, $S^{[e]}$ has no other \hK\ birational model than itself and $\Bir(S^{[e]} )=\{\Id,\sigma\}$.
\end{exam} 

\begin{exam}[A birational involution of a Hilbert cube (case $m=3$ and $e=5$)]\label{exa412}
A general intersection $X:=\Gr(2,\C^5)\cap\P^6\subset \P(\bw2\C^5)$ is a (smooth) Fano threefold of degree $5$ and
we saw in Section~\ref{sec13} that  a general polarized K3 surface $(S,L)$ of degree 10 is the transverse intersection of  $X$ and a quadric $Q\subset \P^6$.\   

A general point of $S^{[3]}$ corresponds to projective lines $\ell_1,\ell_2,\ell_3\subset \P^4$.\ Let $\Sigma_{\ell_i}\subset\Gr(2,\C^5) $ be the Schubert cycle of lines meeting $\ell_i$. Since $\Sigma_{\ell_i}\cdot \Sigma_{\ell_2}\cdot \Sigma_{\ell_3} =1$, there exists a unique line $\ell\subset \P^4$ that meets $\ell_1,\ell_2,\ell_3$.\ The intersection $\Sigma_\ell\cap X$ is a (possibly reducible) rational normal cubic curve, which meets $S$ in six points.\ Since $\ell_i\in \Sigma_\ell$, the points of $S$ corresponding to $\ell_1,\ell_2,\ell_3$ are among these six points.\ By associating with them the three remaining points, we obtain a commutative diagram
 \begin{equation*}\label{d10}
\xymatrix
@C=18pt@M=6pt
{
S^{[3]} \ar@{-->}[rr]^\iota\ar@{-->}[dr]_-\pi&& S^{[3]} \ar@{-->}[dl]^-\pi\\
&\Gr(2,\C^5),}
\end{equation*}
 where  $\iota$ is   a birational involution and $\pi$ is dominant, generically finite of degree 20.\ Another geometric description of this involution is   given in \cite[Proposition~4.1]{ikkr1} (the pair $(S^{[3]},2L_3-3\delta)$ occurs as a degeneration of elements of $\,{}^3\!\!\cM^{(2)}_{4}$; see Section~\ref{sec262a}).
 
 The other extremal ray of the movable cone of  $S^{[3]}$ is spanned by $19L_3-30\delta$ (Example~\ref{exa215}) hence $\iota^*L_3=19L_3-30\delta$ and $\iota^*(2L_3-3\delta)=2L_3-3\delta$.\ Using for example  \cite[Section~2.2]{mon}, we see that the extremal rays of the fundamental chambers in the moving cone are spanned by\footnote{According to that reference, the extremal rays of the three fundamental chambers in the moving cones are   the orthogonal of divisors $D=b\div(D)L_3-a\delta$ such that
 \begin{itemize}
\item either $D^2=-12$ and $\div(D)=2$: it gives the equation $a^2-10b^2=3$, which has no solutions (reduce modulo $5$);
\item either $D^2=-36$ and $\div(D)=4$: it gives the equation $a^2-40b^2=9$, which has solutions $(7,1)$ and $(13,2)$ (the corresponding rays appear in \eqref{rayss3});
\item either $D^2=-2$ and $\div(D)=1$: it gives gives the equation $2a^2-5b^2=1$, which has no solutions (reduce modulo $5$);
\item either $D^2=-4$ and $\div(D)=4$: it gives gives the equation $a^2-40b^2=1$, which has solution $(19,3)$ (the corresponding ray is the other ray of the moving cone);
\item either $D^2=-4$ and $\div(D)=2$: it gives gives the equation $a^2-10b^2=1$ with $b$ odd, which has no solutions.
\end{itemize}
}
\begin{equation}\label{rayss3}
7L_3-10\delta\qquad\textnormal{and}\qquad 13L_3-20\delta=\iota^*(7L_3-10\delta) 
.
\end{equation}
In particular, $\iota$ is not regular, but $S^{[3]}$ has a unique nontrivial \hK\ birational model $X$ on which $\iota$ is regular.\ The class $ 2L_3-3\delta$ is ample on $X$ and the pair $(X,2L_3-3\delta)$ defines a point in $\,{}^3\!\!\cM^{(2)}_{4}$. 
\end{exam}


We end this section with a detailed study of the groups $\Aut(S^{[2]})$ (see \cite[Proposition~5.1]{bcns}) and $\Bir(S^{[2]})$.\

\begin{prop}[Biregular automorphisms of $S^{[2]}$]\label{bbb2a}
Let $(S,L )$ be a polarized K3 surface of degree $2e$ with   Picard group $\Z L $.\ The group $\Aut(S^{[2]})$ is trivial except when the equation $\cP_{e}(-1)$ is solvable and the equation $\cP_{4e}(5)$ is not, or when $e=1$.

If this is the case and $e>1$, the only  nontrivial automorphism of $S^{[2]}$ is an antisymplectic involution $\sigma$, the fourfold $S^{[2]}$ defines an element of $\, {}^2\!\!\cM^{(1)}_2$ (hence generically a double EPW sextic) and $\sigma$ is its canonical involution.
 \end{prop}

\begin{proof}
When $e=1$, the canonical involution of $S$ induces an involution on $S^{[2]}$.\ So we assume $e>1$ and $\Aut(S^{[2]})$ nontrivial.\ The morphism $\overline\Psi_{\sss[m]}^B$ is then injective, so we look for nontrivial involutions $\sigma^*$ of $\Pic(S^{[2]})$ that preserve the nef cone.

If the equation $\cP_{4e}(5)$ has a minimal solution $(a_5,b_5)$, the  extremal rays of $\Nef(S^{[2]})$ are spanned by $L_2$ and $a_5L_2-2eb_5\delta$ (see Example~\ref{exa217}), hence $\sigma^*(L_2)=(a_5L_2-2eb_5\delta)/g$, where $g=\gcd(a_5, 2eb_5)$.\ Taking lengths, we get $2eg^2=10e$, which is absurd.\ Therefore, the equation $\cP_{4e}(5)$ is not solvable, the other ray of the nef cone is spanned by $a_1L_2-eb_1\delta$, where $(a_1,b_1)$ is the minimal solution to the equation $\cP_{e}(1)$, and $\sigma^*(L_2)=a_1L_2-eb_1\delta$.\ As in footnote~\ref{f35} and with its notation, we have $2e\mid a_1-\eps$ and $2\mid b_1$, with $\eps\in\{-1,1\}$.\ Setting $a_1=2ea+\eps$ and $b_1=2b$, we rewrite the equation $\cP_{e}(1)$ as
$$ a(ea+\eps)=b^2.
$$
Since $a$ and $ea+\eps$ are relatively prime, there exists positive integers $r$ and $s$ such that $a=s^2$, $ea+\eps=r^2$, and $b=rs$.\ The pair $(r,s)$ satisfies the equation $r^2=ea+\eps=es^2+\eps$, \ie, $\cP_e(\eps)$.\ Since $(a_1,b_1)$ is the minimal solution to the equation $\cP_{e}(1)$, we must have $\eps=-1$.

Conversely, if $(r,s)$ is the minimal solution to the equation $\cP_e(-1)$, the class $H:=sL_2-r\delta$ has square 2 and is ample (because it is proportional to $ L_2+(a_1L_2-eb_1\delta) $).\ The pair  $ (S^{[2]},H)$ is an element of $\, {}^2\!\!\cM^{(1)}_2$ and as such, has an antisymplectic involution by Proposition~\ref{prop27}.
\end{proof}

\begin{rema}
When $e=2$, the equation  $\cP_2(-1)$ is solvable and the involution $\sigma$ of $S^{[2]}$ is the (regular) Beauville involution  described in Example~\ref{exa37}.\ There is a dominant morphism $S^{[2]}\to \Gr(2,4)$ of degree 6 that sends a length-2 subscheme of $S$ to the line that it spans and it factors through the quotient  $ S^{[2]}/\sigma$.\ In that case, $S^{[2]}/\sigma$ is not an EPW sextic (although one could say that it is three times the smooth quadric $\Gr(2,4)\subset \P^5$, which is a degenerate EPW sextic!) but $(S^{[2]},L_2-\delta)$ is still an element of $ \,{}^2\!\!\cM^{(1)}_2$ (see Example \ref{ex19} and Section~\ref{sec261}).

\end{rema}

\begin{prop}[Automorphisms of $S^{[2]}$]\label{bbb2}
Let $(S,L )$ be a polarized K3 surface of degree $2e$ with   Picard group $\Z L $.\ The group $\Bir(S^{[2]})$ is trivial except in the following cases:
\begin{itemize}
\item  $e=1$, or  the equation $\cP_{e}(-1)$ is solvable and the equation $\cP_{4e}(5)$ is not, in which cases $\Aut(S^{[2]})=\Bir(S^{[2]})\isom\Z/2\Z$;
\item $e=5$, or $e>1$,  $5\nmid e$, and both equations   $\cP_{e}(-1)$ and $\cP_{4e}(5)$ are  solvable,  in which case $\Aut(S^{[2]})=\{\Id\}$ and  $\Bir(S^{[2]})\isom\Z/2\Z$.
\end{itemize}
 \end{prop}
 

\begin{proof} 
 If $\sigma\in \Bir(S^{[2]})$ is not biregular, $  \sigma^* $ is a reflection that acts on the movable cone $\Mov(S^{[2]})$ in such a way that $ \sigma^*(\Amp(S^{[2]}))\cap \Amp(S^{[2]})=\vide$ (if the pull-back by $\sigma$ of an ample class were ample, $\sigma$ would be regular).\ This implies $\Mov(S^{[2]})\ne \Nef(S^{[2]})$ hence (see Example~\ref{exa217})  
 \begin{itemize}
\item either $e=1$ and $ \sigma^*(L_2)=  L_2- \delta$, which is impossible since these two elements do not have the same square;
\item or $e>1$,  the equation $\cP_{4e}(5)$ has a minimal solution $(a_5,b_5)$ (which implies that $e$ is not a perfect square), and, as in the proof of Proposition~\ref{bbb2a}, we have 
$ \sigma^*(L_2)= a_1L_2-eb_1\delta$, $\sigma^*\vert_{\Pic(S^{[2]} )^\bot}=-\Id$,    $a_1= 2eb_{-1}^2-1$, and $b_1=2a_{-1}b_{-1}$, where $(a_{-1},b_{-1})$ is the minimal solution to the equation $\cP_e(-1)$.
\end{itemize}

 The chamber decomposition of the movable cone, which is preserved by $\sigma^*$, was determined in Example \ref{exa217}: since $b_1 $ is even,
 \begin{itemize}
\item either   $5\mid e$, there are two chambers, and the middle wall is spanned by  $ a_5L_2-2eb_5\delta$;
\item or  $5\nmid e$, there are three chambers, and the middle walls are spanned by  $ a_5L_2-2eb_5\delta$ and $(a_1a_5-2eb_1b_5)L_2-e(a_5b_1-2a_1b_5)\delta$ respectively.
\end{itemize}
In the first case, the two chambers are $\Nef( S^{[2]})$ and $\sigma^*(\Nef( S^{[2]}))$ hence the wall must be the axis of the reflection $\sigma^*$.\ It follows that there is an integer $c$ such that $a_5=c b_{-1}$ and $2eb_5=c a_{-1}$; substituting these values in the equation $\cP_{4e}(5)$, we get $5=c^2b_{-1}^2-c^2a_{-1}^2/e= c^2/e$, hence $e=c=5$.\ In that case, one can construct geometrically a nontrivial birational involution on   $S^{[2]}$ (Example \ref{exa311}).

In the second case, since the reflection $\sigma^*$ respects the chamber decomposition, the square-2 class $H:= b_{-1}L_2- a_{-1}\delta$ is in the interior of the ``middle'' chamber, which is the nef cone of a birational model $X $ of $S^{[2]}$.\ It
  is therefore ample on  $X $ and the pair $(X,H) $ is an element of $\, {}^2\!\!\cM^{(1)}_2$; as such, it has an antisymplectic involution by Proposition~\ref{prop27}, which induces a birational involution of $ S^{[2]}$.
 \end{proof}

\begin{exam}[The O'Grady involution]\label{exa311}
A general intersection $X:=\Gr(2,\C^5)\cap\P^6\subset \P(\bw2\C^5)$ is a (smooth) Fano threefold of degree $5$ and
we saw in Section~\ref{sec13} that  a general polarized K3 surface $S$ of degree 10 is the transverse intersection of  $X$ and a quadric $Q\subset \P^6$.\   

 A general point of $S^{[2]}$ corresponds to $V_2,W_2\subset \C^5$ and 
$$\Gr(2,V_2\oplus W_2)\cap S=\Gr(2,V_2\oplus W_2)\cap Q \cap \bw2(V_2\oplus W_2))\subset\P^2
$$
is the intersection of two general conics in $\P^2$ hence consists of 4 points.\ The (birational) O'Grady   involution $S^{[2]}\dra S^{[2]}$ takes the pair of points $([V_2],[W_2])$  to the residual two points of this intersection.

This involution has the following geometric description.\ The scheme $L(X)$ of lines contained in $X$ is a $\P^2$ and the map that takes a line $L\subset X$ to $L\cap Q$ (which is regular if $S$ contains no lines) defines an embedding $L(X)\hra \SS$.\ The involution is the Mukai flop of this $\P^2$ (\cite[Section~4.3]{og8}).
\end{exam}

\begin{rema}
 In the second case of the proposition, when $5\nmid e$, the fourfold $S^{[2]}$ is birationally isomorphic to a double EPW sextic, whose canonical involution induces the only nontrivial birational automorphism of $ S^{[2]}$.
 \end{rema}

\begin{rema}
 There are  cases where both equations $\cP_{e}(-1)$ and $\cP_{4e}(5)$ are solvable  and $ 5\nmid e$: when $e=m^2+m-1$, so that $(2m+1,1)$ is the minimal solution of the equation $\cP_{4e}(5)$, and $m\not\equiv 2\pmod5$, this happens for $m\in\{ 5,6,9,10,13,21\}$.\
I do not know whether this happens for infinitely many integers $m$.
\end{rema}

  \subsubsection{Automorphisms of some other  \hK\ fourfolds with Picard number~2}
 
We determine the automorphism groups of the \hK\ fourfolds of $\KKK^{[2]}$-type studied in Section~\ref{sec27a} and find more interesting groups.\ The polarized \hK\ fourfolds $(X,H)$ under study are those  for which $H$ has square $2n$ and divisibility 2   (so that   $n\equiv -1\pmod4$), and $\Pic(X)=\Z H\oplus \Z L$, with   intersection matrix   $\left(\begin{smallmatrix}2n&0\\0&-2e'\end{smallmatrix}\right)$, where $e'>1$.\ If $n$ is square-free, a very general element of the irreducible hypersurface $  {}^2\cC_{2n, 2e'n}^{(2)}\subset    {}^2\!\!\cM_{2n}^{(2)}$ is of this type (Proposition~\ref{propirr}(2)(b)).

 \begin{prop}\label{autfw*}
Let $(X,H)$ be a polarized \hK\ fourfold as above.
   \begin{itemize}
  \item[{\rm (a)}] If both equations $\cP_{n,e'}(-1)$ and $\cP_{n,4e'}(-5)$ are not solvable and $ne'$ is not a perfect square, the groups   $\Aut(F)$ and   $\Bir(F)$ are equal.\ They are infinite cyclic, except when the equation $\cP_{n,e'}(1)$ is solvable, in which case these groups are isomorphic to the infinite dihedral group.\footnote{This is the group $\Z\rtimes \Z/2\Z$, also isomorphic to the free product $ \Z/2\Z\star \Z/2\Z$.}
  \item[{\rm (b)}] If the equation
    $\cP_{n,e'}(-1)$ is not solvable but the equation $\cP_{n,4e'}(-5)$ is,
   the  group $\Aut(X)$ is trivial and
  the       group $\Bir(X)$ is infinite cyclic, except when the equation $\cP_{n,e'}(1)$ is solvable, in which case it is  infinite dihedral.
   \item[{\rm (c)}] If the equation $\cP_{n,e'}(-1)$ is solvable or if $ne'$ is a perfect square, the group $\Bir(X)$ is trivial.
    \end{itemize}
   \end{prop}
   
   When $n$ is square-free, we can consider equivalently the equation $\cP_{ne'}(\pm n)$ instead of $\cP_{n,e'}(\pm 1)$ and the equation $\cP_{4ne'}(-5n)$ instead of $\cP_{n,4e'}(-5)$.

\begin{proof}
The map $\Psi_X^A\colon \Aut(X)\to  O(H^2(X,\Z),q_X)$ is injective (Proposition \ref{prop111}).\ Its image consists of isometries which  preserve $\Pic(X)$ and the  ample cone and, since $b_2(X)-\rho(X)$ is odd,  restrict  to $\pm \Id $ on $\Pic(X)^\bot$  (\cite[proof of Lemma~4.1]{ogu3}).\ Conversely, by Theorem~\ref{torthhk}, any isometry with these properties is in the image of $\Psi_X^A$.\ We begin with some general remarks on the group $G$ of isometries of   $H^2(X,\Z)$ which preserve $\Pic(X)$ and the components of the positive cone, and restrict   to $\eps \Id $ on $\Pic(X)^\bot$, with $\eps\in\{-1,1\}$.

As we saw in footnote \ref{fsimple}, we have\[
O(\Pic(X),q_X)=\left\{ \begin{pmatrix} a& \alpha e''b\\n'b&\alpha a \end{pmatrix}  \Big|\ a,b\in \Z, \  a^2-n'e''b^2=1 ,\ \alpha\in\{-1, 1\} \right\},
\]
where $g :=\gcd(n,e')$, $n'=n/g$, and $e''=e'/g$.\ Note that $\alpha$ is the determinant of the isometry and
\begin{itemize}
 \item  such an isometry   preserves  the components of  the positive cone if and only if $a>0$; we denote the corresponding subgroup by $O^+(\Pic(X),q_X)$;
 \item when $ne'$ is not a perfect square, the group $SO^+(\Pic(X),q_X)$ is infinite cyclic, generated by the isometry $R$ corresponding to the minimal solution of the equation $\cP_{n'e''}(1)$ and the group $O^+(\Pic(X),q_X)$ is infinite dihedral;
 \item when $ne'$ is  a perfect square, so is $n'e''=ne'/g^2$, and   $O^+(\Pic(X),q_X)=\{ \Id, \left(\begin{smallmatrix} 1&0\\0&-1\end{smallmatrix}\right)\}$. 
\end{itemize}
    
By Eichler's criterion, there exist standard bases $(u_1,v_1)$ and $(u_2,v_2)$    for   two orthogonal hyperbolic planes in $\Lambda_{\KKK^{[2]}}$, a generator $\ell$ for the $I_1(-2)$ factor, and an isometric identification $H^2(X,\Z)\isomto \Lambda_{\KKK^{[2]}}$ such that
\[
H=2u_1+\frac{n+1}{2}v_1+\ell \quad \text{and}\quad L=u_2-e'v_2.
\]
The elements $\Phi$ of $G$  must then satisfy $a>0$ and
\begin{eqnarray*} 
   \Phi(2u_1+\tfrac{n+1}{2}v_1+\ell)&=&a(2u_1+\tfrac{n+1}{2}v_1+\ell)+n'b(u_2-e'v_2)\\
   \Phi(u_2-e'v_2)&=&\alpha e''b(2u_1+\tfrac{n+1}{2}v_1+\ell)+\alpha a(u_2-e'v_2)\\
   \Phi(v_1+\ell)&=&\eps (v_1+\ell)\\
  \Phi(u_1-\tfrac{n+1}{4}v_1)&=&\eps (u_1-\tfrac{n+1}{4}v_1)\\
 \Phi(u_2+e'v_2)&=&\eps (u_2+e'v_2)
   \end{eqnarray*}
   (the last three lines correspond  to elements of $\Pic(X)^\bot$).\ From this, we deduce
     \begin{eqnarray*} 
   n\Phi(v_1)&=&2(a-\eps)u_1+\bigl( (a+\eps)\tfrac{n+1}{2}-\eps\bigr)v_1+(a-\eps)\ell+n'b(u_2-e'v_2)\\
  2 \Phi(u_2)&=&2\alpha e'' bu_1+\alpha e'' b\tfrac{n+1}{2}v_1+
  \alpha e''b \ell+(\eps+\alpha a)u_2+e' (\eps-\alpha a)v_2\\
   2e'\Phi(v_2)&=&-2\alpha e'' bu_1-\alpha e'' b\tfrac{n+1}{2}v_1
  -\alpha e''b \ell+(\eps-\alpha a)u_2+e' (\eps+\alpha a)v_2.
   \end{eqnarray*} 
   From the first equation, we get $g\mid b$ and $a\equiv \eps\pmod{n}$; from the second equation, we deduce that $e''b$ and $\eps+\alpha a$ are even; from the third equation, we get $2g\mid b$ and $a\equiv \alpha\eps\pmod{2e'}$.\ All this is equivalent to $a>0$ and
\begin{equation}\label{eqab**}
2g\mid b\quad,\quad a\equiv \eps\pmod{n} \quad,\quad   a\equiv \alpha\eps\pmod{2e'}.
\end{equation}

 Conversely, if these conditions are realized, one may define $\Phi$ uniquely on $\Z u_1\oplus \Z v_1\oplus \Z u_2\oplus \Z v_2 \oplus \Z \ell$ using  the formulas above, and extend it by $\eps \Id$ on the orthogonal of this lattice in $ \Lambda_{\KKK^{[2]}}$ to obtain an element of $G$.

The first congruence in \eqref{eqab**} tells us that the identity on $\Pic(X)$ extended by $- \Id$ on its orthogonal does not lift to an isometry of $H^2(X,\Z) $.\ This means that the restriction
$ G\to O^+(\Pic(X),q_X)$ is injective.\ Moreover, the two congruences in \eqref{eqab**} imply $a\equiv \eps \equiv \alpha\eps\pmod{g}$.\ If
  $g>1$,   since $n$, hence also $g$, is odd, we get  $\alpha=1$, hence the image of $G$ is contained in $SO^+(\Pic(X),q_X)$.

\noindent{\bf Assume $\alpha=1$.}
The relations \eqref{eqab**} imply that $a-\eps$ is divisible by $n$ and $2e'$, hence by their least common multiple $2gn'e''$.\ We write
$b=2g b'$ and $a=2g n'e''a'+\eps$  
and
 obtain from the equality  $a^2-n'e''b^2=1$  the relation
\[
4g^2n^{\prime 2}e^{\prime\prime 2}a^{\prime 2}+4\eps g n'e''a'=4g^2n'e'' b^{\prime 2},
\]
hence
\[
g n'e''a^{\prime 2}+\eps a'=g b^{\prime 2}.
\]
In particular, $a'':=a'/g$ is an integer and $b^{\prime 2}=a''(ne'a''+\eps)$.\ 

Since $a>0$ and $a''$ and $ne'a''+\eps$ are coprime, both are perfect squares and there exist coprime integers $r$ and $s$, with $r>0$, such that
\[
a''=s^2\quad,\quad ne' a''+\eps=r^2 \quad,\quad b'=rs.
\]

Since  $-1$ is not a square modulo $n$, we obtain $\eps=1$; the pair $(r,s)$ satisfies the Pell equation $r^2- ne' s^2=1$, and $a=2ne's^2+1$ and $b=2g rs$.\ 
In particular, either $ne'$ is not a perfect square and there are always infinitely many solutions, or $ne'$ is a perfect square and we get $r=1$ and $s=0$, so that $\Phi=\Id$.

\noindent{\bf Assume $\alpha=-1$.}
As observed before, we have $g=1$, \ie, $n$ and $e'$ are coprime.\ Using~\eqref{eqab**}, we may
  write $b=2b'$ and $a=2a'e'-\eps$.\ Since $2\nmid n$ and $a\equiv \eps \pmod{n}$, we deduce  $\gcd(a',n)=1$.\ 
Substituting  into the equation $a^2-ne'b^2=1$, we obtain
\[
a'(e'a'-\eps)=nb^{\prime 2},
\]
hence there exist coprime integers $r$ and $s$, with $r\ge 0$, such that $b'=rs$, $a'=s^2$, and $e'a'-\eps=nr^2$.\ 
The pair $(r,s)$ satisfies the equation $nr^2-e's^2=-\eps$, and $a=2e's^2-\eps$ and $b=2rs$.\ 
In particular,  one of the two equations $\cP_{n,e'}(\pm 1)$ is solvable.\ 
Note that at most one of these equations may be solvable: if $\cP_{n,e'}(\eps)$ is solvable, $-\eps e'$ is a square modulo $n$, while $-1$ is not.\ 
These isometries are all involutions and, since $n\ge 2$ and $e'\geq2$, $\left(\begin{smallmatrix} 1&0\\0&-1\end{smallmatrix}\right)$ is not one of them.\ In particular, if $ne'$ is a perfect square,   $G=\{\Id\}$.

\medskip

We now go back to the proof of the proposition.\  
We proved that the composition $\Aut(X)\to G\to O^+(\Pic(X),q_X)$ is injective and, by the discussion in Section~\ref{secpsi}, so is the morphism $\Bir(X)\to G\to O^+(\Pic(X),q_X)$ (any element of its kernel is in $\Aut(X)$).

Under the hypotheses of (a),  both slopes of the nef   cone are   irrational (Section~\ref{sec27a}), hence the groups $\Aut(X)$ and   $\Bir(X)$ are   equal and infinite (Theorem~\ref{thogui}).\ The calculations above allow us to be more precise: in this case, the ample cone is just one component of the positive cone and the groups $\Aut(X)$ and $G$ are isomorphic.\ The conclusion follows from the discussions above.

Under the hypotheses of (c), the slopes of the extremal rays of the nef and movable cones are rational (Section~\ref{sec27a}) hence, by~Theorem~\ref{thogui} again,  $\Bir(X)$ is a finite group.\ By~\cite[Proposition~3.1(2)]{ogu2},  any nontrivial element $\Phi$ of its image in $O^+(\Pic(X))$  is an involution which satisfies $\Phi(\Mov(X))= \Mov(X)$, hence  
  switches the two extremal rays of this cone.\ This means $\Phi(H\pm \mu L)=H\mp \mu L$, hence $\Phi(H)=H$, so that  $\Phi=\left(\begin{smallmatrix} 1&0\\0&-1\end{smallmatrix}\right)$.\ Since we saw that this is impossible,  the group  $\Bir(X)$ is trivial.
 
Under the hypotheses of (b), the slopes of the extremal rays of the nef cone are both rational and the slopes of the extremal rays of the movable cone are both irrational (Section~\ref{sec27a}).\ 
By~Theorem~\ref{thogui} again, $\Aut(X)$ is a finite group and $\Bir(X)$ is infinite.\ 
The same reasoning as in case~(c) shows that the group  $\Aut(X)$ is in fact trivial; moreover, the group $\Bir(X)$  is a subgroup of $\Z$, except when the equation $\cP_{n,e'}(1) $ is solvable, where it is a subgroup of $\Z\rtimes \Z/2\Z$.\
 
In the latter case, such an infinite subgroup is isomorphic  either to $\Z$ or to $\Z\rtimes \Z/2\Z$ and we exclude the first case by  showing that there is indeed a regular involution on  a birational model of $X$.\

As observed in Appendix~\ref{secpell},  the positive solutions $(a,b)$       to the equation $\cP_{n,4e'}(-5)$    determine an infinite sequence of rays $\R_{\ge0}(2e'bH\pm naL)$ in $\Mov(X)$.\ 
The nef cones of hyperk\"ahler fourfolds birational to $X$ can be identified with the chambers with respect to this collection of rays.\  
In order to apply Lemma~\ref{le46} and show that the equation $\cP_{n,4e'}(-5)$ has two classes of solutions, we need to check that $5$  divides neither $n$ nor $e'$.\ 
We will use quadratic reciprocity and, given  integers $r$ and $s$, we denote by $\left(\frac{r}{s}\right)$ their Jacobi symbol.

 Assume first  $5\mid e'$.\
Since the equation $\cP_{n,e'}(1)$ is solvable, we have   $\left(\frac{n}{5} \right)=1$; moreover, since $n\equiv -1\pmod{4}$, we have   $\left(\frac{e'}{n}\right)=-1$.\
The solvability of the equation $\cP_{n,4e'}(-5)$ implies $\left(\frac{5}{n}\right)=\left(\frac{e'}{n}\right)$; putting all that together contradicts quadratic reciprocity.

Assume now  $5\mid n$ and set $n':=n/5$.\ Since the equation $\cP_{n,e'}(1)$ is solvable, we have   $\left(\frac{e'}{5} \right)=1$; moreover, since $n'\equiv -1\pmod{4}$, we have   $\left(\frac{e'}{n'}\right)=-1$.\ Since $5\nmid e'$, the equation $\cP_{n',20e'}(-1)$ is solvable, hence $\left(\frac{5e'}{n'}\right)=1$; again, this contradicts quadratic reciprocity.

The assumptions of Lemma~\ref{le46} are therefore satisfied and the equation $\cP_{n,4e'}(-5)$ has two (conjugate) classes of solutions.\ We can reinterpret this
as follows.\ Let $(r,s)$ be the minimal solution to the equation $\cP_{n,e'}(1)$; by Lemma \ref{lemmpell}, the minimal solution to the equation $\cP_{ne'}(1)$ is
 $(a,b):=(nr^2+e's^2,2rs)$ and it corresponds to 
 the generator $R=\left(\begin{smallmatrix} a&  e'b\\n b& a \end{smallmatrix}\right)$ of $SO^+(\Pic(X),q_X)$ previously defined (we are in the case $\gcd(n,e')=1$).\

The two extremal rays of the nef cone of $X$ are spanned by
$\mathbf{x}_0:=2e'b_{-5}H - na_{-5}L$ and $\mathbf{x}_1:=2e'b_{-5}H + na_{-5}L$, where $(a_{-5},b_{-5})$ is the minimal solution to the equation  $\cP_{n,4e'}(-5)$.\ If we set $\mathbf{x}_{i+2}:=R(\mathbf{x}_i)$, the fact the $\cP_{n,4e'}(-5)$ has two classes of solutions means exactly that the ray $\R_{\ge0}\mathbf{x}_2$ is ``above'' the ray $\R_{\ge0}\mathbf{x}_1$; in other words,  we get an ``increasing'' infinite sequence of  rays 
 $$\cdots< \R_{\ge0}\mathbf{x}_{-1}<\R_{\ge0}\mathbf{x}_0< \R_{\ge0}\mathbf{x}_1< \R_{\ge0}\mathbf{x}_2< \cdots.$$

It follows from the discussion above that the involution $R\left(\begin{smallmatrix} 1&0\\0&-1\end{smallmatrix}\right)$  belongs to the group $G$ and preserves the nef cone of the birational model $X'$ of $X$ whose nef cone is generated by $\mathbf{x}_1$ and $\mathbf{x}_2$.\ It is therefore induced by a biregular involution of $X'$ which defines a birational involution of $X$.\ 
This concludes the proof of the proposition.
\end{proof}

 \begin{rema}
The case $n=3$ and $e'=2$ (where $\Aut(X)=\{\Id\}$ and $\Bir(X)\isom \Z\rtimes \Z/2\Z $) was treated geometrically by Hassett and Tschinkel in \cite{hast3}.\ Here is a table for the groups $\Aut(X) $ and $\Bir(X)$ when $n=3$ and $2\le e'\le 11$. 
 $$\tiny{
 \setlength{\extrarowheight}{1ex}
 \begin{array}{|c|c|c|c|c|c|c|c|c|c|c|c|c|c|c|c|c|c|}
 \hline 
e'&2&
3&4&
 5&6 &7&8&9&10&
11   \\
 \hline 
 \Aut(X)&\Id&
 \Id&\Id& 
 \Id&\Z&\Id&\Id&\Z&\Z&
 \Z\rtimes \Z/2\Z
 \\
 \Bir(X)&\Z\rtimes \Z/2\Z&
 \Id& \Id& 
 \Z&\Z&\Id&\Z&\Z&\Z&
 \Z\rtimes \Z/2\Z
 \\
  \hline
 \end{array}
 }
 $$
  
 When $e'= 3a^2-1 $, the pair  $( a,1)$ is a solution of $\cP_{3,e'}(1)$, but neither  $\cP_{3,e'}(-1)$ nor, when $a\not\equiv\pm1\pmod5$, $\cP_{3,4e'}(- 5)$ are solvable (reduce modulo 3 and 5).\ Therefore, we have $\Aut(X)= \Bir(X)\isom \Z\rtimes \Z/2\Z$.
\end{rema}

\section{Unexpected isomorphisms between special hyperk\"ahler fourfolds\\ and Hilbert squares of  K3 surfaces}\label{sect5}

The nef cone of the Hilbert square of a polarized K3 surface $(S,L)$  of degree $2e$ such that $\Pic(S)=\Z L$ was described in Example~\ref{exa217}: its extremal rays are spanned by $L_2$ and $L_2-\nu_e\delta $, where $\nu_e$ is a positive rational number that can be computed from the minimal solutions to the equations $\cP_e(1)$ or $\cP_{4e}(5)$.

 We can use this result to parametrize some of the special divisors ${}^2\cC_{2n,2e}^{(\gamma)}\subset {}^2\!\!\cM_{2n}^{(\gamma)}$.\ We let $\cK_{2e}$ be the quasiprojective 19-dimensional coarse moduli space  of polarized K3 surfaces of degree $2e$ (Theorem \ref{thm16}).

   \begin{prop} \label{thh}
Let $n$ and $e$  be positive integers and assume   that the equation $\cP_{e}(-n)$ has a   positive solution $(a,b)$   that satisfies the conditions
\begin{equation}\label{nue}
\frac{a}{b}<\nu_e\qquad{and}\qquad \gcd(a,b)=1.
\end{equation}
The rational map
\begin{eqnarray*}
\varpi\colon \cK_{2e} &\dra& {}^2\!\!\cM_{2n}^{(\gamma)}\\
(S,L)&\longmapsto&(S^{[2]}, bL_2-a\delta ),
\end{eqnarray*}
where the divisibility   $\gamma$ is $2$ if $b$ is even and $ 1$ if $b$ is odd, induces a birational isomorphism onto an irreducible component of $\,{}^2\cC_{2n,2e}^{(\gamma)}$.\ In particular, if $n$ is prime  and $b$ is even, it induces a birational isomorphism 
\begin{eqnarray*}
 \cK_{2e} &\isomdra& {}^2\cC_{2n,2e}^{(2)}.
\end{eqnarray*}
 \end{prop}

The case $n=3$ and $e=m^2+m+1 $ (one suitable solution is then $(a,b)=(2m+1,2)$) is a result of Hassett (\cite{has}; see also \cite{add2}).\  

\begin{proof}
  If $(S,L)$ is a   polarized K3 surface of degree $2e$ and $K:=\Z L_2\oplus \Z\delta \subset H^2(S^{[2]},\Z)$, the lattice
 $K^\bot
 $ is the orthogonal in $H^2(S,\Z)$ of the class $L$.\ Since the lattice $H^2(S,\Z)$ is unimodular, $K^\bot$ has discriminant $-2e$, hence $S^{[2]}$ is special of discriminant  $2e$.

The class $ H=bL_2-a\delta$ has   divisibility~$\gamma$ and square $2n$.\ It is   primitive, because $\gcd(a,b)=1$,  and, if $S$ is very general,   ample on $S^{[2]}$ because of   the inequality in~\eqref{nue}.\ Therefore, the pair $(S^{[2]},H)$ corresponds to a point of ${}^2\cC_{2n,2e}^{(\gamma)}$.
 
 The map $\varpi$ therefore sends   a very general point  of $\cK_{2e}$ to $\cC_{2n,2e}^{(\gamma)}$.\ To prove   that $\varpi$ is generically injective, we assume to the contrary that there is an isomorphism $\phi\colon S^{[2]}\isomto S^{\prime [2]}$ such that $\phi^*(bL'_2-a\delta')=bL_2-a\delta$, although  $(S,L)$ and $(S',L')$ are not isomorphic.\ It is straightforward to check that this implies $\phi^*\delta'\ne \delta$ and  that the extremal rays of the nef cone of $S^{[2]}$ are spanned by the primitive classes $L_2$ and $\phi^*L'_2$.\ Comparing this with the description of the nef cone given in Example~\ref{exa217}, we see that $e$ is not a perfect square, $\phi^*L'_2=a_1L_2-eb_1\delta$ and $\phi^*(a_1L'_2-eb_1\delta')=L_2$, where 
 $(a_1,b_1)$ is the minimal solution to the Pell equation $\cP_e(1)$.\ The same proof as that of Proposition~\ref{bbb2a} implies  $e>1$, the equation $\cP_{e}(-1)$ is solvable  and  the equation $\cP_{4e}(5)$ is not.

By Proposition~\ref{bbb2a} again, $S^{[2]}$ has a nontrivial involution $\sigma$ and $(\phi\circ\sigma)^*(L'_2)=L_2$ and $(\phi\circ\sigma)^*(\delta')=\delta$.\ This implies that $\phi\circ\sigma$ is induced by an isomorphism $(S,L) \isomto (S',L')$, which contradicts our hypothesis.\  The map $\varpi$ is therefore generically injective and since $\cK_{2e}$ is irreducible of  dimension~19, its image is a component of $\cC_{2n,2e}^{(\gamma)}$.\ When 
  $n$ is prime and $b$ is even, the conclusion follows from the irreducibility of 
${}^2\cC_{2n,2e}^{(2)}$ (Proposition~\ref{propirr}).
  \end{proof}

\begin{exam}
Assume $n=3$.\ For   $e=7$, one computes $\nu_{7}=\frac{21}{8}$.\ The only  positive  solution to the equation $\cP_{7}(-3)$ with $b$ even that satisfy \eqref{nue} is $(5,2)$.\ A general element of ${}^2\cC_{6,14}^{(2)}$ is therefore isomorphic to the Hilbert square of a polarized K3 surface of degree $14$.\ This is explained geometrically by the Beauville--Donagi construction (\cite{bedo}).
\end{exam}

\begin{exam}
Assume $n=3$.\ For   $e=13$, one computes $\nu_{13}=\frac{2340}{649}$.\ The only  positive  solutions to the equation $\cP_{13}(-3)$ with $b$ even that satisfy \eqref{nue} are $(7,2)$ and $(137,38)$.\ A general element $(X,H)$ of ${}^2\cC_{6,26}^{(2)}$ is therefore isomorphic to the Hilbert square of a polarized K3 surface $(S,L)$ of degree $26$ in two different ways: $H$ is mapped either to  $ 2L_2-7\delta$ or to $ 38L_2-137\delta$.\ This can be explained as follows:   the equation $\cP_{13}(-1)$ has a minimal solution  $(18,5)$ and the equation $\cP_{52}(5)$ is not solvable (reduce modulo 5), hence $S^{[2]}$ has a nontrivial involution $\sigma$ (Proposition~\ref{bbb2a}); one isomorphism $S^{[2]}\isomto X$ is obtained from the other by composing it with $\sigma$ (and indeed, $\sigma^*(2L_2-7\delta)=38L_2-137\delta$).\ This is a case where $X$ is the variety of lines on a cubic fourfold, the Hilbert square of a K3 surfaces, and a double EPW sextic! There is no geometric explanation for this remarkable fact (which also happens  for $e\in\{73,157\}$).
\end{exam}

\begin{rema}
The varieties $\cK_{2e}$ are known to be of general type for $e>61$  (\cite{ghs0}).\ Proposition~\ref{thh} implies that for any prime number $n$ satisfying its hypotheses, the Noether--Lefschetz divisors $ {}^2\cC_{2n,2e}^{(2)}$, which are irreducible by Proposition~\ref{propirr}(2), are also of general type.\ More precise results on the geometry of the varieties $ {}^2\cC_{6,2e}^{(2)}$ can be found in~\cite{nue,vas,lai} (they are known to be of general type for $e>96$ (\cite{vas}) and unirational for $e\le 19$ (\cite{nue})).
\end{rema}
   
  \begin{coro}
Let $n$ be a positive integer.

\noindent{\rm (1)} Inside the moduli space $   {}^2\!\!\cM_{2n}^{(1)}$, the  general points of some component of each of the  infinitely many distinct  hypersurfaces ${}^2\cC^{(1)}_{2n,2(a^2+n)}$, where $a$ describes the set of all positive integers  such that $(n,a)\ne (1,2)$,  correspond to  Hilbert squares of  K3 surfaces.\

\noindent{\rm (2)} Assume moreover $n\equiv -1\pmod4$.\ Inside the moduli space $  {}^2\!\!\cM_{2n}^{(2)}$, the  general points of some component of each of the  infinitely many distinct hypersurfaces ${}^2\cC^{(2)}_{2n,2\left(a^2+a+\tfrac{n+1}{4}\right)}$, where $a$ describes the set of all nonegative integers such that  $(n,a)\ne (3,1)$,  correspond to  Hilbert squares of  K3 surfaces.

In both cases, the union of these hypersurfaces is dense in the moduli space $  {}^2\!\!\cM_{2n}^{(\gamma)}$ for the euclidean topology.
\end{coro}

\begin{proof}
 For (1), the pair $(a,1)$ is a solution of the equation $\cP_{e}(-n)$, with $e=a^2+n$.\ We  need to check that the inequality  $a<\nu_e $ in \eqref{nue} holds.

If $e$ is a perfect square,  we have $\nu_e=\sqrt{e}$ and \eqref{nue} obviously holds.

 If $e$ is not a perfect square and  the equation $\cP_{4e}(5)$ is not solvable, we have $\nu_e=e\frac{b_1}{a_1}$ (Example~\ref{exa217}).\ If   the inequality~\eqref{nue} fails, since  $\nu_e^2=e^2\frac{b_1^2}{a_1^2}=e-\frac{e}{a_1^2}=a^2+n-
\frac{e}{a_1^2}
$, we have $a_1^2\le e/n$.\ Since $a_1^2=eb_1^2+1\ge e+1$, this is absurd and~\eqref{nue} holds in this case.

If the equation $\cP_{4e}(5)$ has a  minimal solution $(a_5 ,b_5 )$, we have $\nu_e=2e\frac{b_5 }{a_5 }$ (Example~\ref{exa217}).\ If   the inequality \eqref{nue} fails, we have  again  $\nu_e^2=4e^2\frac{b_5^2}{a_5^2}=a^2+n-\frac{5e}{a_5^2}\le a^2$, hence $a_5^2\le 5e/n$.\ Since $a_5^2=4eb_5^2+5 $, this is possible only if $n=b_5=1$, in which case $a_5^2=4e+5=4a^2+9 $.\ This implies $a_5>2a$, hence $a_5^2\ge 4a^2+4a+1$ and $a\le 2$.\ If $a=1$, the integer $4a^2+9 $ is not a perfect square.\ If $a=2$, we have $a_5=e=5$ and $a=\nu_e$, but this is a case that we have excluded.\ The  inequality~\eqref{nue} therefore holds in this case.

 For (2), the pair $(2a+1,2)$ is a solution of the equation $\cP_{e}(-n)$, with $e=a^2+a+\tfrac{n+1}{4}$ and we  need to check that the inequality $a+\tfrac12<\nu_e $ in~\eqref{nue}  holds.
 
 If $e$ is a perfect square,  we have $\nu_e=\sqrt{e}$ and \eqref{nue}   holds.
 
 If $e$ is not a perfect square and the equation $\cP_{4e}(5)$ is not solvable, we have $\nu_e=e\frac{b_1}{a_1}$.\ If   the inequality \eqref{nue} fails, since  $\nu_e^2=e^2\frac{b_1^2}{a_1^2}=e-\frac{e}{a_1^2}=a^2+a+\tfrac{n+1}{4}-
\frac{e}{a_1^2}
$, we have $a_1^2\le 4e/n$.\ Since $a_1^2=eb_1^2+1$ and $n\equiv -1\pmod4$, this is possible only if $n=3$ and $b_1=1$, in which case $a_1^2=a^2+a+2$.\ This implies $a_1>a$, hence $a_1^2\ge a^2+2a+1$ and $a\le 1$.\ If $a=0$, the integer $a^2+a+2$ is not a perfect square.\ If $a=1$, we have $a_1=2$, $e=3$, and $a=\nu_e$, but this is a case that we have excluded (and indeed, $ \cC^{(2)}_{6,6}$ is empty as noted in Example~\ref{exa222}).\  The  inequality~\eqref{nue} therefore  holds in this case.

If the equation $\cP_{4e}(5)$ has a   solution $(a_5 ,b_5 )$, we have $\nu_e=2e\frac{b_5 }{a_5 }$.\ If   the inequality \eqref{nue} fails, we have  again  $\nu_e^2=4e^2\frac{b_5^2}{a_5^2}=a^2+a+\tfrac{n+1}{4}-\frac{5e}{a_5^2}\le \bigl(a+\tfrac12\bigr)^2$, hence $a_5^2\le 20e/n$.\ Since $a_5^2=4eb_5^2+5 $, this is possible only if $n=3$ and $b_5=1$, in which case $a_5^2=4e+5=4a^2+4a+9 $.\ This implies $a_5>2a+1$, hence $a_5^2\ge 4a^2+8a+4$ and $a\le 1$.\ If $a=1$, the integer $4a^2+4a+9  $ is not a perfect square.\ If $a=0$, we have $a_5=3$, $e=1$, and $a=\nu_e=\tfrac23>a+\tfrac12$, so the  inequality \eqref{nue} always holds.

Finally, the density of the union of the special hypersurfaces in the moduli space follows from a powerful result of Clozel and Ullmo (Theorem \ref{thcu} below).
\end{proof}

\begin{theo}[Clozel--Ullmo]\label{thcu}
The union of infinitely many Heegner divisors in any moduli space ${}^m\!\!\cM^{(\gamma)}_{2n}$ is dense for the euclidean topology. 
\end{theo}

\begin{proof}
This follows from the main result of \cite{clul}: the space ${}^m\!\!\cM^{(\gamma)}_{2n}$  is a (union of components of a) Shimura variety and   each Heegner divisor $\cD_x$ is a ``strongly special'' subvariety, hence is endowed with a canonical probability measure $\mu_{\cD_x}$;.\ Given any infinite family $(\cD_{x_a})_{a\in\N}$ of Heegner divisors, there exists a subsequence $(a_k)_{k\in\N}$, a  strongly special subvariety $Z \subset {}^m\!\!\cM^{(\gamma)}_{2n}$ which contains  $\cD_{x_{a_k}}$ for all $k\gg 0$ such that $(\mu_{\cD_{x_{a_k}}})_{k\in\N}$ converges weakly to $\mu_Z$ (\cite[th.~1.2]{clul}).\ For dimensional reasons, we have $Z= {}^m\!\!\cM^{(\gamma)}_{2n}$; this implies that $\bigcup_a \cD_{x_a}$ is dense in    
${}^m\!\!\cM^{(\gamma)}_{2n}$.
\end{proof}

\begin{rema}
It was proved in \cite{mame} that Hilbert schemes of projective K3 surfaces are dense in the coarse moduli space of all (possibly nonalgebraic) \hKm s  of $\KKK^{[m]}$-type.
\end{rema}
   
   \appendix\section{Pell-type equations}\label{secpell}

We state or prove a few elementary results on some diophantine equations.

Given nonzero integers $e$ and $t$ with $e>0$, we denote by $\cP_e(t)$ the Pell-type equation
 \begin{equation}\label{pet}
a^2-eb^2=t,
\end{equation}
where $a$ and $b$ are integers (the usual Pell equation is the case $t=1$).\ A solution $(a,b)$ of this equation is called positive if $a>0$ and $b>0$.\ 
If $e$ is not a perfect square, $(a,b)$ is a solution if and only if the norm of $a+b\sqrt{e}$ in the quadratic number field $\Q(\sqrt{e})$ is $t$.

A positive solution with minimal $a$ is called   the minimal solution; it is also the positive solution $(a,b)$ for which the ``slope'' $b/a=\sqrt{\frac{1}{e }-\frac{t}{ea^2}}$ is minimal when $t>0$,   maximal when $t<0$.\ 
Since the function $x\mapsto  x+\frac{t}{x}$ is increasing on the interval $(\sqrt{|t|},+\infty)$, the minimal solution is also the one for which the real number $a+b\sqrt{e}$ is $ >\sqrt{|t|}$ and minimal.

{\em Assume that $e$ is not a perfect square.} There is always a minimal solution $(a_1,b_1)$ to the Pell equation $\cP_e(1)$  and if $x_1:=a_1+b_1\sqrt{e}$, all the   solutions of the equation $\cP_e(1)$  correspond to the ``powers'' $\pm x_1^n$, for $n\in \Z$, in $ \Z[\sqrt{e}]$.\ 

 If an equation $\cP_e(t)$ has a solution $(a,b)$, the elements $\pm (a+b\sqrt{e})x_1^n$ of   $\Z[\sqrt{e}]$, for $n\in\Z$,  all give rise to solutions of $\cP_e(t)$ which are said to be {\em associated with} $(a,b)$.\ The set of all solutions of $\cP_e(t)$ associated with each other form a {\em class of solutions.}\ A class and its conjugate (generated by $(a,-b)$) may be distinct or equal.
 
 Assume that $t$ is positive but not a perfect square.\ Let $(a,b)$ be a solution to the equation  $\cP_e(t)$ and set $x:=a+b\sqrt{e}$.\ If $x=\sqrt{t}$, we have $\bar x=t/x=x$, hence $b=0$; this contradicts our hypothesis that $t$ is not a perfect square, hence $x\ne \sqrt{t}$.\ 
 
 A class of solutions to the equation  $\cP_e(t)$ and its conjugate give rise to real numbers which are ordered as follows\footnote{\label{solsf}Since $x_1>1$, we have $0<x_tx_1^{-1}<x_t$.\ Since $x_t$ corresponds to a minimal solution, this implies $x_tx_1^{-1}< \sqrt{t}$, hence $\bar x_tx_1>\sqrt{t}$.\ By minimality of $x_t$ again, we get $\bar x_tx_1\ge x_t$.}
\begin{equation}\label{sols}
\cdots <x_tx_1^{-2}\le\bar x_tx_1^{-1}<x_tx_1^{-1}\le \bar x_t <\sqrt{t}<x_t\le \bar x_tx_1<x_tx_1\le \bar x_tx_1^2<x_tx_1^2<\cdots
\end{equation}
where $x_t=a_t+b_t\sqrt{e}$ corresponds to a solution which is minimal in its class and $\bar x_t$ is its conjugate.\ We have  $x_t= \bar x_tx_1$ if and only if the class of the solution $(a_t,b_t)$ is associated with its conjugate.\ The inequality $x_t\le \bar x_tx_1$ implies $a_t\le a_ta_1-eb_tb_1$, hence
\begin{equation}\label{slope}
\frac{b_t}{a_t}\le \frac{a_1-1}{eb_1}=\frac{a_1^2-a_1}{ea_1b_1}=\frac{b_1}{a_1}-\frac{a_1-1}{ea_1b_1}<   \frac{b_1}{a_1}.
\end{equation}
 This inequality between slopes also holds for the solution $\bar x_tx_1$ in the conjugate class (because $(\bar x_tx_1)x_1^{-1}=\bar x_t<\sqrt{t}$) but for no other positive solutions in these two classes.
 
 We will need the following variation on this theme.\  We still assume that    $t$ is positive and not a perfect square.\ 
 Let $(a'_t,b'_t)$ be the minimal positive solution to the equation $\cP_{4e}(t)$ (if it exists) and set $x'_t:= a'_t+b'_t\sqrt{4e}$.\footnote{If $b_t$ is even, we have $a'_t=a_t$, $b'_t=b_t/2$, and $x'_t=x_t$.\ If $t$ is even, we have $4\mid t$, $a'_t=2a_{t/4}$, $b'_t=b_{t/4}$, and $x'_t=2x_{t/4}$.\ If $b_t$ and $t$ are odd, 
 \begin{itemize}
\item either $b_1$ is even (and $a_1$ is odd) and the  second argument of any solution of $\cP_e(t)$ associated with $(a_t,b_t)$ or its conjugate remains odd, so that the solution $(a'_t,2b'_t)$ of $\cP_e(t)$ is   in neither of these two classes;
\item or $b_1$ is odd, we have $a_tb_1-a_1b_t\equiv a_t+a_1 \equiv t+e+1+e\equiv  0 \pmod2$, hence $x'_t=\bar x_tx_1$.
\end{itemize}
 }
 
 If $b_1$ is even, $(a_1,b_1/2)$ is the minimal solution to the equation $\cP_{4e}(1)$ and we obtain from \eqref{slope}  the inequality
 \begin{equation}\label{slope2}
\frac{b'_t}{a'_t}  <   \frac{b_1}{2a_1}.
\end{equation}
As above, the only other solution (among the   class  of $(a'_t,b'_t)$ and its conjugate) for which  this inequality between slopes also holds is the ``next'' solution, which corresponds to $\bar x'_tx_1$.

If $b_1$ is odd, $(a'_1,b'_1)= (2eb_1^2+1,a_1b_1)$ is the minimal solution to the equation $\cP_{4e}(1)$, so that $x'_1=x_1^2$.\ The  solutions associated with $(a'_t,b'_t)$ correspond to the $\pm x'_tx_1^{2n}$, $n\in\Z$.\ We still get 
 from \eqref{slope}  the inequality
 \begin{equation}\label{slope3}
\frac{b'_t}{a'_t}  \le\frac{a'_1-1}{4eb'_1}=\frac{(2eb_1^2+1)-1}{4ea_1b_1}=   \frac{b_1}{2a_1} 
\end{equation}
which is in fact strict.\footnote{Since $a_1$ and $b_1$ are relatively prime, the equality $2a_1b'_t=b_1a'_t$ implies that there exists a positive integer $c$ such that $2b'_t=cb_1$ and $a'_t=ca_1$; plugging these values into the equation $\cP_{4e}(t)$, we get $t=c^2$, which contradicts our hypothesis that $t$ is not a perfect square.}\ 
No other  solution associated with $(a'_t,b'_t)$ or its conjugate  satisfies  this inequality between slopes.


The following criterion   (\cite[Theorem~110]{nage})   ensures that in some cases, any two classes of solutions are conjugate, so  that the discussion  above applies to all   solutions.

 \begin{lemm}\label{le46}
Let $u$ be a positive  integer which is either prime or equal to 1,  let $e $ be a positive integer which is not a perfect square, and let
 $\eps\in\{-1,1\}$.\
If the equation $\cP_e(\eps u)$ is solvable, it has one or two classes of solutions according to whether $u$ divides $2e $ or not; if there are two classes, they are conjugate.
\end{lemm}


We now  extend slightly the class of equations that we are considering: if $e_1$ and $e_2$ are positive integers, we denote by $\cP_{e_1,e_2}(t)$ the equation
\[
e_1 a^2 - e_2 b^2 =t.
\]
Given an integral solution $(a,b)$ to $\cP_{e_1,e_2}(t)$, we obtain a solution $(e_1a,b)$ to $\cP_{e_1e_2}(e_1t)$.\ If $e_1$ is square-free, all the solutions to $\cP_{e_1e_2}(e_1t)$ arise  in this way; in general, all the solutions whose first argument is divisible by $e_1$ arise.\ 
A positive solution $(a,b)$ to $\cP_{e_1,e_2}(t)$ is called minimal if $a$ is minimal.\ 
If $e_1e_2$ is not a perfect square, we say that the solutions $(a,b)$ and $(a',b')$ of $\cP_{e_1,e_2}(t)$ are associated if $(e_1a,b)$ and $(e_1a',b')$ are associated solutions of $\cP_{e_1e_2}(e_1t)$.\footnote{If $(a,b)$ is a solution to $\cP_{e_1,e_2}(t)$ and $(a_1,b_1)$ is a solution to $\cP_{e_1e_2}(1)$, and if we set $x_1:=a_1+b_1\sqrt{e_1e_2}$, then $(e_1a+b\sqrt{e_1e_2})x_1=:e_1a'+b'\sqrt{e_1e_2}$, where $(a',b')$ is again a solution to $\cP_{e_1,e_2}(t)$.}

 Let  $\eps\in\{-1,1\}$; assume that the equation $\cP_{e_1,e_2}(\eps)$ has a solution $(a_\eps,b_\eps)$ and set 
 $x_\eps:= e_1a_\eps+b_\eps\sqrt{e_1e_2}$.\ Let $(a,b)\in\Z^2$ and set $x:=e_1a+b\sqrt{e_1e_2}$.\ We have
 $$xx_\eps=e_1(e_1aa_\eps+e_2bb_\eps +(a_\eps b+ab_\eps)\sqrt{e_1e_2} )=:e_1y,
 $$ 
 where $x\bar x=e_1\eps y\bar y$.\ In particular, $x$ is a solution to the equation $\cP_{e_1 e_2}(e_1t)$ if and only if 
 $y$ is a solution to the equation $\cP_{e_1 e_2}( \eps t)$.\ This defines a bijection between the set of solutions to 
the equation $\cP_{e_1, e_2}(t)$ and the set of solutions to 
the equation $\cP_{e_1 e_2}(\eps t)$ (the inverse bijection is given by $y\mapsto x=\eps \bar x_\eps y$).\ 

The proof of the following lemma is left to the reader.


  \begin{lemm}\label{lemmpell}
 Let $e_1$ and $e_2$ be positive integers.\ Assume that for some $\eps\in\{-1,1\}$, the equation $\cP_{e_1,e_2}(\eps )$ is solvable and let $(a_\eps,b_\eps)$ be its minimal solution.\ Then,  $e_1e_2$ is not a perfect square and the  minimal solution of the equation $\cP_{e_1e_2}(1)$ is $(e_1a_\eps^{2}+e_2b_\eps^2,2a_\eps b_\eps)$, unless 
 $e_1= \eps= 1$ or
 $e_2=-\eps=1$, 
 \end{lemm}

We now assume $t<0$ (the discussion is entirely analogous when $t>0$ and leads to the reverse inequality in \eqref{sols3}) and $-t$ is not a perfect square.\ Let $(a_t,b_t)$ be the minimal solution to the equation  $\cP_{e_1,e_2}(t )$ and set    $x_t:= e_1a_t+b_t\sqrt{e_1e_2}$, solution to the equation 
$\cP_{e_1 e_2}(e_1t)$ which is minimal among all solutions whose first argument is divisible by $e_1$.\ We have as in~\eqref{slope} the inequalities
\begin{equation}\label{sols2}
\cdots  <x_tx_1^{-1}\le -\bar x_t <\sqrt{-e_1t}<x_t\le -\bar x_tx_1<x_tx_1\le  \cdots
\end{equation}
with $x_1:= \frac{1}{e_1}x_\eps^2$ by Lemma \ref{lemmpell}.\ The increasing correspondence with the solutions to the equation $\cP_{e_1 e_2}(\eps t)$ that we described above maps 
 $x_tx_1^{-1}$ to $\frac{1}{e_1}x_tx_1^{-1}x_\eps=x_tx_\eps^{-1}$ and  $-\bar x_t$ to $-\frac{1}{e_1}\bar x_tx_\eps $.\ Since the product of these positive numbers is $-t $, we have 
\begin{equation} \label{sols22}
\cdots  \le x_tx_\eps^{-1} <\sqrt{-t}< -\frac{1}{e_1}\bar x_tx_\eps\le  \cdots
\end{equation}
In particular, $-\frac{1}{e_1}\bar x_tx_\eps $ corresponds to a positive solution, hence  
\begin{equation}\label{sols3}
\frac{a_t}{b_t} < \frac{a_\eps}{b_\eps}
.\end{equation}
Moreover, $(a_t,b_t)$ is the only positive solution to the equation  $\cP_{e_1,e_2}(t )$ among those appearing in \eqref{sols2} that satisfies this inequality.\ When $|t|$ is prime, Lemma \ref{le46} implies that any two classes of solutions are conjugate, hence all positive solutions appear in \eqref{sols2}, and \eqref{sols3} holds for only one positive solution  to the equation  $\cP_{e_1,e_2}(t )$.

Finally, one small variation: we assume that  $(a_\eps,b_\eps)$ is the  minimal solution  to the equation $\cP_{e_1,e_2}(\eps )$ but that  $(a_t,b_t)$ is the  minimal solution  to the equation $\cP_{e_1,4e_2}(t )$.\ Then we have 
\begin{equation}\label{sols4}
\frac{a_t}{2b_t} < \frac{a_\eps}{b_\eps}
 \end{equation}
and, when $|t|$ is prime, $(a_t,b_t)$ is the only solution that satisfies that inequality.\ The proof is exactly the same: since $2a_\eps b_\eps$ is even, $x_1$ still corresponds to the minimal solution to the equation $\cP_{4e_1 e_2}(1)$; in \eqref{sols2}, we have  solutions to the equation $\cP_{4e_1 e_2}(e_1t)$ and in \eqref{sols22}, we have solutions to the equation $\cP_{e_1 e_2}(\eps t)$, but this does not change the reasoning.

   \section{The image of the period map (with E.~Macr\`i)}\label{imagep}

In this second appendix, we generalize Theorem \ref{imper} in all dimensions.\ Recall the set up:  a polarization type $\tau$ is the $O(\Lambda_{\KKK^{[m]}})$-orbit of a primitive element $h_\tau$  of $ \Lambda_{\KKK^{[m]}}$ with positive square and   ${}^m\!\cM_\tau$ is the moduli space for \hKm s of $ \KKK^{[m]}$-type with a polarization of type $\tau$.\ There is a period map
\[
\wp_\tau\colon
 {}^m\!\!\cM_\tau\lra\cP_\tau=\widehat O(\Lkkk[m] ,h_\tau)\backslash \Omega_{h_\tau}.
\]

The goal of this appendix is to prove the following result.

\begin{theo}[Bayer, Debarre--Macr\`i, Amerik--Verbitsky]\label{thm:ImagePeriodMap}
Assume $m\ge 2$.\ Let $\tau$ be a polarization type.\ The image of the restriction of the period map $\wp_{\tau}$ to any component of the moduli space
 ${}^m\!\cM_\tau$  is the complement of a finite union of explicit Heegner divisors.
\end{theo}

The procedure for listing these divisors will be explained in Remark~\ref{explicit}.\ The proof of the theorem is based on the description of the nef cone for hyperk\"ahler manifolds of type $\KKK^{[m]}$  in \cite{bht,mon2} (another proof follows from \cite{amve}).\ We start by revisiting these results.\ Recall (see \eqref{mml}) that  the (unimodular) extended K3 lattice is
  \[
 \widetilde{\Lambda}_{\KKK}:= U^{\oplus 4}\oplus E_8(-1)^{\oplus 2}.
\]
Let $X$ be   a projective hyperk\"ahler manifold  of $K3^{[m]}$-type.\ By \cite[Corollary 9.5]{marsur},   there is a canonical $O( \widetilde{\Lambda}_{\KKK})$-orbit $[\theta_X]$ of primitive isometric embeddings
\[
\theta_X\colon H^2(X,\Z)  \xhookrightarrow{\ \ \ }   \widetilde{\Lambda}_{\KKK}.
\]
We denote by $\bv_X$ a generator of the orthogonal of $\theta_X(H^2(X,\Z))$ in $ \widetilde{\Lambda}_{\KKK}$.\ It satisfies $ \bv_X^2=2m-2$.\ We endow $\widetilde{\Lambda}_{\KKK}$ with the weight-2 Hodge structure $\widetilde\Lambda_X$ for which $\theta_X$ is a morphism of Hodge structures and $\bv_X$ is of type $(1,1)$, and we set
\[
\widetilde{\Lambda}_{\textnormal{alg},X} := 
\widetilde\Lambda_X \cap \widetilde\Lambda_X^{ 1,1} \subset \widetilde{\Lambda}_{\KKK}
,
\]
so that $\NS(X)=\theta_X^{-1}(\widetilde{\Lambda}_{\textnormal{alg},X})$.\ Finally, we set
\[
\cS_X :=  \{ \bs\in \widetilde{\Lambda}_{\textnormal{alg},X}\mid  \bs^2\ge -2 \text{ and } 0\leq  \bs \cdot\bv_X\le  \bv_X^2/2  \}\moins \{0\}.
\]

The hyperplanes $\theta_X^{-1}(\bs^\perp)\subset \NS(X)\otimes\R$, for $\bs\in\cS_X$, are locally finite in the positive cone $\Pos(X)$.\ The dual statement of \cite[Theorem 1]{bht} is then the following (see also \cite[Theorem 12.1]{bama}).\

\begin{theo}\label{thm:BHT}
Let $X$ be a projective hyperk\"ahler manifold  of $K3^{[m]}$-type.\ The ample cone of $X$ is the connected component of
\begin{equation}\label{set}
\Pos(X) \setminus \bigcup_{\bs\in\cS_X} \theta_X^{-1}(\bs^\perp)
\end{equation}
that contains the class of an ample divisor.
\end{theo}

Note that changing $\bv_X$ into $-\bv_X$ changes $\cS_X$ into $-\cS_X$, but the
set in \eqref{set} remains the same. 

 \begin{exam}\label{ehm}
Let $(S,L)$ be a polarized K3 surface of degree $2e$ and let $r$ be a positive integer.\ Recall from Remark~\ref{rema45} that the moduli space $X:=\cM(r,L,r)$, when smooth,  is a \hKm\ of type $\KKK^{[m]}$, where $m:=1-r^2+e$, and carries a class $H$ of square 2.\ In fact, the vector $\bv_X$ above is $(r,-r,L)\in U\oplus  {\Lambda}_{\KKK}=\widetilde{\Lambda}_{\KKK}$ and $H$ is $\theta_X^{-1}((1,1,0))$.\ Consider the vector $\bs:= (-1,1,0)\in \widetilde{\Lambda}_{\textnormal{alg},X}$.\ We have $\bs^2=-2$, $\bs\cdot\bv_X=2r $,  and $H\in \theta_X^{-1}(\bs^\perp)$.\ When $2r\le \bv_X^2/2= m-1$, that is, when $ e\ge r^2+2r$, the class $\bs$ is in $\cS_X$, hence, by the theorem, $H$ is {\em not} ample on $X$.

Assume now $\Pic(S)=\Z L$, so that the Picard number of $X$ is 2.\ 
When    
$r^2<e\le r^2+2r-1$, the class $H$ is ample on  $X$     (\cite[Corollary~4.15]{og8} and Remark~\ref{rema45}).

It can actually be shown that for all $e>r^2$, the class $H$ is nef on $X$.\footnote{\label{nf2}One checks that
for all $\alpha^2 > 1/e$, all stable sheaves in  $X$ are   stable for the Bridgeland stability condition $\sigma_{\alpha,0}$ (in Bridgeland's notation).\ Therefore, by \cite[Lemma~9.2]{bama1}, the class $H$ is nef: it  corresponds exactly to the (nonexistent) stability condition at $\alpha^2=1/e$.}\ However, we do not know    the  nef cone    except when $r=1$, where $X=S^{[e]}$ and everything is described in Example~\ref{exa37} (see also Exercise~\ref{er2} for the case $e=r^2+1$).\ One can only say that
\begin{itemize}
\item when $r^2<e\le r^2+2r-1$, the manifold $X$ carries a biregular nontrivial involution (Remark~\ref{rema45}) that interchanges the two extremal rays of $\Nef(X)$;\footnote{We do not know in general whether these rays are rational.}
\item when $ e\ge r^2+2r$, one  ray is spanned by $H$, the other is rational by \cite[Theorem~1.3(1)]{ogu2}.
\end{itemize}

\end{exam}

We  rewrite Theorem \ref{thm:BHT} in terms of the existence of certain rank-2  lattices in the N\'eron-Severi group as follows.

\begin{prop}\label{prop:BHTnew}
Let $m\geq2$.\
Let $X$ be a projective hyperk\"ahler manifold of $K3^{[m]}$-type and let $h_X\in\NS(X)$ be a primitive class such that $h_X^2>0$.\
The following conditions are equivalent: 
\begin{enumerate}
\item[\textnormal{(i)}]
there exist a projective hyperk\"ahler manifold $Y$ of $K3^{[m]}$-type, an ample primitive class $h_Y\in\NS(Y)$, and a Hodge isometry $g\colon H^2(X,\Z) \isomto H^2(Y,\Z)$ such that $[\theta_X]=[\theta_Y\circ g]$ and $g(h_X)=h_Y$;
\item[\textnormal{(ii)}]
there are no 
 rank-2 sublattices $L_X\subset \NS(X)$ such that
\begin{itemize}
\item $h_X\in L_X$, and
\item there exist integers $0\leq k \leq m-1$ and $a\geq -1$,   
and $\kappa_X\in L_X$, such that
\[
\kappa_X^2 = 2(m-1)   ( 4 (m-1) a -k^2 ), \qquad  \kappa_X\cdot h_X =0, \qquad \frac{\theta_X(\kappa_X) + k \bv_X}{2(m-1)} \in \widetilde{\Lambda}_{\KKK}.
\]
\end{itemize}
\end{enumerate}
\end{prop}

Again, if one changes  $\bv_X$ into $-\bv_X$, one needs to take     $-\kappa_X$ instead of $ \kappa_X$.\

\begin{proof}
Assume   that a pair $(Y,h_Y)$ satisfies (i)     but that there is a lattice $L_X$ as in (ii).\
The rank-2 sublattice $L_Y:=g(L_X)\subset H^2(Y,\Z)$ is contained in $ \NS(Y)$.\
We also let $\kappa_Y:=g(\kappa_X)\in L_Y$ and choose the generator $\bv_Y:=\bv_X\in\widetilde{\Lambda}_{\KKK}$ of $H^2(Y,\Z)^\perp$.\ We then have  
\[
\kappa_Y^2 = 2(m-1) ( 4 (m-1) a -k^2 ), \qquad  \kappa_Y\cdot h_Y =0, \qquad \bs_Y:=\frac{\theta_Y(\kappa_Y) + k \bv_Y}{2(m-1)} \in \widetilde{\Lambda}_{\KKK} 
\]
and 
\[
\bs_Y^2=2a\geq -2,\qquad  
0\leq  \bs_Y\cdot \bv_Y =k\leq  m-1 =\frac{\bv_Y^2}{2} ,\qquad     \bs_Y\cdot \theta_Y(h_Y) =0.
\]
Moreover, we have   $\bs_Y\in\widetilde{\Lambda}_{\textnormal{alg},Y}$.\
By Theorem \ref{thm:BHT}, $h_Y$ cannot be ample on $Y$, a contradiction.

Conversely, assume that there exist  no lattices $L_X$ as in (ii).\
By \cite[Lemma 6.22]{marsur}, there exists a Hodge isometry $g'\colon H^2(X,\Z)\isomto H^2(X,\Z)$ such that $[\theta_X\circ g']=[\theta_X]$ and $g'(h_X)\in\overline{\Mov}(X)$.\
Since $g'(h_X)\in\NS (X)$ and $g'(h_X)^2>0$, by \cite[Theorem~6.17 and Lemma~6.22]{marsur} and \cite[Theorem 7]{hast2}, there exist  a projective hyperk\"ahler manifold $Y$ of $K3^{[m]}$-type, a nef divisor class $h_Y\in\Nef(Y)$, and a Hodge isometry $g''\colon H^2(X,\Z) \isomto H^2(Y,\Z)$ such that $[\theta_X]=[\theta_Y\circ g'']$ and $g''(g'(h_X))=h_Y$.\
Assume that $h_Y$ is not ample.\ Since $h_Y^2>0$ and $h_Y$ is nef,  there exists by Theorem \ref{thm:BHT} a class $\bs_Y\in\cS_Y$ such that  
$\bs_Y\cdot \theta_Y(h_Y)=0$.\ 
We set $a:=\frac12\bs_Y^2$, $k:= \bs_Y\cdot \bv_Y$, and
\[
\kappa_Y:= \theta_Y^{-1}\left( 2(m-1)  \bs_Y - k \bv_Y \right) \in \theta_Y^{-1}
  (  \widetilde{\Lambda}_{\textnormal{alg},Y} ) = \NS(Y).\]
Let $L_Y$ be the 
 sublattice of $H^2(Y,\Z)$ generated by $h_Y$ and $\kappa_Y$.\ 
Then $L_X:=g^{-1}(L_Y)$   satisfies the conditions in (ii), a contradiction.
\end{proof}

When $m-1$ is a prime number, this can be written only in terms of $H^2(X,\Z)$.

\begin{prop}\label{prop:BHTnewprime}
Assume that   $p:=m-1 $  is either 1 or   a prime number.\ Let $X$ be a projective hyperk\"ahler manifold of $K3^{[m]}$-type and let $h_X\in\NS(X)$ be a primitive class such that $h_X^2>0$.\  The following conditions are equivalent: 
\begin{enumerate}
\item[\textnormal{(i)}]
there exist a projective hyperk\"ahler manifold $Y$ of $K3^{[m]}$-type, an ample primitive class $h_Y\in\NS(Y)$, and a Hodge isometry $g\colon H^2(X,\Z) \isomto H^2(Y,\Z)$ such that $g(h_X)=h_Y$; 
\item[\textnormal{(ii)}]
there are no 
rank-2 sublattices $L_X\subset \NS(X)$ such that
\begin{itemize}
\item $h_X\in L_X$, and
\item there exist integers $0\leq k \leq p$ and $a\geq -1$,
and $\kappa_X\in L_X$, such that
\[
\kappa_X^2 = 2p   ( 4 p a -k^2 ), \qquad  \kappa_X\cdot h_X =0, \qquad  2p \mid \div_{H^2(X,\Z)}(\kappa_X).
\]
\end{itemize}
\end{enumerate}
\end{prop}

\begin{proof}
As explained in \cite[Sections~9.1.1 and~9.1.2]{marsur}, under our assumption on $m$, there is a unique $O(\widetilde{\Lambda}_{\KKK})$-orbit of primitive isometric embeddings $ \Lambda_{\KKK^{[m]}}\hra\widetilde{\Lambda}_{\KKK}$.\
This implies that any Hodge isometry $  H^2(X,\Z) \isomto H^2(Y,\Z)$   commutes with the orbits $[\theta_X]$ and $[\theta_Y]$.\

We can choose the embedding $\theta\colon\Lambda_{\KKK^{[m]}}\hra\widetilde{\Lambda}_{\KKK}$ as follows.\ Let us fix a canonical basis $(u,v)$ of a   hyperbolic plane $U$ in $\widetilde{\Lambda}_{\KKK}$.\ We   define $\theta $ by mapping the generator $\ell$ of $I_1(-2p)$ to $u-pv$.\ We also choose $\bv:=u+pv$ as the generator of $\Lambda_{\KKK^{[m]}}^\perp$.

By Proposition \ref{prop:BHTnew}, we only have to show that given a class $\kappa\in\Lambda_{\KKK^{[m]}}$ with divisibility $2p$ and such that $\kappa^2 = 2p ( 4p a -k^2 )$, either $ \theta(\kappa)+ k\bv$ or $\theta(-\kappa)+ k\bv$ is divisible by $2p$ in $\widetilde{\Lambda}_{\KKK}$.\ Since $\div_{H^2(X,\Z)}(\kappa)$ is divisible by $2p$, we can write $\kappa=2pw+r\ell$, with $w\in U^{\oplus 3}\oplus E_8(-1)^{\oplus 2}$ and $r\in\Z$.\ By computing $\kappa^2$, we obtain the equality
\[
r^2-k^2 = 2p(w^2-2a).
 \]
In particular, $r^2-k^2$ is even, hence so are $ r + k $ and $ -r + k $.\ Moreover, $p$   divides $r^2-k^2$, hence also   $ \eps r + k $, for some $\eps\in\{-1,1\}$.\   Then
\[
  \theta(\eps \kappa) + k\bv = \eps2p\theta(w) + (\eps r + k)u + (-\eps r + k)p v 
\]
is divisible by $2p$, as we wanted.
\end{proof}

Before proving the theorem, we  briefly review polarized marked hyperk\"ahler manifolds of $\KKK^{[m]}$-type, their moduli spaces,  and their periods,   following the presentation in \cite[Section~7]{marsur} (see also \cite[Section~1]{apo}).\

Let $ {}^m\mathfrak{M}$ be the (smooth, nonHausdorff, $21$-dimensional) coarse moduli space of marked hyperk\"ahler manifolds of $\KKK^{[m]}$-type (see \cite[Section~6.3.3]{huyinv}), consisting of isomorphisms classes of pairs $(X,\eta)$ such that $X$ is a compact complex (not necessarily projective) hyperk\"ahler manifold of $\KKK^{[m]}$-type and $\eta\colon H^2(X,\Z)\isomto \Lambda_{\KKK^{[m]}}$ is an isometry.\ We let $({}^m\mathfrak{M}^{\,t})_{t\in \mathfrak{t}}$ be the family of connected components of~${}^m\mathfrak{M}$.\ By \cite[Lemma~7.5]{marsur}, the set $\mathfrak{t}$ is finite and is acted on transitively by the group $O(\Lambda_{\KKK^{[m]}})$.\

Let $h\in\Lambda_{\KKK^{[m]}}$ be a class with $h^2>0$ and let $t\in\mathfrak{t}$.\
Let ${}^m\mathfrak{M}^{\,t,+}_{h^\perp}\subset{}^m\mathfrak{M}^{\,t}$ be the subset parametrizing the pairs $(X,\eta)$ for which the class $\eta^{-1}(h)$ is of Hodge type $(1,1)$ and  belongs to the positive cone of $X$, and let $ {}^m\mathfrak{M}^{\,t,a}_{h^\perp}\subset {}^m\mathfrak{M}^{\,t,+}_{h^\perp}$ be the  open subset where $\eta^{-1}(h)$ is ample on~$X$.\ By \cite[Corollary~7.3]{marsur},  ${}^m\mathfrak{M}^{\,t,a}_{h^\perp}$ is connected, Hausdorff, and $20$-dimensional.

The relation with our moduli spaces ${}^m\!\cM_\tau$ is as follows.\ Let $\tau$ be a polarization type and let $M$ be an irreducible component of ${}^m\!\cM_\tau$.\
Pick a point $(X_0,H_0)$ of $ M$ and choose a marking $\eta_0\colon H^2(X_0,\Z)\isomto \Lambda_{\KKK^{[m]}}$.\ If $h_0:=\eta_0( H_0 )$, the   $O(\Lambda_{\KKK^{[m]}})$-orbit of $h_0$ is the  polarization type $\tau$, and the pair $(X_0,\eta_0)$ is in ${}^m\mathfrak{M}^{\,t_0,a}_{h_0^\perp}$, for some $t_0\in\mathfrak{t}$.\ As in \cite[(7.4)]{marsur}, we consider the disjoint union
\[
{}^m\mathfrak{M}^{\,a}_{\overline h_0 } := \hskip-3mm\coprod_{(h,t)\in O(\Lambda_{\KKK^{[m]}})\cdot (h_0,t_0)} \hskip-7mm{}^m\mathfrak{M}^{\,t,a}_{h^\perp}.
\]
It is acted on by  $O(\Lambda_{\KKK^{[m]}})$ and, by \cite[Lemma~8.3]{marsur},  there is an analytic bijection 
\begin{equation}\label{bijj}
M\isomlra {}^m\mathfrak{M}^{\,a}_{\overline h_0 }/O(\Lambda_{\KKK^{[m]}}).
\end{equation}

\begin{proof}[Proof of Theorem \ref{thm:ImagePeriodMap}.]
Let $M$ be a irreducible component of ${}^m\!\cM_\tau$.\
Pick a point $(X_0,H_0)$ of $ M$ and choose a marking $\eta_0\colon H^2(X_0,\Z)\isomto \Lambda_{\KKK^{[m]}}$.\ As above, set $h_0:=\eta_0( H_0 )\in\Lambda_{\KKK^{[m]}}$ and let $t_0\in\mathfrak{t}$ be such that the pair $(X_0,\eta_0)$ is in ${}^m\mathfrak{M}^{\,t_0,a}_{h_0^\perp}$.\ Given another point $(X ,H )$ of $M$, by using the bijection \eqref{bijj}, we can always find a marking $\eta \colon H^2(X ,\Z)\isomto \Lkkk[m]$ for which $(X , \eta )$ is in ${}^m\mathfrak{M}^{\,t_0,a}_{h_0^\perp}$ and $H =\eta^{-1}(h_0)$.  

By \cite[Corollary~9.10]{marsur}, for all $(X,\eta)\in{}^m\mathfrak{M}^{\,t_0}$, the  primitive embeddings 
$$\theta_X\circ\eta^{-1}\colon\Lkkk[m]\hra\widetilde{\Lambda}_{\KKK}$$
 are all in the same $O(\widetilde{\Lambda}_{\KKK})$-orbit.\   We fix one such  embedding $\theta$ and  a generator $\bv\in\widetilde{\Lambda}_{\KKK}$ of $\theta(\Lambda_{\KKK^{[m]}})^\perp$.\ As in Proposition~\ref{prop:BHTnew}, we consider the Heegner divisors  $ \cD_{\tau,K}$
in $ \cP_\tau:= \widehat O(\Lkkk[m] ,h_0)\backslash \Omega_{\tau}$, where $K$ is a primitive, rank-2, signature-$(1,1)$, sublattice of $\Lkkk[m]$ such that $h_0\in K$ and there exist integers $0\leq k \leq m-1$ and $a\geq -1$, and $\kappa\in K$ with
\begin{equation}\label{nc}
\kappa^2 = 2(m-1)   ( 4 (m-1) a -k^2 ), \qquad  \kappa\cdot h_0 =0, \qquad \frac{\theta(\kappa) + k \bv}{2(m-1)} \in \widetilde{\Lambda}_{\KKK}.
\end{equation}

There are finitely many such divisors.\ 
Indeed, since the signature of $K$ is $(1,1)$, we have $\kappa^2<0$.\ Moreover, we have $\kappa^2\ge 2(m-1)   ( -4 (m-1)   -(m-1)^2 )$, hence $\kappa^2$ may take only finitely many values.\ As explained in the proof of Lemma~\ref{le324},
  this implies that there   only finitely many Heegner divisors of the above form.

We claim that the image of the period map  
\[
\wp_\tau\colon
M\lra \cP_\tau 
\]
coincides with the complement of the union of these Heegner divisors.\
We first show the image does not meet these divisors.\
Let $(X,H)\in M$ and, as explained above, choose a marking $\eta$ such that $(X,\eta)\in{}^m\mathfrak{M}^{\,t_0,a}_{h_0^\perp}$, where $h_0:=\eta(H)$.\
If $(X,H)\in  \cD_{\tau,K}$, for $K$ as above, the lattice $K\subset \NS(X)$   satisfies   condition (ii) of Proposition \ref{prop:BHTnew}, which is impossible.

Conversely, take a point $x\in \cP_{\tau}$.\
The refined period map defined in \cite[(7.3)]{marsur} is surjective (this is a consequence of \cite[Theorem~8.1]{huyinv}).
Hence there exists $(X,\eta_X)\in{}^m\mathfrak{M}^{\,t_0,+}_{h_0^\perp}$ such that $H_X:=\eta_X^{-1}(h_0)$ is an algebraic class in the positive cone of $X$ and $(X,H_X)$ has period point $x$.\ 
By~\cite[Theorem~3.11]{huyinv}, $X$ is projective.\
We can now apply Proposition \ref{prop:BHTnew}: if $x$ is outside the union of  the Heegner divisors described above, there exist a projective hyperk\"ahler manifold $Y$ of $K3^{[m]}$-type, an ample primitive class $H_Y$, and a Hodge isometry $g\colon H^2(X,\Z) \isomto H^2(Y,\Z)$ such that $[\theta_X]=[\theta_Y\circ g]$ and $g(H_X)=H_Y$.\ 
By \cite[Theorem 9.8]{marsur}, $g$ is a \emph{parallel transport operator;}  
by \cite[Definition 1.1(1)]{marsur}, this means that if $\eta_Y:=\eta\circ g^{-1}$, the pair $(Y,\eta_Y)$ belongs to the same   connected component ${}^m\mathfrak{M}^{\,t_0}$ of ${}^m\mathfrak{M}$.\
 Moreover, $(Y,\eta_Y )$ is in $ {}^m\mathfrak{M}^{\,t_0,+}_{h_0^\perp}$ and since $\eta_Y^{-1}(h_0)=H_Y$ is ample,  it is even in ${}^m\mathfrak{M}^{\,t_0,a}_{h_0^\perp}$.\ 
By \eqref{bijj}, it defines a point of $M$ which  still has  period point $x$.\ This means that $x$ is in the image of $\cP_\tau$, which is what we wanted.
\end{proof}

\begin{rema}\label{explicit}
 Let us explain how to list the Heegner divisors referred to in the statement of Theorem~\ref{thm:ImagePeriodMap}.\  Fix a  representative  $h_\tau\in \Lambda_{\KKK^{[m]}}$  of the polarization $\tau$.\ For each $O(\widetilde{\Lambda}_{\KKK})$-equivalence class of   primitive embeddings   $  \Lambda_{\KKK^{[m]}}\hra \widetilde{\Lambda}_{\KKK}$,\footnote{When $m-1$ is prime or equal to $1$, there is a unique equivalence class; in general, the number of equivalence classes is the index of $ \widehat O(\Lkkk[m]  )$ in $  O^+(\Lkkk[m])$, that is $2^{\max\{\rho(m-1)-1,0\}}$ (Section~\ref{sec27} or \cite[Lemma~9.4]{marsur}).}  pick a representative $\theta$ and a generator $\bv$ of $\theta(\Lambda_{\KKK^{[m]}})^\perp$.\ Let $T$ be the saturation in $\widetilde{\Lambda}_{\KKK}$ of the sublattice generated by $\theta(h_\tau)$ and $\bv$.\ 
By \cite[Theorem 2.1]{apo}, the abstract isometry class  of the pair $(T,\theta(h_\tau))$ determines a component $M$ of ${}^m\!\cM_\tau$ (and all the components are obtained in this fashion).\ Now list, for all integers $0\leq k \leq m-1$ and $a\geq -1$, all $ \widehat O(\Lkkk[m] ,h_\tau)$-orbits of primitive, rank-2, signature-$(1,1)$, sublattices $K$ of $\Lkkk[m]$ such that $h_\tau\in K$ and there exist   $\kappa\in K$ satisfying~\eqref{nc}.

The image   $\wp_\tau(M)$ is then the complement in $\cP_\tau$ of the union of the corresponding Heegner divisors $\cD_{\tau,K}$.\ The whole procedure is worked out in a particular case in Example~\ref{ex:LLSvS} below (there are more examples  in Section~\ref{sec210e} in the case $m=2$). 

{\it A priori,}
the image of $\wp_{\tau}$ may be different when restricted to different components $M$ of ${}^m\!\cM_\tau$. 
\end{rema}

\begin{exam}\label{ex:LLSvS}
 The moduli space ${}^4\!\cM_2^{(2)}$ for hyperk\"ahler manifolds of $\KKK^{[4]}$-type   with a polarization of square $2$ and divisibility $2$ is irreducible (Theorem~\ref{ghsa}).\ Let us show that the image of the period map ${^4}\wp_{2}^{(2)}$ is the complement of the three irreducible Heegner divisors ${^4}\cD_{2,2}^{(2)}$, ${^4}\cD_{2,6}^{(2)}$, and ${^4}\cD_{2,8}^{(2)}$.

We begin with the more general case where $m$ is divisible by 4 and $p:=m-1$ is an odd prime number.\ We follow the recipe given  in Remark~\ref{explicit}.\ The irreducibility of ${}^m\!\cM_2^{(2)}$ (Theorem~\ref{ghsa}) means that we can fix the embedding $\theta\colon\Lkkk[4]\hra\widetilde{\Lambda}_{\KKK}$ and the class $h=h_\tau\in\Lkkk[4]$ as we like.\ 
Let us fix bases $(u_i,v_i)$, for $i\in\{1,\dots,4\}$, for each of the  four copies of $U$ in $\widetilde{\Lambda}_{\KKK}$.\ 
We choose the embedding
\[
\Lambda_{\KKK^{[m]}} = U^{\oplus 3}\oplus E_8(-1)^{\oplus 2}\oplus I_1(-2p) \hookrightarrow U^{\oplus 4}\oplus E_8(-1)^{\oplus 2}=\widetilde{\Lambda}_{\KKK}
\]
given by mapping the generator $\ell$ of $I_1(-2p)$ to $u_1-pv_1$.\ We also set $\bv:=u_1+pv_1$ (so that $\bv^\perp=\Lambda_{\KKK^{[m]}}$) and $h:= 2(u_2+\frac{m}{4}v_2) + \ell$.

As explained in Remark~\ref{explicit} (and using Proposition \ref{prop:BHTnewprime}), the image of the period map is the complement of the Heegner divisors of the form $ \cD_{\tau,K}$, where $K$ is a primitive, rank-2, 
sublattice of $ \Lkkk[m]$ such that  {$h\in K$ and} there exist integers $0\leq k \leq p$ and $a\geq -1$, and $\kappa\in K$ with
\[
\kappa^2 = 2p( 4pa -k^2 )<0, \qquad  \kappa\cdot h =0, \qquad 
 2p\mid \div_{\Lkkk[m]}(\kappa)  . 
 \]
Since we are interested in the lattices $K^\bot$, this is equivalent, by strange duality, to looking at the lattices $K^\bot$ in $h^\bot=\Lkkk[2]$, where $K$ now contains the primitive class $\bv$, of square $2p$.\ This is a computation that was done during the proof of Proposition~\ref{propirr}(2)(c): in the notation of that proof, if we write $\kappa=b\kp$, with $\kp$ primitive, and $\disc(K^\bot)=:-2e$,
 \begin{itemize}
 \item either 
$(\div_{\Lkkk[m]}(\kp),\kp^2)=(1,-2e/p)$ and  $p\mid  e$; 
\item or
$(\div_{\Lkkk[m]}(\kp),\kp^2)=(p,-2pe)$ and $p\nmid e$.
 \end{itemize}
In the first case, we have $ 2p\mid \div_{\Lkkk[m]}(\kappa) =2b$ and, writing $b=2pb'$ and $e=pe'$, we obtain 
$$  2p( 4pa -k^2 )=\kappa^2 = 4p^2b^{\prime 2}\kp^2= 4p^2b^{\prime 2}\kp^2=-8peb^{\prime 2}=-8p^2e'b^{\prime 2}.
$$
This implies $2p\mid k$, and since $0\leq k \leq p$, we get $k=0$, $a=-1$, and $e'=b'=1$.\ This corresponds to the irreducible divisor ${^m}\cD_{2,2p}^{(2)}$.

In the second case, we write similarly $b=2b'$ and we obtain
$$  2p( 4pa -k^2 )=\kappa^2 = 4b^{\prime 2}\kp^2= -8peb^{\prime 2}.
$$
Writing $k=2k'$, we get $ k^{\prime 2} -pa= eb^{\prime 2}$, with $k'>0$.

{Assume now $p=3$.} The only possible pairs $(k,a)$ are then  $(0,-1)$, $(2,0)$, and $(2,-1)$.\ When $(k,a)=(0,-1)$, the discussion above shows that we obtain the divisor ${^4}\cD_{2,6}^{(2)}$.

When $(k,a)=(2,0)$, we obtain $e=1$, hence the divisor ${^4}\cD_{2,2}^{(2)}$.\ 
When $(k,a)=(2,-1)$, we obtain   $e\in\{1,4\}$, hence the extra divisor ${^4}\cD_{2,8}^{(2)}$.
\end{exam}


\begin{exer}
\footnotesize\narrower
 (a) Show that the image of the period map $
{}^8\!\cM_2^{(2)}\to{}^8\!\cP_2^{(2)}$ is the complement of ${^8}\cD_{2,2}^{(2)}\cup {^8}\cD_{2,4}^{(2)}\cup {^8}\cD_{2,8}^{(2)}\cup {^8}\cD_{2,14}^{(2)}\cup {^8}\cD_{2,16}^{(2)}\cup {^8}\cD_{2,18}^{(2)}\cup {^8}\cD_{2,22}^{(2)}\cup {^8}\cD_{2,32}^{(2)}$.


\noindent (b) Show that the image of the period map $
{}^{12}\!\cM_2^{(2)}\to{}^{12}\!\cP_2^{(2)}$ is the complement of 
${^{12}}\cD_{2,2}^{(2)}\cup 
{^{12}}\cD_{2,6}^{(2)}\cup 
{^{12}}\cD_{2,8}^{(2)}\cup 
{^{12}}\cD_{2,10}^{(2)}\cup
 {^{12}}\cD_{2,18}^{(2)}\cup
  {^{12}}\cD_{2,22}^{(2)}\cup
   {^{12}}\cD_{2,24}^{(2)}\cup 
   {^{12}}\cD_{2,28}^{(2)}\cup
    {^{12}}\cD_{2,30}^{(2)}\cup 
    {^{12}}\cD_{2,32}^{(2)}\cup
     {^{12}}\cD_{2,40}^{(2)}\cup
      {^{12}}\cD_{2,50}^{(2)}\cup 
      {^{12}}\cD_{2,54}^{(2)}\cup
       {^{12}}\cD_{2,72}^{(2)}$.


\noindent (c) Let $p$ be a prime number such that $p\equiv-1\pmod4$.\ Show that the complement of the image of the period map $
{}^{p+1}\!\cM_2^{(2)}\to{}^{p+1}\!\cP_2^{(2)}$ contains   ${^{p+1}}\cD_{2,2(l^2-pa)}^{(2)}$ for all $0\le l\le \frac{p-1}2$ and $-1\le a< l^2/p$, but none of the divisors ${^{p+1}}\cD_{2,2e}^{(2)}$ for $e> \frac{(p+1)^2}4$.

 \end{exer}

\begin{rema}\label{neerem}
The ample cone of a projective \hK\ fourfold of $K3^{[2]}$-type was already described in Theorem \ref{thm:NefConeHK4}.\ This description involved only classes $\bs\in\cS_X$ such that   $\bs^2=-2$ and $\bs\cdot \bv_X\in\{0,1\}$.\ We still find these classes in higher dimensions (\cite[Corollary~2.9]{mon}):
\begin{itemize}
\item    classes $\bs$ with  $\bs^2=-2$ and $\bs\cdot \bv_X=0$; in the notation of Proposition \ref{prop:BHTnew}, we have $a=-1$, $k=0$, and $\theta_X(\kappa_X)=2(m-1)\bs$, and the lattice $ \Z h\oplus \Z \frac{1}{2(m-1)}\kappa_X$ has intersection matrix $\begin{pmatrix} h^2 & 0 \\ 0 & -2\end{pmatrix}$
but may or may not be primitive;
 \item classes $\bs$ with  $\bs^2=-2$ and $\bs\cdot \bv_X=1$; in the notation of Proposition \ref{prop:BHTnew}, we have $a=-1$, $k=1$, and $\theta_X(\kappa_X)=2(m-1)\bs-\bv_X$; the lattice $ \Z h\oplus \Z \kappa_X$ has intersection matrix $\begin{pmatrix} h^2 & 0 \\ 0 & -2(m-1)(4m-3)\end{pmatrix}$ and may or may not be primitive.
\end{itemize}

When $m\in\{3,4\}$,   complete lists of possible pairs $(\kappa_{\textnormal{prim}}^2,\div(\kappa_{\textnormal{prim}}))$ are given in \cite[Sections~2.2 and~2.3]{mon}:
 \begin{itemize}
 \item when $m=3$,  we have
 $$(\kappa_{\textnormal{prim}}^2,\div(\kappa_{\textnormal{prim}}))\in\{(-2,1),(-4,2),(-4,4), (-12,2), (-36,4)\};$$
\item when $m=4$,  we have
$$(\kappa_{\textnormal{prim}}^2,\div(\kappa_{\textnormal{prim}}))\in\{(-2,1),(-6,2),(-6,3),(-6,6),(-14,2),(-24,3),(-78,6)\}.$$
\end{itemize}
Depending on $m$ and on the polarization, not all pairs in these lists occur: in Example~\ref{ex:LLSvS} (where $m=4$), only the pairs $(-6,3)$, $(-2,1)$, and $(-24,3)$ do occur.

Whenever $m-1$ is a power of a prime number, the pair $(\kappa_{\textnormal{prim}}^2,\div(\kappa_{\textnormal{prim}}))=(-2m-6,2)$ is also   in the list (\cite[Corollary~2.9]{mon}): in our notation, it corresponds to $\bs^2=2a=-2$  and $\bs\cdot \bv_X =k=m-1$.
\end{rema}

When the divisibility $\gamma  $ of the polarization $h$ is 1 or 2, we list some of the Heegner divisors that are avoided by  the period map.

\begin{prop} 
Consider the period map  \begin{equation*} 
  {}^m\wp^{(\gamma)}_{2n}\colon  {}^m\!\!\cM_{2n}^{(\gamma)}\lra  O(L_{\KKK^{[m]}},h_\tau)\backslash \Omega_{\tau}
 \end{equation*}
When $\gamma=1$, 
the image of  $\,{}^m\wp^{(1)}_{2n}$  does not meet 
\begin{itemize}
\item $\nu$ of the components of the hypersurface ${}^m\cD_{2n,2n(m-1)}^{(1)}$, where $\nu\in\{0,1,2\}$ is the number of the following congruences that hold: $n+m\equiv 2\pmod 4$, $m\equiv 2\pmod 4$, $n\equiv 1\pmod 4$\textnormal{;}
\item one of the components of the hypersurface ${}^m\cD_{2n,8n(m-1)}^{(1)}$\textnormal{;}
\item when $m-1$ is a prime power, one   component  of the hypersurface ${}^m\cD_{2n,2(m-1)(m+3)n}^{(1)}$.
\end{itemize}
  When $\gamma=2$ (so that $n+m\equiv 1\pmod4$), the image of $\, {}^m\wp^{(2)}_{2n}$ does not meet 
 \begin{itemize}
\item one of the components of the hypersurface ${}^m\cD_{2n,2n(m-1)}^{(2)}$\textnormal{;}
\item when $m-1$ is a power of $2$, one   component  of the hypersurface ${}^m\cD_{2n,(m-1)(m+3)n/2}^{(2)}$.
\end{itemize}
\end{prop}

\begin{proof}
No algebraic  class with square $-2$   can be orthogonal to $h_\tau$ (see Remark~\ref{neerem}).\ As in the proof of Theorem~\ref{imper}, there is such a class $\kappa$ if and only if the period point belongs to the hypersurface ${}^m\cD^{(\gamma)}_{2n,K} $, where $K$ is the rank-2 lattice with intersection matrix  $\bigl(\begin{smallmatrix}2n&0\\0&-2\end{smallmatrix}\bigr)$.\ This hypersurface is a component of   ${}^m\cD^{(\gamma)}_{2n,d} $, where $d:=|\disc(K^\bot)|$ can be computed using  the formula   \begin{equation*}
d=\left|\frac{\kappa^2\disc(h_\tau^\bot)}{s^2}\right|
\end{equation*}
from~\cite[Lemma~7.5]{ghssur} (see \eqref{eqd}), where $s:=\div_{h_\tau^\bot}(\kappa)\in\{1,2\}$.

Assume $\gamma=1$, so that $D(h_\tau^\bot)\isom  \Z/(2m-2)\Z\times \Z/2n\Z$ (see \eqref{eq1}).\ As in  the proof of Proposition~\ref{propirr}, we write
 $\kappa=a(u-nv)+b\ell +cw$, so that $1=-\frac12\kappa^2= na^2+(m-1)b^2+(\frac12 w^2)c^2$, and $s =\gcd(2na,2(m-1)b,c) $ is $1$ if $c$ is odd, and 2 otherwise.
 
If $s=1$, we have $d=8n(m-1)$ and $\kappa_*=0$, and, by Eichler's criterion, this defines an irreducible component of ${}^2\cD_{2n,8n(m-1)}^{(1)}$ (to prove that it is nonempty, just take $a=b=0$, $c=1$, and $w^2=-2$).

If $s=2$, we have $d=2n(m-1)$ and $\kappa_*$ has order 2.\ 
More precisely,
$$\kappa_*=\begin{cases} (0,n)&\textnormal{if $a$ is odd and $b$ is even;}\\
(m-1,0)&\textnormal{if $a$  is even and $b$ is odd;}\\
(m-1,n)&\textnormal{if $a$ and $b$ are odd.} 
\end{cases}$$
Conversely, to obtain such a class $\kappa$ with square $-2$, one needs to solve the equation $na^2+(m-1)b^2+rc^2=1$, with $c$   even and $r$  any integer.\ It is equivalent to solve $na^2+(m-1)b^2\equiv 1\pmod4$ and one checks that this gives the first item of the proposition.

Assume $\gamma=2$, so that $\disc(h_\tau^\bot)=n(m-1)$ (see \eqref{eq2}) and $d=\frac{2}{s^2}n(m-1)\in\{2n(m-1),\frac{n(m-1)}{2}\} $.\ As in  the proof of Proposition~\ref{propirr}, we may take
$h_\tau^\bot=\Z w_1\oplus \Z w_2\oplus M$, with    $w_1:= (m-1)v+\ell$ and $w_2:=-u+\frac{n+m-1}{4}v$ and write $\kappa=aw_1+bw_2+cw$.\ The equality $\kappa^2=-2$ now reads $1=a(a+b)(m-1)+b^2 \bigl(\frac{n+m-1}{4}\bigr)-c^2\bigl(\frac12 w^2\bigr)$, and again, $s=1$ if $c$ is odd, $s=2$ if $c$ is even.

Again, the case $s=1$ gives rise to a single component of ${}^2\cD_{2n,2n(m-1)}^{(1)}$, since it corresponds to $\kappa_*=0$ (take $a=b=0$ and $c=1$).\footnote{If $s=2$, we have $d=n(m-1)/2$, so that both $n$ and $m-1$ are even, and we need to solve
$a(a+b)(m-1)+b^2 \bigl(\frac{n+m-1}{4}\bigr)\equiv 1\pmod4$.\ The reader is welcome to determine the number of solutions of this equation and see  how many different elements of order 2 in the group $D(h_\tau^\bot)$ (which was determined in footnote \ref{foo11}) are obtained in this way in each case.\ This will produce more hypersurfaces avoided by the image of the period map.}

{\em Assume now that $m-1$ is a prime power.} In that case, no algebraic  class with square  $-2m-6$ and divisibility 2 can be orthogonal to $h_\tau$ (see Remark~\ref{neerem}).\ Assume that there is such a class  $\kappa$.\ 

When $\gamma=1$, keeping the same notation as above, we need to solve $na^2+(m-1)b^2+(\frac12 w^2)c^2=m+3$, with $a$ and $c$ even.\ Just taking $c=2$, $a=0$, and $b=1$ works, with $d=\frac{8(m+3)n(m-1) }{\gcd(2na,2(m-1)b,c)^2}=2(m-1)(m+3)n$.\ In the notation of Example~\ref{ex:LLSvS}, we may take $\bv=u_1+(m-1)v_1$, $\ell=u_1-(m-1)v_1$,  $h= u_2+nv_2 $,   $\kappa=(m-1)(2(u_3-v_3)+\ell)$, and $\bs=u_3-v_3+u_1$.

When $\gamma=2$, we need to solve $a(a+b)(m-1)+b^2 \bigl(\frac{n+m-1}{4}\bigr)-c^2\bigl(\frac12 w^2\bigr)=m+3$, with $a(m-1)$ and $c$ even, and $d=\frac{2(m+3)n(m-1) }{\gcd(2na,2(m-1)b,c)^2}\mid (m-1)(m+3)n/2$, so that both $n$ and $m-1$ are even.\ In that case, take $c=2$, $b=0$, and $a=1$.\ In the notation of Example~\ref{ex:LLSvS}, we may take $\bv=u_1+(m-1)v_1$, $\ell=u_1-(m-1)v_1$,  $h= 2(u_2+\frac{n+m-1}{4}v_2)+\ell $,   $\kappa=(m-1)((m-1)v_2+2(u_3-v_3)+\ell)$, and $\bs=\frac{m-1}{2}v_2+u_3-v_3+u_1$.
  \end{proof}

\subsection{The period map for cubic fourfolds}

Let $\cM_{\textnormal{cub}}$ be the moduli space of smooth   cubic hypersurfaces in $\P^5$.\  With any cubic fourfold $W$, we can associate its period, given by the weight-2 Hodge structure on $H^4_{\textnormal{prim}}(W,\Z)$.\  The Torelli Theorem for cubic fourfolds (\cite{vo}) says that the period map is an open embedding.\  By \cite[Proposition 6]{bedo}, there is a Hodge isometry $H^4_{\textnormal{prim}}(W,\Z)(-1) \isom H^2_{\textnormal{prim}}(F(W),\Z)$.\  Hence, we can equivalently consider the period map for (smooth) cubic fourfolds as the period map of their (smooth) Fano varieties of lines, which are hyperk\"ahler fourfolds of $\KKK^{[2]}$-type with a polarization of  square 6 and divisibility 2.\ The latter form a dense open subset in the moduli space ${}^2\!\cM_{6}^{(2)} $.

By \cite[Theorem 8.1]{dm},   the image  of the corresponding period map $  {}^2\!\cM_{6}^{(2)}  \to {}^2\cP_{6}^{(2)}$ is exactly the complement of the Heegner divisor ${}^2\cD_{6,6}^{(2)}$.\  The description of the image of the period map for cubic fourfolds, which is smaller, is part of a celebrated result of Laza and Looijenga (\cite{laza,looi}).

\begin{theo}[Laza, Looijenga]\label{lalo}
The image of the period map for cubic fourfolds is the complement of the divisor  ${}^2\cD_{6,2}^{(2)}\cup {}^2\cD_{6,6}^{(2)}$.
\end{theo}

It had been known for some time (see \cite{has}) that the image of the period map for cubic fourfolds {\em is contained in} the complement of ${}^2\cD_{6,2}^{(2)}\cup {}^2\cD_{6,6}^{(2)}$.

Let $W$ be a smooth cubic fourfold containing no planes (this is equivalent to saying that its period point is not on the Heegner divisor ${}^2\cD_{6,8}^{(2)}$).\ We explained in   Section~\ref{sec262} the construction (taken from  \cite{lls}) of a hyperk\"ahler eightfold $X(W)\in{}^4\!\cM_2^{(2)}$.\ It follows from \cite[Proposition~1.3]{lpz} that the varieties $F(W)$ and $X(W)$ are ``strange duals''  in the sense of Remark~\ref{stdu}.\footnote{\label{idea}This follows from a more general ``strange duality'' statement as follows.\ 
Consider the \emph{Kuznetsov component} $\cK\!u(W)$ of the derived category $\mathrm{D}^\mathrm{b}(W)$ (\cite{ku}).\ 
As explained in \cite[Section 2]{at},  one can associate with such a category a weight-2 Hodge structure on the lattice $\widetilde{\Lambda}_{\KKK}$; we denote it by $\widetilde{\Lambda}_{W}$.\ 
Moreover, there is a natural primitive sublattice $A_2\hra\widetilde{\Lambda}_{\textnormal{alg}, W}$; we denote its canonical basis by $(u_1,u_2)$: it satisfies $u_1^2=u_2^2=2$ and $u_1\cdot u_2=-1$.\ 

Let $\sigma_0=(\mathcal{A}_0,Z_0)$ be the Bridgeland stability condition constructed in \cite[Theorem~1.2]{blms}.\  
Given a Mukai vector $\bv\in\widetilde{\Lambda}_{\textnormal{alg}, W}$, we denote by $M_{\sigma_0}(\bv)$ the moduli spaces of $\sigma_0$-semistable objects in $\cA_0$ with Mukai vector $\bv$.\ 
If there are no properly $\sigma_0$-semistable objects,  $ M_{\sigma_0}(\bv)$ is, by \cite{blmnps}, a smooth projective hyperk\"ahler manifold of dimension $\bv^2+2$ and there is a natural Hodge isometry $H^2(M_{\sigma_0}(\bv),\Z)\isom \bv^\perp$ such that the embedding $\theta_{M_{\sigma_0}(\bv)}$ can be identified with $\theta_{\bv}\colon \bv^\perp \hra \widetilde{\Lambda}_{\KKK}$.\ 
Moreover, there is  by \cite{bama1} a natural ample class $\ell_{\sigma_0}(\bv)$ on $M_{\sigma_0}(\bv)$.\ 

By \cite{lpz}, there are isomorphisms $F(W)\isom M_{\sigma_0}(u_1)$ and $X(W)\isom M_{\sigma_0}(u_1+2u_2)$.\ Moreover, possibly after  multiplying by a positive constant, the class $\ell_{\sigma_0}(u_1)$ on $M_{\sigma_0}(u_1)$ (respectively, the class $\ell_{\sigma_0}(u_1+2U_2)$ on $M_{\sigma_0}(u_1+2u_2)$) corresponds   to the Pl\"ucker polarization on $F(W)$ (respectively, to the degree-$2$ polarization on $X(W)$).\
Finally, an easy computation shows   $\ell_{\sigma_0}(u_1)=\theta_{u_1}(u_1+2u_2)$ and $\ell_{\sigma_0}(u_1+2u_2)=\theta_{u_1+2u_2}(u_1)$.\ The strange duality statement follows directly from this: the periods of both polarized varieties are identified with the orthogonal of the $A_2$ sublattice.\ 
We notice that, in this example, this gives a precise formulation of the strange duality between polarized hyperk\"ahler manifolds in terms of Le Potier's strange duality (\cite{le}) between the two moduli spaces (on a noncommutative K3 surface).}

It follows from Theorem~\ref{lalo} that the image of the set of $X(W)$ by the period map $  {}^4\!\cM_{2}^{(2)}  \to {}^4\cP_{2}^{(2)}$ contains the complement of the Heegner divisors ${}^2\cD_{6,2}^{(2)}$, $ {}^2\cD_{6,6}^{(2)}$, and $ {}^2\cD_{6,8}^{(2)}$.\ Since the image of this period map is 
 exactly the complement of these Heegner divisors (Example~\ref{ex:LLSvS}), we obtain that any element of 
 $ {}^4\!\cM_{2}^{(2)} $ is of the form $X(W)$.\ We provide below a different proof (due to Bayer and Mongardi) of this statement and of a  weaker form of Theorem~\ref{lalo}.\  

\begin{prop}[Bayer, Mongardi]\label{prop:BayerMongardi}
{\rm (a)} Any element of $  {}^4\!\cM_{2}^{(2)}$ is of the type $X(W)$, for some smooth cubic fourfold $W\subset \P^5$ containing no planes.

\noindent {\rm (b)}
The image of the period map for cubic fourfolds contains the complement of ${}^2\cD_{6,2}^{(2)}\cup {}^2\cD_{6,6}^{(2)}\cup {}^2\cD_{6,8}^{(2)}$. 
\end{prop}

%

\begin{proof}[Sketch of proof of Proposition \ref{prop:BayerMongardi}]
By Proposition~\ref{prop27},   any eightfold $X$ in the moduli space $^{4}\!\cM_{2}^{(2)}$ has an  anti-symplectic regular involution.\  
Upon varying the   eightfold in the moduli space, the fixed loci form  a smooth family.\ 
In particular, since, when $X=X(W)$, the fixed locus contains a copy of the cubic fourfold $W$, and any smooth deformation of a cubic fourfold is a cubic fourfold as well (this is due to Fujita; see  \cite[Theorem 3.2.5]{iskpro}),  there exists for any $X\in {^{4}\!\cM_{2}^{(2)}}$ a smooth cubic fourfold $W_X\subset X$   fixed by the involution, and the periods of $W_X$ and $X$ are compatible because they are strange duals, as explained above.\ 

By Example~\ref{ex:LLSvS}, the image of  the period map for $^{4}\!\cM_{2}^{(2)}$ is exactly the complement of the union of the Heegner divisors ${^4}\cD_{2,2}^{(2)}$, ${^4}\cD_{2,6}^{(2)}$, and ${^4}\cD_{2,8}^{(2)}$, and
 these divisors correspond by   strange duality to the (irreducible) Heegner divisors ${}^2\cD_{6,2}^{(2)}$, ${}^2\cD_{6,6}^{(2)}$, and ${}^2\cD_{6,8}^{(2)}$ in the period domain of cubic fourfolds.\ 
\end{proof}


\begin{thebibliography}{BLMNPS}

 \bibitem[A]{add2} Addington, N., On two rationality conjectures for cubic fourfolds,   {\it Math. Res. Lett.}  {\bf23} (2016),  1--13. 
 
 
  \bibitem[AL]{adle} Addington, N.,  Lehn, M., On the symplectic eightfold associated to a Pfaffian cubic
fourfold,   {\it J. reine angew. Math.} {\bf731} (2017), 129--137.

   \bibitem[AT]{at} Addington N., Thomas, R., Hodge theory and derived categories of cubic fourfolds, {\it Duke Math. J.} {\bf 163} (2014), 1885--1927.


 \bibitem[AV1]{av} Amerik, E.,  Verbitsky, M., Morrison--Kawamata cone conjecture for hyperk\"ahler manifolds, {\it Ann. Sci. \'Ec. Norm. Sup\'er.} {\bf50} (2017),  973--993. 
 
 
  \bibitem[AV2]{amve} Amerik, E., Verbitsky, M., Teichm\"uller space for hyperk\"ahler and symplectic structures, {\it J. Geom. Phys.} {\bf 97} (2015), 44--50.
 
 \bibitem[Ap1]{apo}  Apostolov, A., 
Moduli spaces of polarized irreducible symplectic manifolds are not necessarily connected, {\it Ann. Inst. Fourier} {\bf64} (2014), 189--202. 

 \bibitem[Ap2]{apot}  Apostolov, A., 
On irreducible symplectic varieties of $\KKK^{[n]}$-type, Ph.D. thesis, Universit\"at Hannover, 2014.

 \bibitem[BHPvV]{bhpv}  Barth, W., Hulek, K., Peters, C., Van de Ven, A., {\em Compact complex surfaces,} Second edition, Ergebnisse der Mathematik und ihrer Grenzgebiete {\bf4}, Springer-Verlag, Berlin, 2004.

 \bibitem[BHT]{bht} Bayer, A., Hassett, B., Tschinkel, Y., Mori cones of holomorphic symplectic varieties of K3 type, {\it Ann. Sci. \'Ec. Norm. Sup\'er.}, {\bf 48} (2015), 941--950.

 \bibitem[BM1]{bama1}   Bayer, A.,   Macr\`i, E., Projectivity and birational geometry of Bridgeland moduli spaces, {\it J. Amer. Math. Soc.}  {\bf 27} (2014), 707--752.

 \bibitem[BM2]{bama}   \bysame, MMP for moduli of sheaves on K3s via wall-crossing:
nef and movable cones, Lagrangian fibrations, {\it Invent. Math.}  {\bf198} (2014),  505--590.


 \bibitem[BLMNPS]{blmnps} Bayer, A., Lahoz, M., Macr\`i, E., Nuer, H., Perry, A., Stellari, P., Stability conditions in family, in preparation 2017. 
 
 \bibitem[BLMS]{blms} Bayer, A., Lahoz, M., Macr\`i, E., Stellari, P., Stability conditions on Kuznetsov components, eprint  {\tt   arXiv:1703.10839}.

 \bibitem[B1]{bea}  Beauville, A.,  Vari\'et\'es K\"ahleriennes dont la premi\`ere classe de Chern est nulle, {\it J.~Differential  Geom.}  {\bf 18} (1983), 755--782.
 
 \bibitem[B2]{beac10} 
   \bysame,
 Some remarks on K\"ahler manifolds with $c_1= 0$. {\it Classification of algebraic and analytic  manifolds (Katata, 1982)},  1--26, Progr. Math., {\bf 39}, Birkha\"user Boston, Boston, MA, 1983. 

 \bibitem[BD]{bedo} Beauville, A., Donagi, R., La vari\'{e}t\'{e} des droites d'une hypersurface cubique de dimension 4, {\it C.\  R.\  Acad.\  Sci.\  Paris S\'er.\  I Math.}  {\bf 301}  (1985),  703--706.
 
  \bibitem[BCMS]{sbisom} Boissi\`ere, S., Camere, C., Mongardi, G., Sarti, A., 
Isometries of ideal lattices and hyperk\"ahler manifolds,  {\it Int. Math. Res. Not.} (2016),  963-–977.

 \bibitem[BCS]{sbcomp} Boissi\`ere, S., Camere, C., Sarti, A., 
Complex ball quotients from manifolds of K3$^{[n]}$-type, 
{\em J. Pure   Appl. Algebra} {\bf223} (2019),  1123--1138. 

 
  \bibitem[BCNS]{bcns}  Boissi\`ere, S., Cattaneo, A., Nieper-Wi\ss kirchen, M.,   Sarti, A., The automorphism group of the Hilbert scheme of two points on a generic projective K3 surface, 
  in {\em Proceedings of the Schiermonnikoog conference, K3 surfaces and their moduli,} 
Progress in Mathematics {\bf315}, Birkh\"auser, 2015. 

  \bibitem[BS]{bs}  Boissi\`ere, S.,  Sarti, A., A note on automorphisms and birational transformations  of holomorphic symplectic manifolds, 
{\it  Proc. Amer. Math. Soc.}  {\bf 140} (2012), 4053--4062.

 \bibitem[C]{can} Cantat, S.,
Dynamique des automorphismes des surfaces projectives complexes, {\em C. R. Acad. Sci. Paris S\'er. I Math.} {\bf328} (1999), 901--906.
 
 \bibitem[CG]{cago} Catanese, F.,  G\"ottsche, L., $d$-very-ample line bundles and embeddings of Hilbert schemes
of 0-cycles, {\it Manuscripta Math.} {\bf68} (1990),  337--341.

 \bibitem[Ca]{cat} Cattaneo, A.,  Automorphisms of  Hilbert schemes of  points on a generic projective K3 surface, eprint {\tt   arXiv:1801.05682. }
 
 \bibitem[CF]{cf}    Cattaneo, A., Fu, L., Finiteness of Klein actions and real structures on compact hyper-K\"ahler
manifolds, to appear in {\it
Math. Ann.}


 \bibitem[CU]{clul} Clozel, L., Ullmo, E., \'Equidistribution de sous-vari\'et\'es sp\'eciales, {\it Ann. of Math.} {\bf161}
(2005), 1571--1588.
 
  \bibitem[DIM]{dims} Debarre, O., Iliev, A., Manivel, L., Special prime Fano fourfolds of degree 10 and index 2, {\it  Recent Advances in Algebraic Geometry,} 123--155,  C.\  Hacon, M.\  Musta\c t\u a, and M.\  Popa editors, London Mathematical Society Lecture Notes Series {\bf417}, Cambridge University Press, 2014.

\bibitem[DK]{DK}  Debarre, O., Kuznetsov, A., Gushel--Mukai varieties: classification and birationalities,   {\it Algebr. Geom.} {\bf5} (2018), 15--76.

\bibitem[DM]{dm}  Debarre, O., Macr\`i, E.,  On the period map for polarized hyperk\"ahler fourfolds, to appear in {\em Int. Math. Res. Not. IMRN.,} eprint {\tt arXiv:1704.01439.}

\bibitem[DV]{devo}  Debarre, O., Voisin, C.,  Hyper-K\"ahler fourfolds and Grassmann geometry, {\it J. reine angew. Math.} {\bf649} (2010), 63--87.


\bibitem[FGvGvL]{vgee}  Festi, D.,  Garbagnati, A., van Geemen, B.,   van Luijk, R.,
The Cayley--Oguiso free automorphism of positive entropy on a K3 surface, {\it J. Mod. Dyn.} {\bf 7} (2013) 75--96. 



\bibitem[F]{fuj} Fujiki, A., On automorphism groups of compact K\"ahler manifolds, {\it Invent. Math.} {\bf 44} (1978),  225--258.

 \bibitem[GHS1]{ghs0} Gritsenko, V., Hulek, K., Sankaran, G.K., The Kodaira dimension of the moduli of K3 surfaces, {\it Invent. Math.} {\bf 169} (2007), 519--567.

 \bibitem[GHS2]{ghs} \bysame, Moduli spaces of irreducible symplectic manifolds, {\it   Compos. Math.}  {\bf 146} (2010),  404--434. 

 \bibitem[GHS3]{ghssur} \bysame,  Moduli of K3 surfaces and irreducible symplectic manifolds, {\it
Handbook of moduli.} Vol. I, 459--526, Adv. Lect. Math. (ALM) {\bf24}, Int. Press, Somerville, MA, 2013.

 
 \bibitem[Gu]{gua}  Guan, D.,
On the Betti numbers of irreducible compact hyperk\"ahler manifolds of complex dimension four, {\em Math. Res. Lett.} {\bf8}
(2001),  663--669.

\bibitem[H]{has}  Hassett,  B., Special cubic fourfolds, {\it Compos. Math.}  {\bf 120} (2000),   1--23.

 \bibitem[HT1]{hast2} Hassett,  B., Tschinkel, Y., Moving and ample cones of holomorphic symplectic fourfolds,
{\it Geom. Funct. Anal.}  {\bf 19} (2009),   1065--1080.

 \bibitem[HT2]{hast4} \bysame, Hodge theory and Lagrangian planes on generalized Kummer fourfolds, {\it 
Mosc. Math. J.} {\bf 13} (2013),   33–-56. 

 \bibitem[HT3]{hast3} \bysame, Flops on holomorphic symplectic fourfolds and determinantal cubic hypersurfaces, {\it J. Inst. Math. Jussieu}  {\bf9} (2010), 125--153. 

 \bibitem[Hu1]{huyk3} Huybrechts, D., {\it Lectures on K3 surfaces,} Cambridge Studies in Advanced Mathematics {\bf158},  Cambridge University Press, 2016.
 
  \bibitem[Hu2]{huyinv} \bysame,  Compact hyperk\"ahler manifolds: basic results, {\em Invent. Math.} {\bf135} (1999), 63--113.
  
    \bibitem[Hu3]{ghj}  Huybrechts, D., Compact Hyperk\"ahler Manifolds, in {\it Calabi-Yau manifolds and related geometries}, Lectures from the Summer School held in Nordfjordeid, June 2001, Universitext, Springer-Verlag, Berlin, 2003.
   
  \bibitem[IKKR1]{ikkr1} Iliev, A.,    Kapustka, G., Kapustka, M., Ranestad, K.,   EPW cubes, to appear in {\it J. reine angew. Math.}
   
   \bibitem[IKKR2]{ikkr} 
   \bysame,
     Hyper-K\"ahler fourfolds and Kummer surfaces, {\em Proc. Lond. Math. Soc.} {\bf
     115} (2017), 1276--1316.
   
  \bibitem[IR]{ir1} Iliev, A.,  Ranestad, K., K3 surfaces of genus 8 and varieties of sums of
powers of cubic fourfolds, {\it Trans. Am. Math. Soc.} {\bf 353} (2001), 1455--1468.

  \bibitem[IP]{iskpro} Iskovskikh, V., Prokhorov, Y., {\it Fano varieties. Algebraic geometry, V,} Springer, Berlin, 1999. 

\bibitem[K]{kaw} Kawamata, Y.,
On  Fujita's  freeness  conjecture  for 3-folds  and 4-folds, {\it Math. Ann.} {\bf308} (1997),  491--505.

\bibitem[Kn]{knu} Knutsen, A.L., On $k$th-order embeddings of K3 surfaces and Enriques surfaces, {\it Manuscripta
Math.} {\bf104} (2001), 211--237.

  \bibitem[Ku]{ku} Kuznetsov, A., Derived categories of cubic fourfolds, in {\it Cohomological and geometric approaches to rationality problems,} 219--243, Progr. Math. {\bf 282}, Birkh\"auser, Boston, 2010.


  \bibitem[J]{jam} James, D.G., On Witt's theorem for unimodular quadratic forms, {\it Pacific J. Math.}  {\bf 26} (1968), 303--316. 
 
 \bibitem[L]{laf} Laface, A., Mini-course on K3 surfaces, available at\newline {\tt http://halgebra.math.msu.su/Lie/2011-2012/corso-k3.pdf}
 
   \bibitem[La]{lai}
 Lai, K.-W., New cubic fourfolds with odd degree unirational parametrizations, {\it  Algebra Number Theory} {\bf11} (2017),   1597--1626. 

 \bibitem[Laz]{laza} Laza, R., The moduli space of cubic fourfolds via the period map, {\it Ann. of Math.} {\bf 172} (2010),  673--711.
 
  \bibitem[Le]{le} Le Potier, J.,  Dualit\'e \'etrange sur les surfaces, 2005. 
 
\bibitem[LLSvS]{lls} Lehn, C., Lehn, M., Sorger, C., van Straten, D.,
Twisted cubics on cubic fourfolds,   {\it J. reine angew. Math.} {\bf731} (2017), 87--128.

 
 \bibitem[LPZ]{lpz} Li, C., Pertusi, L., Zhao, X., Twisted cubics on cubic fourfolds and stability conditions, eprint {\tt  arXiv:1802.01134. }.
 
 \bibitem[LL]{lili} Li, J., Liedtke, C.,
Rational curves on K3 surfaces, {\it Invent. Math.} {\bf188} (2012), 713--727. 

  \bibitem[Li]{lie}  Liedtke, C., Lectures on supersingular K3 surfaces and the crystalline Torelli theorem 
  in {\em Proceedings of the Schiermonnikoog conference, K3 surfaces and their moduli,} 
Progress in Mathematics {\bf315}, Birkh\"auser, 2015. 


 \bibitem[Lo]{looi} Looijenga, E., The period map for cubic fourfolds, {\it Invent. Math.} {\bf 177} (2009), 213--233.


 \bibitem[MO]{maop} Marian, A., Oprea, D., Generic strange duality for  K3  surfaces, with an appendix by   Yoshioka, K.,
 {\it  Duke Math. J.}  {\bf 162} (2013), 1463--1501. 

 



 \bibitem[M1]{mar2}  Markman, E., Integral constraints on the monodromy group of the hyperK\"ahler resolution of a symmetric product of a K3 surface, {\it Internat. J. Math.}  {\bf 21} (2010), 169--223.

 \bibitem[M2]{marsur}  \bysame,  A survey of Torelli and monodromy results for holomorphic-symplectic varieties, in {\it Complex and differential geometry,} 257--322,
Springer Proc. Math. {\bf8}, Springer, Heidelberg, 2011. 

 \bibitem[M3]{mar3} \bysame, Prime exceptional divisors on holomorphic symplectic varieties and monodromy-reflections, {\it Kyoto J. Math.} {\bf 53} (2013), 345--403.
 
  \bibitem[MM]{mame}  Markman, E.,  Mehrotra, S., Hilbert schemes of K3 surfaces are dense in moduli, {\it
Math. Nachr.} {\bf 290} (2017),  876--884.
 
 \bibitem[MY]{mayo}  Markman, E.,  Yoshioka, K.,  A proof of the Kawamata--Morrison Cone Conjecture for holomorphic symplectic varieties of K3$^{[n]}$ or generalized Kummer deformation type,  {\it Int. Math. Res. Not.}  {\bf24} (2015), 13563--13574.
 
  \bibitem[Ma]{mat} Matsushita, D., On almost holomorphic Lagrangian fibrations, {\it Math. Ann.} {\bf 358} (2014), 565--572.


  \bibitem[Mo1]{mon}  Mongardi, G., Automorphisms of Hyperk\"ahler
manifolds,  Ph.D. thesis, Universit\`a   Roma Tre, 2013, eprint {\tt
 arXiv:1303.4670. }
 
  \bibitem[Mo2]{mon2}  
  \bysame,
  A note on the K\"ahler and Mori cones of hyperk\"ahler manifolds, {\it Asian J. Math.} {\bf 19} (2015), 583--591.
  
  \bibitem[Mu1]{muk0}   Mukai, S.,    Symplectic structure of the moduli space of sheaves on an abelian or K3 surface, {\em Invent. math.} {\bf77} (1984), 101--116.
 
  \bibitem[Mu2]{muk1}   
     \bysame, 
     Curves, K3 surfaces and Fano 3-folds of genus $\le10$, in {\em  Algebraic geometry and commutative algebra, Vol. I,} 357--377, Kinokuniya, Tokyo, 1988.
  
  \bibitem[Mu3]{muk2}  
   \bysame, 
   Biregular classification of Fano 3-folds and Fano manifolds of coindex 3, {\em Proc. Nat. Acad. Sci. U.S.A.} {\bf86} (1989),  3000--3002.
 
  \bibitem[Mu4]{muk3}   
    \bysame,
 Curves and K3 surfaces of genus eleven, in {\em  Moduli of vector bundles (Sanda, 1994; Kyoto, 1994),} 189--197, Lecture Notes in Pure and Appl. Math. {\bf179}, Dekker, New York, 1996.
  
   \bibitem[Mu5]{muk4}  
     \bysame,
  Polarized K3 surfaces of genus 18 and 20, in {\em Complex projective geometry (Trieste, 1989/Bergen, 1989),} 264--276, London Math. Soc. Lecture Note Ser. {\bf179}, Cambridge Univ. Press, Cambridge, 1992.
 
  \bibitem[Mu6]{muk5}    \bysame,
  Polarized K3 surfaces of genus thirteen, in {\em Moduli spaces and arithmetic geometry,} 315--326, Adv. Stud. Pure Math.  {\bf45}, Math. Soc. Japan, Tokyo, 2006.
 
  \bibitem[Mu7]{muk6}    \bysame,
 K3 surfaces of genus sixteen, in {\em Minimal models and extremal rays (Kyoto, 2011),} 379--396, Adv. Stud. Pure Math. {\bf70}, Math. Soc. Japan,  2016.
 
   
  \bibitem[N]{nage} Nagell, T., {\it Introduction to number theory,} Second edition, Chelsea Publishing Co., New York, 1964.

  \bibitem[Ni]{nik}      Nikulin, V.,     Integral symmetric bilinear forms and some of their geometric applications,   {\it   Izv. Akad. Nauk SSSR Ser. Mat.}  {\bf   43}  (1979), 111--177.\ English transl.: {\it Math.\ USSR Izv.} {\bf 14} (1980), 103--167.

 \bibitem[Nu]{nue} Nuer, H., Unirationality of moduli spaces of special cubic fourfolds and K3 surfaces, in {\it Rationality Problems in Algebraic Geometry, Levico Terme, Italy, 2015,} 161--167,   Lecture Notes in Math. {\bf 2172}, Springer International Publishing, 2016.
 
\bibitem[O1]{ognew1}	O'Grady,  K.,  Desingularized moduli spaces of sheaves on a K3, {\it J. reine angew.
Math.} {\bf 512} (1999), 49--117.    

\bibitem[O2]{ognew2}	
\bysame, A new six-dimensional irreducible symplectic variety, {\it  J. Algebraic
Geom.} {\bf 12} (2003), 435--505.


 
  \bibitem[O3]{og2}	\bysame,      Dual double EPW-sextics and their periods, {\it  Pure Appl.\ Math.\ Q.}  {\bf  4}  (2008),   427--468.

%
%
  \bibitem[O4]{og6}	\bysame,    Periods of double EPW-sextics,  {\it Math. Z.} {\bf 280} (2015),  485--524.
 
 
  \bibitem[O5]{og8}	\bysame, Involutions and linear systems on holomorphic 
 symplectic manifolds, {\it Geom. Funct. Anal.} {\bf 15} (2005),  1223--1274.

\bibitem[Og1]{ogu1} Oguiso K.,
Free Automorphisms of Positive Entropy on Smooth K\"ahler Surfaces, in {\em Algebraic geometry in east Asia, Taipei 2011,} 187--199,
Adv. Stud. Pure Math. {\bf65}, Math. Soc. Japan, Tokyo, 2015. 

 \bibitem[Og2]{ogu2}       \bysame, Automorphism groups of Calabi--Yau manifolds of Picard number two, {\em J. Algebraic
Geom.} {\bf 23} (2014), 775--795.

 \bibitem[Og3]{ogu3}  \bysame, K3 surfaces via almost-primes, {\it 
Math. Res. Lett.} {\bf9}   (2002),   47--63. 

 \bibitem[Og4]{ogu4}  \bysame,   Tits alternative in \hKm s, {\it Math. Res. Lett.} {\bf 13} (2006), 307--316.

 \bibitem[Og5]{ogu5}  \bysame,  Bimeromorphic automorphism groups of non-projective \hKm s---a
note inspired by C. T. McMullen, {\it J. Differential Geom.} {\bf 78} (2008), 163--191.

\bibitem[SD]{sd} Saint-Donat, B.,  Projective Models of K - 3 Surfaces, {\em Amer. J. Math.} {\bf96} (1974), 602--639.

\bibitem[S]{saw} Sawon, J.,  A bound on the second Betti number of hyperk\"ahler manifolds of complex dimension six,  eprint {\tt arXiv:1511.09105.}


\bibitem[Se]{ser} Serre, J.-P., {\em Cours d'arithm\'etique,} P.U.F.,
Paris, 1970.

  \bibitem[TV]{vas}
Tanimoto, S., V\'arilly-Alvarado, A., Kodaira dimension of moduli of special cubic fourfolds, to appear in {\it J. reine angew. Math.}

\bibitem[vD]{vdd} van den Dries, B., {\em Degenerations of cubic fourfolds
and
holomorphic symplectic geometry,} Ph.D. thesis, Universiteit Utrecht, 2012, available at {\tt https://dspace.library.uu.nl/bitstream/handle/1874/233790/vandendries.pdf}
 
 

 
 \bibitem[V]{ver} Verbitsky, M.,   Mapping class group and a global Torelli theorem for hyperk\"ahler manifolds.
Appendix A by Eyal Markman,
{\it  Duke Math. J.}  {\bf 162} (2013), 2929--2986. 

 \bibitem[Vi]{vie} Viehweg, E., Weak positivity and the stability of certain Hilbert points. III, {\it Invent. Math.} {\bf101} (1990), 521--543.


 \bibitem[Vo1]{voi}  Voisin, C., Remarks and questions on coisotropic subvarieties and 0-cycles of hyper-K\"ahler varieties, in {\it K3 surfaces and their moduli,} 365--399, Progr. Math. {\bf315}, Birkh\"auser,  2016.
 
   \bibitem[Vo2]{vo} 
     \bysame,
 Th\'eor\`eme de Torelli pour les cubiques de $\P^5$, {\it Invent. Math.} {\bf 86} (1986), 577--601, and Erratum: ``A Torelli theorem for cubics in $\P^5$,'' 
{\it Invent. Math.} {\bf 172} (2008),   455--458. 

 \end{thebibliography}
 \end{document}